\documentclass[11pt]{article}
\usepackage{amsmath}
\usepackage{amsthm}
\usepackage{amsfonts}
\usepackage{fancybox}
\usepackage{epsfig}
\usepackage{epsf}
\usepackage{amssymb}
\usepackage{epic,eepic}
\usepackage{latexsym,bm}
\usepackage{graphicx}
\usepackage{subfigure}
\usepackage{float}
\usepackage{color}
\usepackage[top=1in, bottom=1in, left=1in, right=1in, dvips,letterpaper]{geometry}
\usepackage{algorithm}
\usepackage{algorithmic}
\usepackage{cite}
\usepackage{boxedminipage}
\usepackage{graphicx}
\usepackage{epstopdf}
\usepackage{algorithmic}

\newtheorem{theorem}{Theorem}
\newtheorem{lemma}{Lemma}
\newtheorem{assumption}{Assumption}
\newtheorem{definition}{Definition}
\newtheorem{corollary}{Corollary}

\newcommand\crule[1][5cm]{%
  \par
  \nointerlineskip
  \centerline{\hbox to #1{\hrulefill}}%
  \nointerlineskip}

\numberwithin{equation}{section}
\numberwithin{algorithm}{section}

\newcommand{\be}{\begin{equation}}
\newcommand{\ee}{\end{equation}}
\newcommand{\bee}{\begin{equation*}}
\newcommand{\eee}{\end{equation*}}
\newcommand{\bea}{\begin{eqnarray}}
\newcommand{\eea}{\end{eqnarray}}
\newcommand{\beaa}{\begin{eqnarray*}}
\newcommand{\eeaa}{\end{eqnarray*}}

\title{Group Sparse Optimization for Inpainting of Random Fields on the Sphere
\footnote{This work was partly supported by Hong Kong Research Grant Council PolyU15300120 and  the Shenzhen Research Institute of Big Data Grant 2019ORF01002.}}

\author{Chao Li\footnote{Department of Applied Mathematics, The Hong Kong Polytechnic University({chaoo.li@connect.polyu.hk}).} \ and\ Xiaojun Chen \footnote{Department of Applied Mathematics, The Hong Kong Polytechnic University({xiaojun.chen@polyu.edu.hk}).}}

\date{}

\begin{document}

\maketitle

\begin{abstract}
{We propose a group sparse optimization model
for inpainting of a square-integrable isotropic random field on the unit sphere,
where the field is represented by spherical harmonics with  random complex coefficients. In the proposed optimization model, the variable is an  infinite-dimensional complex vector and the objective function is a real-valued  function defined by a hybrid of the $\ell_2$ norm and non-Liptchitz  $\ell_p (0<p<1)$ norm that preserves rotational invariance property and  group structure of the random complex coefficients.
 We show that the infinite-dimensional optimization problem  is equivalent to a convexly-constrained finite-dimensional optimization problem.
Moreover, we propose a smoothing penalty algorithm to solve the finite-dimensional problem via  unconstrained optimization problems.
We provide an approximation error bound of the inpainted random field defined  by a scaled KKT point of the constrained optimization problem
in the square-integrable space on the sphere with probability measure.
Finally, we conduct numerical experiments on band-limited random fields  on the sphere and images  from Earth topography data to show the promising performance of the smoothing penalty algorithm for inpainting of random fields on the sphere.}\\
{\bf{ Keywords}}:
{group sparse optimization; exact penalty; smoothing method; random field.}
\end{abstract}

\section{Introduction}
\label{sec:intro}
{Let $(\Omega,\mathcal{F}, P)$ be a probability space, and let $\mathfrak{B}(\mathbb{S}^{2})$ be the Borel sigma algebra on the unit sphere $\mathbb{S}^{2}:=\{{\rm{x}}\in\mathbb{R}^{3}:\|\rm{x}\|=1\}$, where $\|\cdot\|$ is the $\ell_2$ norm.
A real-valued random field on $\mathbb{S}^{2}$ is an $\mathcal{F}\otimes \mathfrak{B}(\mathbb{S}^{2})$-measurable function $T: \Omega\times\mathbb{S}^{2}\rightarrow\mathbb{R}$, and it is said to be $2$-weakly isotropic if the expected value and covariance of $T$ are rotationally invariant} (see \cite{q2018}).
It is known that a $2$-weakly isotropic random field on the sphere has the following Karhunen-Lo\`{e}ve (K-L) expansion \cite{lang,M2011},
\begin{equation}\label{K-L}
T(\omega,{\rm{x}})=\sum_{l=0}^{\infty}\sum_{m=-l}^{l}\alpha_{l,m}(\omega)Y_{l,m}({\rm{x}}), \quad {\rm{x}}\in\mathbb{S}^{2},\quad \omega\in\Omega,
\end{equation}
where  $\alpha_{l,m}(\omega)=\int_{\mathbb{S}^{2}}T(\omega,{\rm{x}})\overline{Y_{l,m}({\rm{x}})}d\sigma({\rm{x}})\in\mathbb{C}$
are random spherical harmonic coefficients of $T(\omega,{\rm{x}})$, $Y_{l,m}$, for $m=-l,\ldots,l$, $l=0,1,2,\ldots$ are spherical harmonics with degree  $l$ and order $m$, and $\sigma$ is the surface measure on  $\mathbb{S}^{2}$ satisfying  $\sigma(\mathbb{S}^2)=4\pi.$
For convenient, let
$$\alpha(\omega)=(\alpha_{0,0}(\omega), \underbrace{\alpha_{1,-1}(\omega),\alpha_{1,0}(\omega),\alpha_{1,1}(\omega)}_{3},\ldots,\underbrace{\alpha_{1,-l}(\omega),\ldots,\alpha_{l,l}(\omega)}_{2l+1},\ldots)^T
$$
denote the coefficient vector of an isotropic random field $T(\omega,\mathrm{x})$.
The infinite-dimensional coefficient vector $\alpha(\omega)$
can be grouped according to degrees $l$ and written in the following group structure
\begin{equation}\label{gstru}
  \alpha(\omega)=(\alpha_{0\cdot}^T(\omega),\alpha_{1\cdot}^T(\omega),\ldots,\alpha^T_{l\cdot}(\omega),\ldots)^T,
\end{equation}
where $$\alpha_{l\cdot}(\omega)=(\alpha_{l,-l}(\omega),\ldots,\alpha_{l,0}(\omega),\ldots,\alpha_{l,l}(\omega))^T\in\mathbb{C}^{2l+1}, \quad \quad l\ge0.$$
From \cite{q2018}, we know that for any given degree $l\ge1$ and any $\omega\in \Omega$, the sum
$$\sum_{m=-l}^{l}\vert\alpha_{l,m}(\omega)\vert^2=\|\alpha_{l\cdot}(\omega)\|^2$$
is rotationally invariant, while $\sum_{m=-l}^{l}\vert\alpha_{l,m}(\omega)\vert$ is not invariant under rotation of the coordinate axes. Moreover,   in \cite{q2018}
 the angular power spectrum is given by
$$C_l=\frac{1}{2l+1}\mathbb{E}[\|\alpha_{l\cdot}(\omega)\|^2] $$
if the expected value $\mathbb{E}[T(\omega,{\rm{x}})] =0$ for all ${\rm{x}}\in \mathbb{S}^{2}$.  In the study of isotropic random fields \cite{Creasey,lang,M2011},
the angular power spectrum plays an important role since it contains full information of the covariance of the field and provides characterization of the field.
Hence, considering the group structure (\ref{gstru}) of the coefficients of an isotropic random field is essential.

Sparse representation of a random field $T$ is an approximation of $T$ with few non-zero elements $\alpha_{l,m}(\omega)$ of $\alpha(\omega)$. In group sparse representation,   instead of considering elements individually, we seek an approximation of $T$ with few non-zero groups $\alpha_{l\cdot}(\omega)$ of $\alpha(\omega)$, i.e., few non-zero entries of the following vector
 $$(\|\alpha_{0\cdot}(\omega)\|, \, \|\alpha_{1\cdot}(\omega)\|, \ldots, \|\alpha_{l\cdot}(\omega)\|,\ldots,)^T.$$
In recent years, sparse representation of isotropic random fields has been extensively studied.   In \cite{c2015} the authors studied the sparse representation of random fields via an $\ell_{1}$-regularized minimization problem.
In \cite{q2018}, the authors considered isotropic group sparse representation of Gaussian and isotropic random fields on the sphere through an unconstrained optimization problem with a weighted $\ell_{2,1}$ norm,  which is an example of   group Lasso \cite{yuan2006}.
In \cite{chen20},  the authors proposed a non-Lipschitz regularization problem  with a weighted $\ell_{2,p}$ ($0<p<1$) norm  for the group sparse representation of isotropic random fields.
The non-Lipschitz regularizer also preserves the rotational invariance property of $\|\alpha_{l\cdot}(\omega)\|^2$ for any given degree $l\ge1$ and any $\omega\in\Omega$.

Isotropic random fields on the sphere have many applications
\cite{Cabella,Jeong,Oh,Porcu,Stein},
especially in the study of the Cosmic Microwave Background (CMB) analysis
\cite{A2007,S2013}.
However,
the true spherical random field usually presents masked regions and missing observations.
Many inpainting methods \cite{A2007,Feeney,S2013,JD} have been proposed for recovering the true field based on sparse representation.
However, they did not take  the group structure (\ref{gstru}) of the coefficients into consideration.

Group sparse optimization  models have been successfully used in signal processing, imaging sciences and predictive analysis  \cite{beck2019,chen2021,huang2010,huang2009,pan2021}.
In this paper,
 we propose a constrained group sparse optimization model for inpainting of isotropic random fields on the sphere based on group structure (\ref{gstru}).
 Moreover,  to recover the true random field by using information only on a subset $\Gamma\subseteq\mathbb{S}^2$ which contains an open set,
 we need the unique continuation property \cite{isakov2006} of any realization of a
 random field,
that is, for any fixed $\omega\in\Omega$, if the value of a field equals to zero on $\Gamma$, then the random field is identically zero on the sphere (see Appendix A for more details).

Let  $L_{2}(\Omega\times\mathbb{S}^{2})$ be the $L_{2}$ space on the product space $\Omega\times\mathbb{S}^{2}$ with product measure $P\otimes\sigma$ and  the induced norm $\|\cdot\|_{L_{2}(\Omega\times\mathbb{S}^{2})}=\mathbb{E}[\|\cdot\|_{L_{2}(\mathbb{S}^{2})}]$,
where
$L_{2}(\mathbb{S}^{2})$ is the space of square-integrable functions over $\mathbb{S}^{2}$ endowed with the inner product
$$\langle f, g \rangle_{L_{2}(\mathbb{S}^{2})}=\int_{\mathbb{S}^{2}}f({\rm{x}})\overline{g({\rm{x}})}d\sigma({\rm{x}}), \quad \quad \forall f, g\in L_{2}(\mathbb{S}^{2}),$$
and the induced $L_{2}$-norm $\|f\|_{L_{2}(\mathbb{S}^{2})}=(\langle f, f \rangle_{L_{2}(\mathbb{S}^{2})})^{\frac{1}{2}}$.
By Fubini's theorem, for $\omega\in \Omega$, $T(\omega,{\rm{x}})\in L_{2}(\mathbb{S}^{2})$, $P$-a.s., where
 ``$P$-a.s." stands for almost surely with probability $1$.
For brevity, we write
$T(\omega,{\rm{x}})$ as $T({\rm{x}})$ or $T$ if no confusion arises.

Let the observed field be given by
\begin{equation}\label{t2}
 T^{\circ}:=\mathcal{A}(T^*)+\Delta,
 \end{equation}
where $T^*\in L_{2}(\Omega \times \mathbb{S}^{2})$ is the true isotropic random field that we aim to recover, $\Delta:\Omega\times\mathbb{S}^{2}\rightarrow\mathbb{R}$ is observational noise and
$\mathcal{A}: L_{2}(\Omega\times \mathbb{S}^{2})\rightarrow L_{2}(\Omega \times \mathbb{S}^{2})$ is  an inpainting operator defined by
\begin{equation}
\mathcal{A}(T({\rm{x}}))=\left\{\begin{array}{cl}
                           T({\rm{x}}) & \, {\rm if} \ {\rm{x}}\in  \Gamma\\
                        0 & \, {\rm if} \ {\rm{x}}\in \mathbb{S}^{2}\setminus\Gamma,\\
                        \end{array}
\right.
\end{equation}
 where $\mathbb{S}^{2}\setminus\Gamma \subset\mathbb{S}^{2}$ is a nonempty inpainting area and the set $ \Gamma\subset\mathbb{S}^{2}$  has an open subset.
 In our optimization model, we consider the observed field $T^{\circ}$ as one realization of a random field.
  For notational simplicity,  let
  \begin{equation}\label{gstru:o}
    \alpha=(\alpha^{T}_{0\cdot},\alpha^{T}_{1\cdot},\ldots,\alpha^{T}_{l\cdot},\ldots)^{T}
  \end{equation}
  denote the spherical harmonic coefficient vector of an isotropic random field $T$ for a fixed $\omega\in\Omega$,
 where $\alpha_{l\cdot}=(\alpha_{l,-l},\ldots,\alpha_{l,0},\ldots,\alpha_{l,l})^T\in\mathbb{C}^{2l+1}$, $l\in \mathbb{N}_{0}:=\{0,1,2,\ldots\}$.

By Parseval's theorem, we have
$$\|T\|^{2}_{L_{2}(\mathbb{S}^{2})}=\sum\limits_{l=0}^{\infty}\sum\limits_{m=-l}^{l}\vert\alpha_{l,m}\vert^{2}=\sum\limits_{l=0}^{\infty}\|\alpha_{l\cdot}\|^{2}<\infty.$$
{Hence the sequence $\{\|\alpha_{l\cdot}\|\}_{l\in\mathbb{N}_{0}}\in \ell^{2}(\mathbb{N}_{0})$, where $\ell^{2}(\mathbb{N}_{0})$ is the space of square-summable sequences.}
We write $\alpha\in\ell^{2}$ to indicate $\{\|\alpha_{l\cdot}\|\}_{l\in\mathbb{N}_{0}}\in \ell^{2}(\mathbb{N}_{0})$.

{The number of non-zero groups of $\alpha$ with group structure (\ref{gstru:o}) is calculated by
 \begin{equation*}\label{l2pnorm}
  \|\alpha\|_{2,0}=\sum\limits_{l=0}^{\infty}\|\alpha_{l\cdot}\|^{0},
\end{equation*}
where
$\|\alpha_{l\cdot}\|^{0}=\left\{\begin{array}{cc}
                                      1 & \textrm{if}\,\, \|\alpha_{l\cdot}\|>0\\
                                      0 &\textrm{otherwise,}
                                    \end{array}
                                    \right.
$
and $\|\alpha\|_{2,0}$ is called $\ell_{2,0}$ norm of  $\alpha$.  The $\ell_{2,0}$ norm is discontinuous and $\|T\|^{2}_{L_{2}(\mathbb{S}^{2})}<\infty$ does not ensure $\|\alpha\|_{2,0}<\infty$.    In \cite{q2018}, the authors used a weighted $\ell_{2,1}$ norm of  $\alpha$ defined by
\begin{equation*}\label{l2pnorm}
  \|\alpha\|_{2,1}=\sum\limits_{l=0}^{\infty}\beta_l\|\alpha_{l\cdot}\|,
\end{equation*}
where $\beta_l>0$.  It is easy to see that $\lim_{p\downarrow0} \|\alpha_{l\cdot}\|^{p}=\|\alpha_{l\cdot}\|^{0}$.   The $\ell_{2,p}$ norm $(0<p<1)$ has been widely used in sparse approximation \cite{chen2021,huang2010}. In this paper, we use a weighted $\ell_{2,p}$ norm of  $\alpha$.
Let $$\ell^p_{\beta}:=\left\{\alpha\in\ell^2:\|\alpha\|_{2,p}^{p}:=\sum_{l=0}^{\infty}\beta_{l}{{\|\alpha_{l\cdot}\|^{p}}}
  <\infty\right\}$$ be a weighted $\ell^p$ ($0<p<1$) space with positive weights    $\beta_0=1$ and $\beta_{l}=\eta^{l}l^{p}$  for $l\geq 1$, where $\eta>1$ is a constant.}
For a fixed $\omega \in \Omega$, we consider the following constrained optimization problem
\begin{equation}\label{inin:model}
  \begin{array}{cl}
    \min\limits_{\alpha\in\ell^p_{\beta}} &  \|\alpha\|_{2,p}^{p} \\
    {\rm{s.t.}} &  \|\mathcal{A}(T({\rm{x}}))-T^{\circ}({\rm{x}})\|^{2}_{L_{2}(\mathbb{S}^{2})}\leq \varrho,
  \end{array}
\end{equation}
 where  $\varrho>\int_{\mathbb{S}^{2}\setminus\Gamma}\vert T^\circ({\rm{x}})\vert^2d\sigma(\rm{x})$.

 {A feasible field of problem (\ref{inin:model}) takes the form
 $$T({\rm{x}})=\sum_{l=0}^{\infty}\sum_{m=-l}^{l}\alpha_{l,m}Y_{l,m}({\rm{x}}), \quad {\rm{x}}\in\mathbb{S}^{2}, \quad \alpha \in  \ell_\beta^p.$$
Since $T^\circ\in L_{2}(\mathbb{S}^2)$, for   $\epsilon=\varrho-\int_{\mathbb{S}^{2}\setminus\Gamma}\vert T^\circ({\rm{x}})\vert^2d\sigma(\rm{x})$,
there is a finite number $L$ such that $\|T^\circ(\mathrm{x})-T^{\circ}_L(\mathrm{x})\|^2_{L_{2}(\mathbb{S}^2)}<\epsilon$, where $T^{\circ}_L(\mathrm{x})=\sum_{l=0}^L\sum_{m=-l}^lb_{l,m} Y_{l,m}(\mathrm{x})$ and $b_{l,m}=\int_{\mathbb{S}^2} T^{\circ}(\mathrm{x}) \overline{Y_{l,m}(\mathrm{x})}d\sigma(\mathrm{x})$. Let $b_{l,m}=0$ for $l=L+1,\ldots$, $m=-l,\ldots,l$, then $(b_{0,0},\ldots,b_{L,L},0,0,0,\ldots)^T\in\ell_{\beta}^p$.
From the definition of $\mathcal{A}$, we have
\begin{eqnarray*}
  \|\mathcal{A}(T_L^{\circ}({\rm{x}}))-T^{\circ}({\rm{x}})\|^{2}_{L_{2}(\mathbb{S}^{2})} &=& \int_{\mathbb{S}^2\setminus\Gamma}\vert T^{\circ}({\rm{x}})\vert^2d\sigma(\mathrm{x})+\int_{\Gamma}\vert T^{\circ}({\rm{x}})-T_L^{\circ}({\rm{x}})\vert^2d\sigma(\rm{x}) \\
   &<&  \int_{\mathbb{S}^2\setminus\Gamma}\vert T^{\circ}({\rm{x}})\vert^2d\sigma(\mathrm{x})+\epsilon=\varrho.
\end{eqnarray*}
 Thus, the feasible set of problem (\ref{inin:model}) is nonempty.
 Moreover  the objective function of problem (\ref{inin:model}) is continuous and level-bounded. Hence an optimal solution of problem (\ref{inin:model}) exists.}

Since $\alpha=0$ implies that $\|\alpha\|_{2,p}^p=0$,  we assume $\|T^{\circ}({\rm{x}})\|^{2}_{L_{2}(\mathbb{S}^{2})}> \varrho$, which implies  that $T({\rm{x}}) =0$ is not a feasible field of problem (\ref{inin:model}).

Our main  contributions are summarized as follows.
 \begin{itemize}
   \item Based on the K-L expansion, the constraint of problem (\ref{inin:model})  can be written in a discrete form with variables $\alpha\in\ell^p_{\beta}$.
   We derive a lower bound for the $\ell_{2}$ norm of  nonzero groups of scaled KKT points of
   (\ref{inin:model}).
   Based on the lower bound, we prove that the infinite-dimensional problem (\ref{inin:model}) is equivalent to a finite-dimensional optimization problem.
   \item We propose a penalty method for solving the finite-dimensional optimization problem via  unconstrained optimization problems.
   We establish  the exact penalization results regarding  local minimizers and $\varepsilon$-minimizers.
   Moreover, we propose a smoothing penalty algorithm and prove that the sequence generated by the algorithm is bounded and any accumulation point of the sequence is a scaled KKT point of the finite-dimensional optimization problem.
      \item We give the approximation error of the random field  represented by scaled KKT points of the  finite-dimensional optimization problem in $L_2(\Omega\times \mathbb{S}^{2}).$
 \end{itemize}

 The rest of this paper is organised as follows.
 In Section 2, 
we prove that the infinite-dimensional discrete optimization problem  is equivalent to a finite-dimensional problem.
 In Section 3, we present the penalty method and give exact penalization results.
 In Section 4, we discuss optimality conditions of the finite-dimensional optimization problem and its penalty problem.
 Moreover, we propose a smoothing penalty algorithm  and establish its convergence.
 In Section 5, we give the approximation error in $L_2(\Omega\times \mathbb{S}^{2}).$
 In Section 6, we conduct numerical experiments on band-limited random fields and images from Earth topography data to compare our approach with some existing methods on the quality of the solutions and inpainted images.  Finally, we give conclusion remarks in Section 7.

\section{Discrete formulation  of problem (\ref{inin:model})}\label{subsec:models}
\label{sec:models}
In this section, we propose the discrete formulation of problem (\ref{inin:model}) and prove that the discrete problem is equivalent to a finite-dimensional problem (\ref{fini:cmodel}).

 Based on the definition of $\mathcal{A}$ and the spherical harmonic expansion of $T$, we obtain that
\begin{eqnarray}\label{constraint}
\begin{split}
&\|\mathcal{A}(T({\rm{x}}))-T^{\circ}({\rm{x}})\|^{2}_{L_{2}(\mathbb{S}^{2})}\\
 &= \int_{\mathbb{S}^{2}}\vert\mathcal{A}(T({\rm{x}}))-T^{\circ}({\rm{x}})\vert^{2}d\sigma({\rm{x}})
 =\int_{\Gamma}\vert T({\rm{x}})-T^{\circ}({\rm{x}})\vert^{2}d\sigma({\rm{x}})
 +\int_{\mathbb{S}^{2}\setminus\Gamma}\vert T^{\circ}({\rm{x}})\vert^{2}d\sigma({\rm{x}})\\
 &=\int_{\Gamma}\bigg\vert\sum\limits_{l=0}^{\infty}\sum\limits_{m=-l}^{l}\alpha_{l,m}Y_{l,m}({\rm{x}})\bigg\vert^{2} d\sigma({\rm{x}})
-2\textrm{Re}\left(\int_{\Gamma}T^{\circ}({\rm{x}})\left(\sum\limits_{l=0}^{\infty}\sum\limits_{m=-l}^{l}\bar{\alpha}_{l,m}\overline{Y_{l,m}({\rm{x}})}\right) d\sigma({\rm{x}})\right)\\
  & \quad+\int_{\mathbb{S}^{2}}\vert T^{\circ}({\rm{x}})\vert^{2}d\sigma({\rm{x}})\\
 &=\alpha^{T}Y\bar{\alpha}-2\textrm{Re}\left(\sum\limits_{l=0}^{\infty}\sum\limits_{m=-l}^{l}\bar{\alpha}_{l,m}\int_{\Gamma}T^{\circ}({\rm{x}})\overline{Y_{l,m}({\rm{x}})} d\sigma({\rm{x}})\right)+\int_{\mathbb{S}^{2}}\vert T^{\circ}({\rm{x}})\vert^{2}d\sigma({\rm{x}})\\
 &=\alpha^{H}Y\alpha-2\textrm{Re}(\alpha^{H}\alpha^{\circ})+c,
 \end{split}
\end{eqnarray}
where $Y$ is an infinite-dimensional  matrix with
$(Y)_{l^2+l+m+1,l'^2+l'+m'+1}=\int_{\Gamma}Y_{l,m}({\rm{x}})\overline{Y_{l',m'}({\rm{x}})} d\sigma({\rm{x}})\in\mathbb{C},$ $\alpha^{\circ}=((\alpha_{0\cdot}^{\circ})^T,\ldots,(\alpha^{\circ}_{l\cdot})^T,\ldots)^T$ with $\alpha^{\circ}_{l\cdot}\in\mathbb{C}^{2l+1}$, $ \alpha^{\circ}_{l,m}=\int_{\Gamma}T^{\circ}({\rm{x}})\overline{Y_{l,m}({\rm{x}})} d\sigma({\rm{x}})$ and $  c=\int_{\mathbb{S}^{2}}\vert T^{\circ}({\rm{x}})\vert^{2}d\sigma({\rm{x}}).$

The minimization problem (\ref{inin:model}) can be written as the following optimization problem
\begin{equation}\label{model2}
  \begin{array}{cl}
    \min\limits_{\alpha\in\ell_{\beta}^{p}} &  \|\alpha\|_{2,p}^{p} \vspace{1mm}\\
    {\rm{s.t.}} &  \alpha^{H}Y\alpha-2\textrm{Re}(\alpha^{H}\alpha^{\circ})+c\leq \varrho.
  \end{array}
\end{equation}
From the setting of $\varrho$ in problem (\ref{inin:model}) and Theorem \ref{uniq}, the feasible set of
(\ref{model2}) has an interior point  and bounded.
The following assumption follows from our assumption on problem (\ref{inin:model}).
\begin{assumption}\label{assu:inpa} The feasible set of problem (\ref{model2}) does not contain $\alpha=0$.
\end{assumption}

Note that the objective function and constraint function in (\ref{model2}) are real-valued functions with complex variables. Thus, following \cite{chen20}, we apply the Wirtinger calculus (See Appendix B for more details) in this paper.

Since the objective function in problem (\ref{model2}) is not Lipschitz continuous at points containing zero groups.
We shall consider first-order stationary points of (\ref{model2})
and  extend the definition of scaled  KKT points for finite-dimensional optimization problem with real variables in \cite{chen2016,CXY,rock}.
\begin{definition}\label{def:kkt}
 We call $\alpha^*\in\ell_{\beta}^{p}$  a scaled KKT point of (\ref{model2}), if there exists $\nu\in\mathbb{R}$ such that
\begin{equation}\label{def:infi:kkt}
\begin{split}
p\beta_{l}\|\alpha^*_{l\cdot}\|^{p}\alpha^*_{l\cdot}+2\nu\|\alpha^*_{l\cdot}\|^{2}(Y_{l\cdot}\alpha^*-\alpha_{l\cdot}^\circ) = 0, & \quad \forall\,l\in\mathbb{N}_{0},\\
 \nu((\alpha^*)^{H}Y\alpha^*-2{\rm{Re}}((\alpha^*)^{H}\alpha^{\circ})+c- \varrho)=0,&\\
 (\alpha^*)^{H}Y\alpha^*-2{\rm{Re}}((\alpha^*)^{H}\alpha^{\circ})+c\leq \varrho, &\\
 \nu\geq0.&
\end{split}
\end{equation}
\end{definition}
From (\ref{def:infi:kkt}),  for any nonzero vector $\alpha_{l\cdot}^* \in \mathbb{C}^{2l+1}$, we have
\begin{equation*}
 p\beta_{l}\|\alpha^*_{l\cdot}\|^{p-2}\alpha^*_{l\cdot}+2\nu(Y_{l\cdot}\alpha^*-\alpha_{l\cdot}^\circ)=0,
 \end{equation*}
 and $\nu>0$.
Then,
$$
p\beta_{l}\|\alpha^*_{l\cdot}\|^{p-1} = 2\nu\|Y_{l\cdot}\alpha^*-\alpha_{l\cdot}^\circ\|\leq2\nu\|Y\alpha^*-\alpha^\circ\|.
$$
 By definition of $Y$ and feasibility of $\alpha^*$, there exists $\tilde{c}>0$ such that $\|Y\alpha^*-\alpha^\circ\|\leq \tilde{c}$.
 Hence for any nonzero vector $\alpha_{l\cdot}^*\in\mathbb{C}^{2l+1}$, we have
\begin{equation}\label{lowerbound:kkt2}
\|\alpha^*_{l\cdot}\|\geq\left(\frac{p\beta_{l}}{2\nu\tilde{c}}\right)^{\frac{1}{1-p}}.
\end{equation}
By definition of $\beta_{l}$, $l\in\mathbb{N}_0$,  we obtain that
 \begin{eqnarray*}
   \infty>\|\alpha^*\|_{2,p}^{p} &=&\sum\limits_{l=0}^{\infty}\beta_{l}\|\alpha^*_{l\cdot}\|^{p}\, =\sum\limits_{\{l\in\mathbb{N}_{0}:\,\|\alpha^*_{l\cdot}\|\neq0\}}\beta_{l}\|\alpha^*_{l\cdot}\|^{p}
    \geq\left(\frac{p}{2\nu\tilde{c}}\right)^{\frac{p}{1-p}}\sum\limits_{\{l\in\mathbb{N}_{0}:\,\|\alpha^*_{l\cdot}\|\neq0\}}(\eta^ll^p)^{\frac{1}{1-p}},
 \end{eqnarray*}
which implies that
$\{l\in\mathbb{N}_{0}:\|\alpha^*_{l\cdot}\|\neq0\}$ is a finite set.
Thus, the number of nonzero vectors $\alpha_{l\cdot}^* \in \mathbb{C}^{2l+1}$ of a scaled KKT point of problem (\ref{model2}) is finite and there exists an $L\in\mathbb{N}_{0}$ such that
$L=\max\{l\in\mathbb{N}_{0}:\|\alpha^*_{l\cdot}\|\neq0\}$.
Therefore, the infinite-dimensional problem (\ref{model2}) can be truncated to a finite-dimensional problem and approximated in a finite-dimensional space.

For notational simplicity, we truncate $\alpha\in\ell_{\beta}^{p}$ to $\alpha=(\alpha_{0,0},\alpha_{1\cdot}^T,\ldots,\alpha_{L\cdot}^T)^T\in\mathbb{C}^{d}$, where $d:=(L+1)^{2}$.
We use $\hat{\alpha}^\circ\in\mathbb{C}^{d}$ to denote the truncated vector whose elements are the first $d$ elements of $\alpha^\circ$ and $\hat{Y}\in\mathbb{C}^{d\times d}$
to denote the leading principal submatrix of order $d$ of $Y$.
Since $ \Gamma$ has an open subset, we have $z^H\hat{Y}z>0$ for any nonzero $z\in\mathbb{C}^{d}$.
Thus, the matrix $\hat{Y}$ is positive definite.

The truncated finite-dimensional problem of problem (\ref{model2}) has the following version
\begin{equation}\label{fini:cmodel}
  \begin{array}{cl}
    \min\limits_{\alpha\in\mathbb{C}^{d}} & \Phi(\alpha):=\sum_{l=0}^{L}\beta_{l}\|\alpha_{l\cdot}\|^p \vspace{1mm}\\
    {\rm{s.t.}} &  \alpha^{H}\hat{Y}\alpha-2\textrm{Re}(\alpha^{H}\hat{\alpha}^{\circ})+c\leq \varrho.
  \end{array}
\end{equation}
By the definition of $\varrho$, the feasible set of (\ref{fini:cmodel}) is nonempty and has an interior point.
The objective function $\Phi(\alpha)$ is continuous, nonnegative with $\Phi(0)=0$, and
 differentiable except at  points containing zero groups.
The penalty formulation for (\ref{fini:cmodel}) is
\begin{equation}\label{fini:penalty}
    \min\limits_{\alpha\in\mathbb{C}^{d}} \quad  F_{\lambda} (\alpha):= \Phi(\alpha)+\lambda(\alpha^{H}\hat{Y}\alpha-2\textrm{Re}(\alpha^{H}\hat{\alpha}^{\circ})+c-\varrho)_{+},
\end{equation}
for some $\lambda>0$, where $(\cdot)_{+}:=\max\{\cdot,0\}$.
For any nontrivial solution $\alpha^*$  of (\ref{fini:penalty}), we have
\begin{equation}\label{penalty:opti}
 0< F_{\lambda} (\alpha^*)=\min\limits_{\alpha\in\mathbb{C}^{d}} F_{\lambda} (\alpha)<F_{\lambda}(0)=\lambda(c-\varrho)_+=\lambda(c-\varrho)<\lambda c.
\end{equation}
In Sections 3-4, we focus on problems (\ref{fini:cmodel}) and (\ref{fini:penalty}).

\section{Exact penalization}
\label{sec:eactpenalization}
In this section, we consider the relationship between problems (\ref{fini:cmodel}) and (\ref{fini:penalty}).  We first give some notations.
For a closed set $S\subset\mathbb{C}^n$,  ${\rm{dist}}(z,S)=\inf_{z'\in S}\|z-z'\|$ denotes the distance from a point  $z\in\mathbb{C}^n$ to $S$ and $\mathbf{B}(b;r)=\{z\in \mathbb{C}^n:{{\|z-b\|}}\leq r\}$ denotes a closed ball with radius $r>0$ and center $b\in\mathbb{C}^n$.
Let
$g: \mathbb{C}^{d}\rightarrow\mathbb{R}$ be defined as
 $$g(\alpha):=\alpha^{H}\hat{Y}\alpha-2\textrm{Re}(\alpha^{H}\hat{\alpha}^{\circ})+c-\varrho,$$
  $S(\alpha):=\{\delta\in\mathbb{R}: g(\alpha)\leq\delta\}$ for $\alpha\in \mathbb{C}^{d}$,
and $S^{-1}(\delta):=\{\alpha\in \mathbb{C}^{d}: g(\alpha)\leq\delta\}$.  We denote the feasible set of problem (\ref{fini:cmodel})  by
  $\mathcal{F}_e:=\{\alpha\in \mathbb{C}^{d}: g(\alpha)\leq0\}.$  Let $\mathbb{L}:=\{0,1,\ldots, L\}$.

Since $\hat{Y}$ is positive definite, $g$ is strongly convex and
 has a unique global minimizer $\hat{Y}^{-1}\hat{\alpha}^{\circ}\neq\hat{\alpha}^{\circ}$
which implies that  $g(\hat{\alpha}^\circ)\in (\inf_{\alpha\in \mathbb{C}^{d}}g(\alpha),\infty).$
Since $g$ is strongly convex and quadratic, there is $C$ such that $\|\alpha-\bar{\alpha}\|\leq C\vert g(\alpha)-g(\bar{\alpha})\vert$ for $\alpha,\bar{\alpha}\in \mathbb{C}^{d}$.
Choosing $\bar{\alpha}$ such that $g(\bar{\alpha})=0$,
from \cite[Theorem 9.48]{rock}, we obtain the following lemma.
\begin{lemma}\label{dis}
There exists a constant $C>0$ such that for any $\alpha\in \mathbb{C}^{d}$,
$${\rm{dist}}(\alpha,\mathcal{F}_e){={\rm{dist}}(\alpha,S^{-1}(0))}\leq C{\rm{dist}}(0,S(\alpha))=C(g(\alpha))_+.$$
\end{lemma}

 \begin{theorem}\label{pena:local} There exists a $\lambda^*>0$ such that a local minimizer $\alpha^*\in\mathbb{C}^{d}$ of problem (\ref{fini:cmodel})  is  a local minimizer of problem (\ref{fini:penalty}) whenever $\lambda\geq\lambda^*$.
 \end{theorem}
\begin{proof}
Let $\alpha^*\in\mathbb{C}^{d}$ be a local minimizer of problem (\ref{fini:cmodel}){, that is there exists a neighborhood $\mathcal{N}$ of $\alpha^*$ such that $\Phi(\alpha^*)\leq\Phi(\alpha)$ for $\alpha\in\mathcal{N}\cap\mathcal{F}_e$}.
We denote the group support set of $\alpha^{*}$ by
$\gamma:=\{l\in\mathbb{L}: \|\alpha_{l\cdot}^{*}\|\neq0\}$, and the complement set of $\gamma$ in $\mathbb{L}$ by $\tau:=\{l\in\mathbb{L}: \|\alpha_{l\cdot}^{*}\|=0\}$.
Let $\alpha_{\gamma}$ and $\alpha_{\tau}$ denote the restrictions of $\alpha$ onto $\gamma$ and $\tau$, respectively.
We obtain that $\alpha^*_{\gamma}$ is a local minimizer of the following problem
\begin{equation}\label{model:gam}
  \begin{array}{cl}
    \min\limits_{\alpha_{\gamma}} &  \sum_{l\in\gamma}\beta_{l}\|\alpha_{l\cdot}\|^{p} \\
    {\rm{s.t.}} &  \alpha_{\gamma}^{H}\hat{Y}_{\gamma}\alpha_{\gamma}-2\textrm{Re}(\alpha_{\gamma}^{H}\hat{\alpha}_{\gamma}^{\circ})+c\leq \varrho.
  \end{array}
\end{equation}
Let $\bar{\epsilon}=\tfrac{1}{2}\min\{\|\alpha^*_{l\cdot}\|: l\in\gamma\}>0$. Then, there exists a small $\delta^*$ such that $\alpha^*_{\gamma}$ is a local minimizer of ({\ref{model:gam}}) and $\min\{\|\alpha_{l\cdot}\|: l\in\gamma\}>\bar{\epsilon}$ for all $\alpha_{\gamma}\in\mathbf{B}(\alpha^*_{\gamma};\delta^*)$. Let
$$g_\gamma(\alpha_\gamma)= \alpha_{\gamma}^{H}\hat{Y}_{\gamma}\alpha_{\gamma}-2{\textrm{Re}}(\alpha_{\gamma}^{H}\alpha_{\gamma}^{\circ})+c- \varrho \quad {\rm and} \quad
 \Omega_{1}:=\{\alpha_{\gamma}: g_\gamma(\alpha_\gamma)\leq0\}.$$
Let $[\alpha_\gamma; 0_\tau]$ be the vector with $([\alpha_\gamma; 0_\tau])_{l\cdot}=\alpha_{l\cdot}, l\in \gamma$
and $([\alpha_\gamma; 0_\tau])_{l\cdot}=0, l\in \tau.$   It is easy to see that $g([\alpha_\gamma; 0_\tau]) =g_\gamma(\alpha_\gamma)$ and ${\rm{dist}}(\alpha_{\gamma},\Omega_{1})={\rm{dist}}([\alpha_\gamma; 0_\tau],{\cal F}_e).$
 By Lemma \ref{dis},  ${\rm{dist}}(\alpha_{\gamma},\Omega_{1})\leq C(g_\gamma(\alpha_{\gamma}))_{+}$ for all $\alpha_{\gamma}\in\mathbf{B}(\alpha^*_{\gamma};\delta^*)$.

The objective function of  ({\ref{model:gam}}) is Lipschitz continuous on  $\mathbf{B}(\alpha^*_{\gamma};\delta^*)$. Then by \cite[Lemma 3.1]{chen2016}, there exists a $\lambda^*>0$ such that,  for  any $\lambda\geq\lambda^*$, $\alpha^*_{\gamma}$ is a local minimizer of the following problem
\begin{equation*}\label{model:pen:gam}
    \min\limits_{\alpha_{\gamma}}   F_{\lambda}^{\gamma}(\alpha_{\gamma}):= \sum_{l\in\gamma}\beta_{l}\|\alpha_{l\cdot}\|^{p}
 +\lambda(g_\gamma(\alpha_{\gamma}))_{+},
\end{equation*}
that is, there exists a neighborhood $U_{\gamma}$ of $0$ with $U_{\gamma}\subseteq \mathbf{B}(0;\delta^*)$ such that
\begin{equation*}
  F_{\lambda}^{\gamma}(\alpha_{\gamma})\geq F_{\lambda}^{\gamma}(\alpha^*_{\gamma}), \quad \forall  \, \alpha_{\gamma}\in\alpha_{\gamma}^*+U_{\gamma}.
\end{equation*}
We now show that $\alpha^*$ is a local minimizer of (\ref{fini:penalty}) with $\lambda\geq\lambda^*$.

For  fixed $\epsilon \in (0, \bar{\epsilon})$ and $\lambda\ge \lambda^*$, we consider problem (\ref{fini:penalty}) in the neighborhood $U:=U_{\gamma}\times (-\epsilon,\epsilon)^{d_\tau}$, where $d_\tau=\sum_{l \in \tau} (2l+1)$.
Let $\tilde{\kappa}$ be a Lipschitz constant of the function $\lambda (g(\alpha))_+$ over $\alpha^*+U$. There exists an $\epsilon_{0}\in(0,\epsilon)$ such that whenever $\|\alpha_{l\cdot}\|<\epsilon_{0}$, $l\in \tau$,
\begin{equation}\label{phibound}
\Phi(\alpha)
= \sum\limits_{l\in \gamma}\beta_l\|\alpha_{l\cdot}\|^{p} + \sum\limits_{l\in \tau}\beta_l\|\alpha_{l\cdot}\|^{p} \geq
\sum\limits_{l\in \gamma}\beta_l\|\alpha_{l\cdot}\|^{p} + \tilde{\kappa}\sum\limits_{l\in \tau}\|\alpha_{l\cdot}\|.
\end{equation}
Hence for any $y\in U_{\gamma}\times (-\epsilon_{0},\epsilon_{0})^{d_\tau}$, we have
\begin{eqnarray*}
  F_{\lambda}(\alpha^*+y) &=& \lambda (g(\alpha^*+y))_++\sum\limits_{l\in\gamma}\beta_{l}\| \alpha^*_{l\cdot}+y_{l\cdot}\|^{p}+\sum\limits_{l\in\tau}\beta_{l}\| y_{l\cdot}\|^{p}\\
   &\geq & \lambda (g( [\alpha_{\gamma}^*+y_{\gamma};0_\tau]))_+
                           -\tilde{\kappa}\|y_{\tau}\|+\sum\limits_{l\in\gamma}\beta_{l}\| \alpha^*_{l\cdot}+y_{l\cdot}\|^{p}+ \tilde{\kappa}\|y_{\tau}\|\\
    &\geq & F^{\gamma}_{\lambda}(\alpha_{\gamma}^*)=F_{\lambda}(\alpha^*),
\end{eqnarray*}
where the first inequality follows from the Lipschitz continuity of $g_{\lambda}$ with Lipschitz constant $\tilde{\kappa}$ and (\ref{phibound}), and the last inequality follows from the local
optimality of $\alpha_{\gamma}^*$.
Thus, $\alpha^*$ is a local minimizer of problem (\ref{fini:penalty}) with $\lambda\geq\lambda^*$. This completes the proof.
\end{proof}

\begin{theorem} Let $\tilde{\alpha}=\hat{Y}^{-1}\hat{\alpha}^{\circ}$, $\varepsilon>0$ and
$\lambda\geq C\bar{\beta}^{\frac{1}{p}}(\varepsilon(L+1)^{\frac{p}{2}-1})^{-\frac{1}{p}}\Phi(\tilde{\alpha})$,  where $C$ is defined in Lemma \ref{dis} and $\bar{\beta}=\eta^LL^p$. Then for any global minimizer $\alpha^*$ of problem (\ref{fini:penalty}), the projection $\alpha_{\varepsilon}:=\mathcal{P}_{\mathcal{F}_e}(\alpha^*)$ is an $\varepsilon$-minimizer of (\ref{fini:cmodel}),  that is, $\Phi(\alpha_{\varepsilon})\leq\min\{\Phi(\alpha):\alpha\in\mathcal{F}_e\}+\varepsilon$.
\end{theorem}
\begin{proof}  Note that $g$ is a strongly convex function and $\tilde{\alpha}$ is a minimizer of $g$. Since the feasible set ${\cal F}_e$ of (\ref{fini:cmodel}) has an interior point, we have $\tilde{\alpha} \in {\rm int} {\cal F}_e$.

By the global optimality of $\alpha^*$, we have $F_{\lambda}(\alpha^*)\leq F_{\lambda}(\tilde{\alpha})$ and
\begin{equation}\label{inq1}
  ((\alpha^*)^H\hat{Y}\alpha^*-2{\rm{Re}}((\alpha^*)^H\hat{\alpha}^{\circ})+c-\varrho)_{+}\leq \tfrac{1}{\lambda} F_{\lambda}(\alpha^*)\leq \tfrac{1}{\lambda}F_{\lambda}(\tilde{\alpha})= \tfrac{1}{\lambda}\Phi(\tilde{\alpha}).
\end{equation}
Hence for any $\alpha\in\mathcal{F}_e$, we obtain
\begin{eqnarray*}
\Phi(\alpha_\varepsilon)-\Phi(\alpha)
&\leq&  \sum_{l=0}^{L}\beta_{l}(\|(\alpha_\varepsilon)_{l\cdot}\|^{p}-\|\alpha^*_{l\cdot}\|^{p})
        \leq  \sum\limits_{l=0}^{L}\beta_{l}\|(\alpha_\varepsilon)_{l\cdot}-\alpha^*_{l\cdot}\|^{p}
   \leq \bar{\beta} \sum\limits_{l=0}^{L}\left(\|(\alpha_\varepsilon)_{l\cdot}-\alpha^*_{l\cdot}\|^{2}\right)^{\frac{p}{2}} \\
& \leq& \bar{\beta} (L+1)\left(\tfrac{1}{L+1}\sum\limits_{l=0}^{L}\|(\alpha_\varepsilon)_{l\cdot}-\alpha^*_{l\cdot}\|^{2}\right)^{\frac{p}{2}}
   \le \bar{\beta}  (L+1)^{1-\frac{p}{2}}({\rm{dist}}(\alpha^*,\mathcal{F}_e))^p\\
&\leq&\bar{\beta} (L+1)^{1-\frac{p}{2}}(C((\alpha^*)^H\hat{Y}\alpha^*-2{\rm{Re}}((\alpha^*)^H\hat{\alpha}^{\circ})+c-\varrho)_+)^p\\
&\leq&\bar{\beta} (L+1)^{1-\frac{p}{2}}\left(\tfrac{C}{\lambda}\Phi(\tilde{\alpha})\right)^p
   \leq\varepsilon,
\end{eqnarray*}
where the first inequality is from the global optimality of $\alpha^*$, the second inequality is from Lemma 2.4 in \cite{chen2016}, the fifth inequality is from the concavity of function $t\rightarrow t^\frac{p}{2}$ for $t\geq0$, the sixth inequality is from Lemma \ref{dis}, the seventh inequality is from (\ref{inq1}) and the last inequality follows from  the choice of $\lambda$.
Thus, the projection $\mathcal{P}_{\mathcal{F}_e}(\alpha^*)$ is an $\varepsilon$-minimizer of (\ref{fini:cmodel}). This completes the proof.
\end{proof}

\section{Optimality conditions and a smoothing penalty algorithm}

In this section, we first define  first-order optimality conditions  of
(\ref{fini:cmodel}) and (\ref{fini:penalty}), which are necessary conditions for local optimality. We also derive lower bounds for the $\ell_2$ norm of nonzero groups of first-order stationary points of (\ref{fini:cmodel}) and (\ref{fini:penalty}). Next we propose a smoothing penalty algorithm for solving (\ref{fini:cmodel}) and prove its convergence to a first-order stationary point of
(\ref{fini:cmodel}).

\subsection{First-order optimality conditions}
   We first present a first-order optimality condition of   problem (\ref{fini:penalty}).
\begin{definition} We call  $\alpha^{*}\in\mathbb{C}^{d}$ a scaled first-order stationary point of problem (\ref{fini:penalty}) if
\begin{equation}\label{def:pena:first}
 p\beta_{l}\|\alpha^{*}_{l\cdot}\|^{p}\alpha^{*}_{l\cdot}+ 2\lambda \xi\|\alpha^{*}_{l\cdot}\|^{2} (\hat{Y}_{l\cdot}\alpha^{*}-\alpha^{\circ}_{l\cdot})=0,\quad \forall\,l\in\mathbb{L},
\end{equation}
for some $\xi$ satisfying $\xi\left\{\begin{array}{ll}
                  =0 & {\rm{if}} \,(\alpha^*)^{H}\hat{Y}\alpha^*-2{\rm{Re}}((\alpha^*)^{H}\hat{\alpha}^{\circ})+c<\varrho\\
                  \in[0,1] & {\rm{if}} \,(\alpha^*)^{H}\hat{Y}\alpha^*-2{\rm{Re}}((\alpha^*)^{H}\hat{\alpha}^{\circ})+c=\varrho\\
                  =1 & {\rm{otherwise}}.
                \end{array}\right.
$
\end{definition}
\begin{theorem}\label{pena:firstorder} Let
 $\alpha^{*}\in\mathbb{C}^{d}$ be a local minimizer of (\ref{fini:penalty}).
 Then  $\alpha^{*}$ is a scaled first-order stationary point of (\ref{fini:penalty}).
\end{theorem}
\begin{proof}
Let $\alpha^{*}\in\mathbb{C}^{d}$ be a local minimizer of (\ref{fini:penalty}).
We  obtain that $\alpha^{*}_{\gamma}$ is a local minimizer of
\begin{equation}\label{fini:gammapenalty}
    \min\limits_{\alpha_{\gamma}} \quad  \sum\limits_{l\in\gamma}\beta_{l}\|\alpha_{l\cdot}\|^{p}+\lambda(\alpha_{\gamma}^{H}\hat{Y}_{\gamma}\alpha_{\gamma}-2\textrm{Re}(\alpha_{\gamma}^{H}\hat{\alpha}_{\gamma}^{\circ})+c-\varrho)_{+}.
\end{equation}
Note that $\|\alpha^{*}_{l\cdot}\|\neq0$, $l\in\gamma$, the first term of the objective function of (\ref{fini:gammapenalty}) is continuously differentiable at $\alpha^{*}_{\gamma}$.
The first-order necessary optimality condition for problem (\ref{fini:gammapenalty}) holds at $\alpha^{*}_{\gamma}$, that is,
\begin{equation}\label{pf:first}
  p\beta_{l}\|\alpha^{*}_{l\cdot}\|^{p-2}\alpha^{*}_{l\cdot}+2\lambda \xi ((\hat{Y}_{\gamma}\alpha_{\gamma}^{*})_{l\cdot}-\alpha^{\circ}_{l\cdot})=0,\quad \forall\,l\in\gamma,
\end{equation}
for some $\xi$ satisfying $\xi\left\{\begin{array}{ll}
                  =0 & {\rm{if}} \,(\alpha_{\gamma}^*)^{H}\hat{Y}_{\gamma}\alpha_{\gamma}^*-2\textrm{Re}((\alpha_{\gamma}^*)^{H}\hat{\alpha}_{\gamma}^{\circ})+c<\varrho\\
                  \in[0,1] & {\rm{if}} \,(\alpha_{\gamma}^*)^{H}\hat{Y}_{\gamma}\alpha_{\gamma}^*-2\textrm{Re}((\alpha_{\gamma}^*)^{H}\hat{\alpha}_{\gamma}^{\circ})+c=\varrho\\
                  =1 & {\rm{otherwise}}.
                \end{array}\right.
$

Multiplying $\|\alpha^{*}_{l\cdot}\|^{2}$ on both sides of (\ref{pf:first}), we obtain
\begin{equation*}
  p\beta_{l}\|\alpha^{*}_{l\cdot}\|^{p}\alpha^{*}_{l\cdot}+2\lambda \xi\|\alpha^{*}_{l\cdot}\|^{2} ((\hat{Y}_{\gamma}\alpha_\gamma^{*})_{l\cdot}-\alpha^{\circ}_{l\cdot})=0,\quad \forall\,l\in\gamma.
\end{equation*}
Since $\|\alpha^{*}_{l\cdot}\|=0$ for $l\in\tau$,  we have
\begin{equation*}
  p\beta_{l}\|\alpha^{*}_{l\cdot}\|^{p}\alpha^{*}_{l\cdot}+2\lambda \xi\|\alpha^{*}_{l\cdot}\|^{2} (\hat{Y}_{l\cdot}\alpha^{*}-\alpha^{\circ}_{l\cdot})=0,\quad \forall\,l\in\mathbb{L}.
\end{equation*}
Hence
 (\ref{def:pena:first}) holds at $\alpha^{*}$.
\end{proof}

The next theorem gives a lower bound for the $l_{2}$ norm of  nonzero groups of stationary points of (\ref{fini:penalty}).
\begin{theorem}\label{pena:lowerbound}Let
 $\alpha^{*}\in\mathbb{C}^{d}$ be a  scaled first-order stationary point of (\ref{fini:penalty}) and $\|\hat{Y}\alpha^{*}-\hat{\alpha}^\circ\|\leq \tilde{c}$ for some $\tilde{c}>0$.
 Then,
 \begin{equation*}\label{lowerbound:unfirst}
   \|\alpha^{*}_{l\cdot}\|\geq\left(\frac{p\beta_{l}}{2\lambda \tilde{c}}\right)^{\frac{1}{1-p}}, \forall\,l\in\gamma.
 \end{equation*}
\end{theorem}
The proof is similar with that of (\ref{lowerbound:kkt2}), and thus we omit it here.
According to Theorems \ref{pena:local} and \ref{pena:firstorder}, Theorem \ref{pena:lowerbound} gives a lower bound for the $\ell_{2}$ norm of nonzero groups of local minimizers of problem (\ref{fini:cmodel}).

Next, we consider the first-order optimality condition of problem (\ref{fini:cmodel}).
Similar to  Definition \ref{def:kkt}, we call $\alpha^*$ is a scaled first-order stationary point or a scaled KKT point of (\ref{fini:cmodel}) if the following conditions hold,
\begin{equation}\label{finite:kkt}
\begin{split}
p\beta_{l}\|\alpha^*_{l\cdot}\|^{p}\alpha^*_{l\cdot}+2\nu\|\alpha^*_{l\cdot}\|^{2}(\hat{Y}_{l\cdot}\alpha^*-\alpha_{l\cdot}^\circ) = 0, & \,\,\,  \forall\,l\in\mathbb{L},\\
 \nu g(\alpha^*)=0,\quad
 g(\alpha^*)\leq 0, \quad
 \nu\geq0.&
\end{split}
\end{equation}
Since $\alpha=0$ is not a feasible point, thus $\nu\neq0$. Hence,  replacing $\lambda$ by $\nu$ in Theorem \ref{pena:lowerbound}, we can obtain a lower bound for the $\ell_2$ norm of nonzero groups of the scaled KKT point of problem (\ref{fini:cmodel}).

\begin{theorem}\label{them:44}
Let $\alpha^*$ be a local minimizer of problem (\ref{fini:cmodel}). If  there exists $l\in\mathbb{L}$ such that
$\|\alpha_{l\cdot}^*\|(Y_{l\cdot}\alpha^*-\alpha_{l\cdot}^\circ)\neq0$,
 then $\alpha^*$ is a scaled KKT point of (\ref{fini:cmodel}).
\end{theorem}
\begin{proof}
Let $\alpha^*$ be a local minimizer of problem (\ref{fini:cmodel}), then $\alpha_{\gamma}^*$ is a local minimizer of problem (\ref{model:gam}).
Recall  $g_{\gamma}(\alpha_{\gamma})=\alpha_{\gamma}^{H}\hat{Y}_{\gamma}\alpha_{\gamma}-2\textrm{Re}(\alpha_{\gamma}^{H}\hat{\alpha}_{\gamma}^{\circ})+c-\varrho$,
then $g_\gamma(\alpha^*_\gamma)=g(\alpha^*)$.

The condition that  there exists $l\in\mathbb{L}$ such that
$\|\alpha_{l\cdot}^*\|(Y_{l\cdot}\alpha^*-\alpha_{l\cdot}^\circ)\neq0$
implies that $\hat{Y}_{\gamma}\alpha^*_{\gamma}-\hat{\alpha}_{\gamma}^\circ\neq0$ if $g_\gamma(\alpha_\gamma^*)=0$.
Thus the LICQ (linear independence constraint qualification) holds at $\alpha_\gamma^*$ for problem (\ref{model:gam}).

Note that, the objective function of problem (\ref{model:gam}) and $g_{\gamma}$ are continuously differentiable at $\alpha_{\gamma}^*$ and $\partial_{\bar{\alpha}_{\gamma}}g_{\gamma}(\alpha_{\gamma}^*)=\hat{Y}_{\gamma}\alpha_{\gamma}^*-\hat{\alpha}_{\gamma}^\circ.$
Hence  $\alpha_{\gamma}^*$ is a KKT point of (\ref{model:gam}), that is,
there exists $\nu$ such that
\begin{equation}\label{optimality:kkt}
\begin{split}
p\beta_{l}\|\alpha^*_{l\cdot}\|^{p-2}\alpha^*_{l\cdot}+2\nu(\hat{Y}_\gamma\alpha^*_\gamma-\alpha_{\gamma}^\circ)_{l\cdot} = 0, & \,\,\,  \forall\,l\in\gamma,\\
 \nu g_{\gamma}(\alpha_{\gamma}^*)=0,\quad
 g_{\gamma}(\alpha_{\gamma}^*)\leq 0, \quad
 \nu\geq0.&
\end{split}
\end{equation}
Multiplying $\|\alpha^{*}_{l\cdot}\|^{2}$ on both sides of the first equality in (\ref{optimality:kkt}), we obtain
$$p\beta_{l}\|\alpha^*_{l\cdot}\|^{p}\alpha^*_{l\cdot}+2\nu\|\alpha^*_{l\cdot}\|^{2}(\hat{Y}_\gamma\alpha^*_\gamma-\alpha_{\gamma}^\circ)_{l\cdot} = 0,\quad  \forall\,l\in\gamma.$$
Since $\alpha^*_{l\cdot}=0$ for $l\in \tau$, we obtain
$$p\beta_{l}\|\alpha^*_{l\cdot}\|^{p}\alpha^*_{l\cdot}+2\nu\|\alpha^*_{l\cdot}\|^{2}(\hat{Y}_{l\cdot}\alpha^*-\alpha_{l\cdot}^\circ) = 0,\quad  \forall\,l\in\mathbb{L}.$$
Combining this with (\ref{optimality:kkt}) and $g(\alpha^*)=g_\gamma(\alpha_\gamma^*)$, we find that
 $\alpha^*$ is a scaled KKT point of (\ref{fini:cmodel}). The proof is completed.
\end{proof}

From the definitions we can see that for some $\lambda>0$, any scaled KKT point of problem (\ref{fini:cmodel}) is a scaled first order stationary point of (\ref{fini:penalty}). Moreover, any
scaled first order stationary point of problem (\ref{fini:penalty}) which belongs to the feasible set  $\mathcal{F}_e$ is a scaled KKT point of problem (\ref{fini:cmodel}).

Following the proof of Theorem \ref{them:44}, we can show the existence of a KKT point of problem (\ref{model2}) by replacing $\mathbb{L}$ by $\mathbb{N}_0$.
\begin{corollary}
Let $\tilde{\alpha}$ be a local minimizer of problem (\ref{model2}). If  there exists $l\in\mathbb{\mathbb{N}}_0$ such that
$\|\tilde{\alpha}_{l\cdot}\|(Y_{l\cdot}\tilde{\alpha}-\alpha_{l\cdot}^\circ)\neq0$,
 then $\tilde{\alpha}$ is a scaled KKT point of (\ref{model2}).
\end{corollary}

\subsection{A smoothing penalty algorithm  for problem (\ref{fini:cmodel})}

 We define a smoothing function of  the nonsmooth function $\lambda (g(\alpha))_+$ as follows
 $$ f_{\lambda,\mu}(\alpha)=\psi_{\lambda,\mu}(g(\alpha))$$
with $\psi_{\lambda,\mu}(s):=\lambda\max\limits_{0\leq t\leq1}\{st-\frac{\mu}{2}t^2\}$
and $\mu>0$ is a smoothing parameter.
It is easy to verify that
$\psi'_{\lambda,\mu}(s)=\lambda\min\{\max\{\frac{s}{\mu},0\},1\}\geq0$ and $\vert\psi'_{\lambda,\mu}(s_{1})-\psi'_{\lambda,\mu}(s_{2})\vert\leq\frac{\lambda}{\mu}\vert s_{1}-s_{2}\vert$, $\forall\,s_{1},s_{2}\in\mathbb{R}.$
It is not hard to show that
$$f_{\lambda,\mu}(\alpha)=\left\{\begin{array}{ll}
0 & {\rm{if }} \, g(\alpha)\leq0\\
\frac{\lambda}{2\mu}(g(\alpha))^2 & {\rm{if }} \,0\leq g(\alpha)\leq\mu\\
\lambda g(\alpha)-\frac{\lambda\mu}{2} & {\rm{if }} \,g(\alpha)\geq\mu
\end{array}
\right.$$
and
\begin{equation}\label{smooth:diff}
  0\leq \lambda(g(\alpha))_+-f_{\lambda,\mu}(\alpha)\leq\tfrac{\lambda\mu}{2}.
  \end{equation}
By Wirtinger  calculus, we obtain
$\nabla f_{\lambda,\mu}(\alpha)=\small\left[\setlength{\arraycolsep}{0.8pt}\begin{array}{c}
                                                      \partial_{\alpha}f_{\lambda,\mu}(\alpha) \\
                                                      \partial_{\bar{\alpha}}f_{\lambda,\mu}(\alpha)
                                                    \end{array}\right],$
where $\partial_{\bar{\alpha}}f_{\lambda,\mu}(\alpha)=\overline{\partial_{\alpha}f_{\lambda,\mu}(\alpha)}$ and
\begin{equation*}\label{smoothf:partial}
  \partial_{\bar{\alpha}}f_{\lambda,\mu}(\alpha)=\left\{\begin{array}{ll}
0 &{\rm{if }} \, g(\alpha)\leq0\\
\frac{\lambda}{\mu}g(\alpha)(\hat{Y}\alpha-\hat{\alpha}^{\circ}) &{\rm{if }} \, 0\leq g(\alpha)\leq\mu\\
\lambda (\hat{Y}\alpha-\hat{\alpha}^{\circ})  &{\rm{if }} \, g(\alpha)\geq\mu.
\end{array}
\right.
\end{equation*}
More details about the smoothing function can be found in \cite{CHEN2012} and references therein.

We consider the following optimization problem
  \begin{eqnarray}\label{penalty:smooth}
       \min\limits_{\alpha\in\mathbb{C}^{d}}&& F_{\lambda,\mu}(\alpha):=
\Phi(\alpha)+f_{\lambda,\mu}(\alpha).
\end{eqnarray}
For fixed positive parameters $\lambda$ and $\mu$, $F_{\lambda,\mu}$ is continuous and level-bounded since $\Phi$ is level-bounded and $f_{\lambda,\mu}$ is nonnegative. Moreover, the gradient of $f_{\lambda,\mu}$ is Lipschitz continuous.

Now, we  propose a smoothing penalty algorithm for solving problem (\ref{fini:cmodel}).
\begin{algorithm}
{A smoothing penalty algorithm  for problem (\ref{fini:cmodel})}\label{alg:smoothnpg}
\begin{algorithmic}
\STATE {
 Choose $\lambda^{0}>0$, $\mu^{0}>0$, $\varepsilon^{0}>0$, $\varsigma_{1}>1$, and $0<\varsigma_{2}<1$.
  Set $k=0$ and $\alpha^{0}=\tilde{\alpha}:=\hat{Y}^{-1} \alpha^{{\circ}}$.}
\STATE {(1) If $F_{\lambda^{k},\mu^{k}}(\alpha^{k})> F_{\lambda^{k},\mu^{k}}(\tilde{\alpha})$, set  $\alpha^{k}=\tilde{\alpha}$; otherwise $\alpha^{k}=\alpha^k$.}
\STATE {(2)  Solve problem (\ref{penalty:smooth})  with initial point $\alpha^{k}$,  $\lambda=\lambda^{k}$, $\mu=\mu^{k}$,  and  find an $\alpha^{k+1}$ satisfying
              \begin{equation}\label{alg2:dist}
              \| p\beta_{l}\|\alpha^{k+1}_{l\cdot}\|^{p}\alpha^{k+1}_{l\cdot}+2\|\alpha^{k+1}_{l\cdot}\|^{2}(\partial_{\bar{\alpha}}f_{\lambda^{k},\mu^{k}}(\alpha^{k+1}))_{l\cdot} \|\leq\varepsilon^{k},\quad \forall \,l\in\mathbb{L}.
              \end{equation}}\vspace{-4mm}
\STATE{(3) Set  $\lambda^{k+1}= \varsigma_{1}\lambda^{k}$, $\mu^{k+1}=\varsigma_{2}\mu^{k+1}$, $\varepsilon^{k+1}=\varsigma_{2}\varepsilon^{k}$.}
\STATE{(4)  Set $k= k+1$ and go to  (1).}
\end{algorithmic}
\end{algorithm}

We give the convergence of Algorithm \ref{alg:smoothnpg} in the following theorem.
\begin{theorem}Let  $\{\alpha^{k}\}$ be generated by Algorithm \ref{alg:smoothnpg}. Then, the following statements hold.
\begin{enumerate}
  \item [(i)] $\{\alpha^{k}\}$ is bounded.
  \item [(ii)]  Any accumulation point $\alpha^{*}$ of $\{\alpha^{k}\}$ is a scaled KKT point of problem (\ref{fini:cmodel}).
\end{enumerate}
\end{theorem}
\begin{proof}
(i) We can see that
$$\Phi(\alpha^{k+1})\leq F_{\lambda^{k},\mu^{k}}(\alpha^{k+1})\leq F_{\lambda^{k},\mu^{k}}(\tilde{\alpha})=\Phi(\tilde{\alpha}),$$
where the first inequality follows from that $f_{\lambda^{k},\mu^{k}}(\alpha^{k+1})\geq0$, the second inequality  follows from  step (1) of Algorithm \ref{alg:smoothnpg}, and the equality is from $\tilde{\alpha}=\hat{Y}^{-1} \alpha^{{\circ}}\in\mathcal{F}_e$ and $f_{\lambda^{k},\mu^{k}}(\tilde{\alpha})=0$. Since $\Phi$ is level-bounded, $\{\alpha^{k}\}$ is bounded.

(ii) Let $\alpha^{*}$ be an accumulation point of $\{\alpha^{k}\}$ and $\{\alpha^{k}\}_{k\in\mathcal{K}}$ be a  subsequence of $\{\alpha^{k}\}$ such that $\{\alpha^{k}\}\rightarrow\alpha^{*}$ as $k\rightarrow\infty$, $k\in\mathcal{K}$.
Note that
$$\lambda^{k-1}(g(\alpha^{k}))_{+}-\tfrac{\lambda^{k-1}\mu^{k-1}}{2}\leq f_{\lambda^{k-1},\mu^{k-1}}(\alpha^{k})\leq
 F_{\lambda^{k-1},\mu^{k-1}}(\alpha^{k})\leq F_{\lambda^{k-1},\mu^{k-1}}(\tilde{\alpha})\leq\Phi(\tilde{\alpha}),$$
where the first inequality follows from (\ref{smooth:diff}). Then, we have
\begin{equation}\label{ineq:g1}
  (g(\alpha^{k}))_{+}\leq\frac{\Phi(\tilde{\alpha})}{\lambda^{k-1}}+\frac{\mu^{k-1}}{2}.
\end{equation}
From step (3) in Algorithm \ref{alg:smoothnpg}, $\lambda^{k-1}\rightarrow\infty$ {{and $\mu^{k-1}\rightarrow0$}}, as $k\rightarrow\infty$, $k\in\mathcal{K}$.
Taking limits in  (\ref{ineq:g1}) as $k\rightarrow\infty$, $k\in\mathcal{K}$,   we obtain that $(g(\alpha^{*}))_{+}\leq0$. Hence, $\alpha^{*}\in\mathcal{F}_e$.

From (\ref{alg2:dist}), we have
\begin{equation}\label{pf:dist}
  \| p\beta_{l}\|\alpha^{k}_{l\cdot}\|^{p}\alpha^{k}_{l\cdot}+2\|\alpha^{k}_{l\cdot}\|^{2}(\partial_{\bar{\alpha}}f_{\lambda^{{k}},\mu^{{k}}}(\alpha^{k}))_{l\cdot} \|\leq\varepsilon^{{k-1}},\quad \forall \,l\in\mathbb{L}.
\end{equation}
We first assume that
$g(\alpha^{*})<0$. For all sufficiently large  $k$, we obtain $g(\alpha^{k})<0$ and  (\ref{pf:dist}) becomes
$$ p\beta_{l}\|\alpha^{k}_{l\cdot}\|^{p+1}\leq\varepsilon^{{k-1}},\quad \forall \,l\in\mathbb{L}.$$
Taking limits on both sides of the above relation, we obtain $\alpha^{*}=0$ which contradicts to Assumption \ref{assu:inpa}. Thus, $g(\alpha^{*})=0$.

Let $t^{k}:=\psi'_{\lambda^{k},\epsilon^{k}}(g(\alpha^{k}))$ for notational simplicity, we have $t^{k}\geq0$.
Then (\ref{pf:dist}) reduces to
\begin{equation}\label{pf:kkt}
  \| p\beta_{l}\|\alpha^{k}_{l\cdot}\|^{p}\alpha^{k}_{l\cdot}+
  2t^{k}\|\alpha^{k}_{l\cdot}\|^{2}(\hat{Y}_{l\cdot}\alpha^{k}-\alpha_{l\cdot}^\circ) \|\leq\varepsilon^{{k-1}},\quad \forall \,l\in\mathbb{L}.
\end{equation}
Now, we prove that $\{t^{k}\}_{\mathcal{K}}$ is bounded. On the contrary, we assume $\{t^{k}\}_{\mathcal{K}}$  is unbounded and $\{t^{k}\}_{\mathcal{K}}\rightarrow\infty$, then,
\begin{equation*}
  \left\| \tfrac{p\beta_{l}}{t^{k}}\|\alpha^{k}_{l\cdot}\|^{p}\alpha^{k}_{l\cdot}+2
  \|\alpha^{k}_{l\cdot}\|^{2}(\hat{Y}_{l\cdot}\alpha^{k}-\alpha_{l\cdot}^\circ) \right\|\leq\tfrac{\varepsilon^{{k-1}}}{t^{k}},\quad \forall \,l\in\mathbb{L}.
\end{equation*}
Passing to the limit in the above relation gives $$\|\alpha^*_{l\cdot}\|^{2}(\hat{Y}_{l\cdot}\alpha^*-\alpha_{l\cdot}^\circ)=0,\quad \forall \,l\in\mathbb{L}.$$
Since $g(\alpha^*)=0$ implies $\alpha^*$ is in $ {\cal F}_e$, but is not a minimizer of $g$, we have  $\alpha^*\neq 0$ and $\hat{Y}\alpha^*-\hat{\alpha}^\circ\neq 0$.
Moreover, $g(\alpha^*)=g_\gamma([\alpha^*_\gamma;0_\tau])=0$ implies that  $\alpha_{l\cdot}^*\neq 0,$ ${\forall}l\in \gamma$ and $(\hat{Y}\alpha^*-\hat{\alpha}^\circ)_\gamma\neq 0$.
Thus, $\{t^{k}\}_{\mathcal{K}}$  is bounded. Let  $\{t^{k}\}_{\mathcal{K}}\rightarrow t^{*}$. Taking limits on both sides of (\ref{pf:kkt}) gives
\begin{equation*}
  \| p\beta_{l}\|\alpha^*_{l\cdot}\|^{p}\alpha^*_{l\cdot}+
  2t^*\|\alpha^*_{l\cdot}\|^{2}(\hat{Y}_{l\cdot}\alpha^*-\alpha_{l\cdot}^\circ) \|=0,\quad \forall \,l\in\mathbb{L}.
\end{equation*}
Therefore, $\alpha^*$ is a scaled KKT point of (\ref{fini:cmodel}).
\end{proof}

\section{Approximation error}
In \cite{q2018}, the authors gave the approximation error for random field using regularized $\ell_{2,1}$ model based on the observed random field  $T^\circ(\omega,{\rm{x}}) \in L_2(\Omega \times \mathbb{S}^2)$.
In this section, we estimate the approximation error of the inpainted random field  in the space $L_2(\Omega \times \mathbb{S}^2)$ based on the observed random field  $T^\circ(\omega,{\rm{x}}) \in L_2(\Omega \times \mathbb{S}^2)$.

For any fixed $\omega\in\Omega$, let $\hat{\alpha}^*(\omega):=(\alpha^*_{0,0}(\omega),\ldots,\alpha^*_{L_{\omega},L_{\omega}}(\omega))^{T}\in\mathbb{C}^{(L_{\omega}+1)^2}$ with the group support set $\gamma_{\omega}$ be a scaled KKT point of problem (\ref{fini:cmodel}).
By our results in the previous sections, $L_{\omega}$ is a finite number.
Moreover,  $\alpha^*(\omega):=((\hat{\alpha}^*(\omega))^T, 0,\ldots )^T$ is a scaled KKT point of problem (\ref{model2}) in the infinite-dimensional space $\ell^p_\beta$.

Let the random field defined by a  scaled KKT point $\alpha^{*}\in\ell^p_\beta$ of problem (\ref{model2}) be  \begin{equation}\label{REG:T}
  T^{*}(\omega,{\rm{x}})=\sum\limits_{l=0}^{\infty}\sum_{m=-l}^{l}\alpha_{l,m}^{*}(\omega)Y_{l,m}(\rm{x}),
\end{equation}
where  $\alpha_{l,m}^*(\omega)=0$, $l=L_{\omega}+1,\ldots$, $m=-l,\ldots,l$.
\begin{lemma}\label{sec5:TL} If the random variable $\omega\in\Omega$ has finite second order moment that is $\mathbb{E}[\|\omega\|^2]<\infty$ and there is $\kappa$ such that
$\|T^{*}(\omega_1,\mathrm{x})-T^{*}(\omega_2,\mathrm{x})\|_{L_{2}(\mathbb{S}^2)}\leq \kappa\|\omega_1-\omega_2\|$, $\forall \omega_1,\omega_2\in\Omega$, then $ T^*(\omega,{\rm{x}})\in L_{2}(\Omega\times\mathbb{S}^2).$
\end{lemma}
\begin{proof}
Let $\tilde{\omega}\in\Omega$ be fixed.
Since  for any $\omega\in\Omega$, $$\|T^{*}(\omega,\mathrm{x})\|_{L_{2}(\mathbb{S}^2)}-\|T^{*}(\tilde{\omega},\mathrm{x})\|_{L_{2}(\mathbb{S}^2)}\leq\|T^{*}(\omega,\mathrm{x})-T^{*}(\tilde{\omega},\mathrm{x})\|_{L_{2}(\mathbb{S}^2)}\leq
\kappa\|\omega-\tilde{\omega}\|,$$ we have
$$\|T^{*}(\omega,\mathrm{x})\|^2_{L_{2}(\mathbb{S}^2)}
\leq\left(\kappa\|\omega-\tilde{\omega}\|+\|T^{*}(\tilde{\omega},\mathrm{x})\|_{L_{2}(\mathbb{S}^2)}\right)^2
\leq2\kappa^2\|\omega-\tilde{\omega}\|^2+2\|T^{*}(\tilde{\omega},\mathrm{x})\|^2_{L_{2}(\mathbb{S}^2)}.$$
Hence, we obtain  $$\mathbb{E}[\|T^{*}(\omega,\mathrm{x})\|^2_{L_{2}(\mathbb{S}^2)}]\leq 
2\|T^{*}(\tilde{\omega},\mathrm{x})\|^2_{L_{2}(\mathbb{S}^2)}+2\kappa^2\mathbb{E}[\|\omega-\tilde{\omega}\|^2]<\infty, $$
where the last inequality follows from $\mathbb{E}[\|\omega\|^2]<\infty$.
Thus, $ T^*(\omega,{\rm{x}})\in L_{2}(\Omega\times\mathbb{S}^2).$
\end{proof}

\begin{theorem}
Let $T^\circ(\omega,{\rm{x}}) \in L_2(\Omega \times \mathbb{S}^2)$ be the  observed random field. 
Then for any $\epsilon>0$ there exists $L$ such that
$$0\leq\|\mathcal{A}(T_{L}^{*}(\omega,{\rm{x}}))-T^\circ(\omega,{\rm{x}})\|^{2}_{L_{2}(\Omega \times \mathbb{S}^2)}-{\varrho}<\epsilon,$$
where $T_{L}^{*}(\omega,{\rm{x}})=\sum_{l=0}^L\sum_{m=-l}^l\alpha^*_{l,m}(\omega)Y_{l,m}({\rm{x}})$.
\end{theorem}
\begin{proof}
Since $T^\circ(\omega,{\rm{x}})\in L_2(\Omega \times \mathbb{S}^2)$, by Fubini's theorem, for $\omega\in\Omega$, $T^\circ(\omega,{\rm{x}})\in L_{2}(\mathbb{S}^{2})$, $P$-a.s., in which case $T^\circ(\omega,{\rm{x}})$ admits an expansion in terms of spherical harmonics, $P$-a.s., that is,
$T^\circ(\omega,{\rm{x}})=\sum_{l=0}^\infty\sum_{m=-l}^l\alpha_{l,m}^{\rm{obs}}(\omega)Y_{l,m}({\rm{x}}), $ $P$-a.s.,
where $\alpha^{\rm{obs}}(\omega)=(\alpha_{0,0}^{\rm{obs}}(\omega),\alpha_{1,-1}^{\rm{obs}}(\omega),\ldots)^T$ is the Fourier coefficient vector of $T^\circ(\omega,{\rm{x}})$.

By Definition \ref{def:kkt}
for  any $\omega\in\Omega$, $\alpha^*(\omega)\neq0$, we have $\nu_\omega>0$. Now, we prove that there exists a positive scalar $\bar{\nu}$ such that $\nu_\omega\geq\bar{\nu}$ for any $\omega\in\Omega$. On the contrary, if there exists a $\omega\in\Omega$ such that $\nu_\omega\rightarrow 0$, then from
\begin{eqnarray*}
   \|\alpha^*(\omega)\|_{2,p}^{p} =\sum\limits_{\{l\in\mathbb{N}_{0}:\,\|\alpha^*_{l\cdot}(\omega)\|\neq0\}}\beta_{l}\|\alpha^*_{l\cdot}(\omega)\|^{p}
    &\geq&\left(\frac{p}{2\nu_\omega\tilde{c}}\right)^{\frac{p}{1-p}}\sum\limits_{\{l\in\mathbb{N}_{0}:\,\|\alpha^*_{l\cdot}(\omega)\|\neq0\}}(\eta^l l^{p})^{\frac{1}{1-p}},
 \end{eqnarray*}
 we have $\|\alpha^*(\omega)\|_{2,p}^{p}\rightarrow\infty$ which is a contradiction with $\alpha^{*}(\omega)\in\ell_{\beta}^p$.
Thus, from
$\alpha^{*}(\omega)\in\ell_{\beta}^p$ and (\ref{lowerbound:kkt2}),  there exists a positive scalar $\bar{\nu}$ such that $\nu_\omega\geq\bar{\nu}$ for any $\omega\in\Omega$.
 Thus, for any $\epsilon_1>0$ there exists $L_1$ such that $\frac{p}{\bar{\nu}}\sum_{l=L_1+1}^{\infty}\beta_{l}\|\alpha_{l\cdot}^*(\omega)\|^p<\frac{\epsilon_1}{2}$ for any $\omega\in\Omega$, which implies that $$\mathbb{E}\left[\frac{p}{\nu_\omega}\sum\limits_{l=L_1+1}^{\infty}\beta_{l}\|\alpha_{l\cdot}^*(\omega)\|^p\right]<\frac{p}{\bar{\nu}}\mathbb{E}\left[\sum\limits_{l=L_1+1}^{\infty}\beta_{l}\|\alpha_{l\cdot}^*(\omega)\|^p\right]<\frac{\epsilon_1}{2}.$$
By Lemma \ref{sec5:TL}, for any  $\epsilon_2>0$ there exists $L_2$ such that $\sum_{l=L_2+1}^{\infty}\mathbb{E}[\|\alpha_{l\cdot}^*(\omega)\|^2]<\frac{\epsilon_2}{2}$.
Let $\epsilon=\max\{\epsilon_1,\epsilon_2\}$, $L=\max\{L_1,L_2\}$ and $d=(L+1)^2$.

 For notational simplicity, let
 $\alpha^*(\omega)=((\hat{\alpha}^*(\omega))^T,(\tilde{\alpha}^*(\omega))^T)^T\in\ell_{\beta}^p$,  where
 $\hat{\alpha}^*(\omega)\in\mathbb{C}^{d}$, $\omega\in\Omega$
 and $Y=\scriptsize\left[\setlength{\arraycolsep}{0.8pt}
         \begin{array}{cc}
           \hat{Y}& X \\
           X^H & \tilde{Y} \\
         \end{array}
       \right]
$, where  $\hat{Y}\in\mathbb{C}^{d\times d}$.
Let
$\tilde{T}_{L}^{*}(\omega,{\rm{x}})=
\sum_{l=L+1}^{\infty}\sum_{m=-l}^l\alpha_{l,m}^*(\omega)Y_{l,m}({\rm{x}})$,
we have
\begin{eqnarray}\label{eq:541}
 \begin{split}
  \varrho&=\|\mathcal{A}(T^{*}(\omega,{\rm{x}}))-T^\circ(\omega,{\rm{x}})\|^{2}_{L_{2}(\Omega \times \mathbb{S}^2)}= \|\mathcal{A}(T_{L}^{*}(\omega,{\rm{x}})+\tilde{T}_{L}^{*}(\omega,{\rm{x}}))-T^\circ(\omega,{\rm{x}})\|^{2}_{L_{2}(\Omega \times \mathbb{S}^2)}\\
  &=\|\mathcal{A}(T_{L}^{*}(\omega,{\rm{x}}))-T^\circ(\omega,{\rm{x}})\|^{2}_{L_{2}(\Omega \times \mathbb{S}^2)} \\ &\quad+\mathbb{E}\left[\int_{\mathbb{S}^2}\vert\mathcal{A}(\tilde{T}_{L}^{*}(\omega,{\rm{x}}))\vert^2d\sigma({\rm{x}})
   +2\int_{\mathbb{S}^2}(\mathcal{A}(T_{L}^{*}(\omega,{\rm{x}}))-T^\circ(\omega,{\rm{x}}))\mathcal{A}(\tilde{T}_{L}^{*}(\omega,{\rm{x}}))d\sigma({\rm{x}})\right]\\
  &=\|\mathcal{A}(T_{L}^{*}(\omega,{\rm{x}}))-T^\circ(\omega,{\rm{x}})\|^{2}_{L_{2}(\Omega \times \mathbb{S}^2)} \\
  &\quad+\mathbb{E}\left[\int_{\Gamma}\vert\tilde{T}_{L}^{*}(\omega,{\rm{x}})\vert^2d\sigma({\rm{x}})
  +2\int_{\Gamma}(T_{L}^{*}(\omega,{\rm{x}})-T^\circ(\omega,{\rm{x}}))\tilde{T}_{L}^{*}(\omega,{\rm{x}})d\sigma({\rm{x}})\right]\\
&=\|\mathcal{A}(T_{L}^{*}(\omega,{\rm{x}}))-T^\circ(\omega,{\rm{x}})\|^{2}_{L_{2}(\Omega \times \mathbb{S}^2)}
+\mathbb{E}[(\tilde{\alpha}^*(\omega))^H\tilde{Y}\tilde{\alpha}^*(\omega)]\\
& \quad+\mathbb{E}[2(\tilde{\alpha}^*(\omega))^HX^H\hat{\alpha}^*(\omega)-
2(\tilde{\alpha}^*(\omega))^H[\begin{array}{cc}
                                X^H & \tilde{Y} \\
                               \end{array}]\alpha^{\rm{obs}}(\omega)],
 \end{split}
\end{eqnarray}
where the first equality follows from the second equality in (\ref{def:infi:kkt}) with  $\nu_\omega>0$ for any $\omega\in\Omega$.
For any fixed $\omega\in\Omega$, by Definition \ref{def:kkt} for a scaled KKT point $\alpha^*\in\ell_{\beta}^p$ of problem (\ref{model2}),   there is $\nu_\omega >0$ such that
$$p\beta_{l}\|\alpha^*_{l\cdot}(\omega)\|^{p-2}\alpha^*_{l\cdot}(\omega)+2 \nu_\omega(Y_{l\cdot}\alpha^*(\omega)-\alpha_{l\cdot}^\circ(\omega)) = 0,\quad
l\in\gamma_\omega.$$
From (\ref{constraint}), for $\omega\in\Omega$,  $\alpha^\circ(\omega)=Y\alpha^{\rm{obs}}(\omega)$. Thus, for any fixed $\omega\in\Omega$,
\begin{equation}\label{eq:lem51}
p\beta_{l}\|\alpha^*_{l\cdot}(\omega)\|^{p}+2 \nu_\omega(\alpha^*_{l\cdot}(\omega))^H(Y_{l\cdot}\alpha^*(\omega)-Y_{l\cdot}\alpha^{\rm{obs}}(\omega)) = 0,\quad l\in\gamma_\omega.
\end{equation}
We obtain
\begin{equation}\label{eq:54}
\sum\limits_{l=L+1}^{\infty}p\beta_{l}\|\alpha^*_{l\cdot}(\omega)\|^{p}+2 \nu_\omega(\tilde{\alpha}^*(\omega))^H[
                                               \begin{array}{cc}
                                                 X^H & \tilde{Y} \\
                                               \end{array}](\alpha^*(\omega)-\alpha^{\rm{obs}}(\omega)) = 0, \, \omega\in\Omega.
\end{equation}
Thus,
\begin{eqnarray}\label{eq:55}
 \begin{split}
&  \mathbb{E}[(\tilde{\alpha}^*(\omega))^H\tilde{Y}\tilde{\alpha}^*(\omega)]+\mathbb{E}[2(\tilde{\alpha}^*(\omega))^HX^H\hat{\alpha}^*(\omega)
-2(\tilde{\alpha}^*(\omega))^H[\begin{array}{cc}
                                  X^H & \tilde{Y} \\
                               \end{array}]\alpha^{\rm{obs}}(\omega)] \\
&=  \mathbb{E}\left[(\tilde{\alpha}^*(\omega))^H\tilde{Y}\tilde{\alpha}^*(\omega)+2(\tilde{\alpha}^*(\omega))^HX^H\hat{\alpha}^*(\omega)
-2(\tilde{\alpha}^*(\omega))^H[\begin{array}{cc}
                                              X^H & \tilde{Y}
                                            \end{array}]\alpha^*(\omega)\right]\\
&\quad-\mathbb{E}\left[\frac{p}{\nu_\omega}\sum\limits_{l=L+1}^{\infty}\beta_{l}\|\alpha^*_{l\cdot}(\omega)\|^{p}
\right]\\
&= \mathbb{E}\left[-(\tilde{\alpha}^*(\omega))^H\tilde{Y}\tilde{\alpha}^*(\omega)-\frac{p}{\nu_\omega}\sum\limits_{l=L+1}^{\infty}\beta_{l}\|\alpha^*_{l\cdot}(\omega)\|^{p}\right]\\                                            &\geq \mathbb{E}\left[-\sum\limits_{l=L+1}^{\infty}\|{\alpha}_{l\cdot}^*(\omega)\|^{2}-\frac{p}{\nu_\omega}\sum\limits_{l=L+1}^{\infty}\beta_{l}\|\alpha^*_{l\cdot}(\omega)\|^{p}\right]
> -\epsilon,
 \end{split}
\end{eqnarray}
where the first equality follows from (\ref{eq:54}) and the first inequality follows from that $\|\tilde{Y}\|\leq1$.
Combining (\ref{eq:541}) and (\ref{eq:55}), we obtain
$$
0\leq\|\mathcal{A}(T_{L}^{*}(\omega,{\rm{x}}))-T^\circ(\omega,{\rm{x}})\|^{2}_{L_{2}(\Omega\times\mathbb{S}^2)}-
\varrho<\epsilon.
$$
The proof is completed.
\end{proof}

\section{Numerical experiments}
In this section,
we conduct numerical experiments to compare the $\ell_{p}$-$\ell_{2}$ optimization model (\ref{fini:cmodel}) with the  $\ell_{1}$ optimization model (31) in \cite{JD} on the inpainting of banded-limited random fields and images from Earth topography data to show the efficiency of problem (\ref{fini:cmodel}) and Algorithm \ref{alg:smoothnpg}.

Following \cite{chen2016}, we adapt the nonmonotone proximal gradient (NPG) method  to solve subproblem (\ref{penalty:smooth}) in Algorithm \ref{alg:smoothnpg}. For completeness, we present the NPG method as follows.
\begin{algorithm}{NPG method for problem (\ref{penalty:smooth})}
\label{alg:npg}
\begin{algorithmic}
\STATE {Given $\alpha^{0}\in\mathcal{F}_e$. Choose $M_{{\rm{max}}}\geq M_{{\rm{min}}}>0$, $\tilde{\eta}>1$, $b>0$ and an integer $N\geq0$. Set $n=0$.}
\STATE {(1) Choose $M_{n}^{0}\in[M_{{\rm{min}}},M_{{\rm{max}}}]$. Set $M_{n}=M_{n}^{0}$.\\
             \quad(a) Solve the subproblem
            $$ y\in\arg\min\limits_{\alpha\in\mathbb{C}^{d}}\left\{\Phi(\alpha)+2{\rm{Re}}\langle\partial_{\bar{\alpha}} f_{\lambda,\mu}(\alpha^n), \alpha-\alpha^n\rangle+M_{n}\|\alpha-\alpha^n\|^{2}\right\}.\vspace{-2mm}
             $$
             \quad(b) If
             $F_{\lambda,\mu}(y)\leq \max\limits_{[n-N]_{+}\leq j\leq n}F_{\lambda,\mu}(\alpha^j)-b\|y-\alpha^n\|^{2}
             $
             is satisfied, go to (2). Otherwise set $M_{n}= \tilde{\eta} M_{n}$, and go to step (a).}
\STATE{(2) Set $\alpha^{n+1}= y$,  $n= n+1$ and go to (1).}

\end{algorithmic}
\end{algorithm}

For the NPG method to solve (\ref{penalty:smooth}) in Algorithm \ref{alg:smoothnpg} at $\lambda=\lambda_k$ $\mu=\mu_k$, we set $M_{{\rm{min}}}=1$, $M_{{\rm{max}}}=10^6$, $\tilde{\eta}=2$, $b=10^{-4}$, $N=4$,
 $M_{0}^{0}=1$, and for any $n\geq1$,
 $$M_{n}^{0}=\min\left\{\max\left\{\frac{\vert(\alpha^{n}-\alpha^{n-1})^{H}(\partial_{\bar{\alpha}} F_{\lambda,\mu}(\alpha^{n})-\partial_{\bar{\alpha}} F_{\lambda,\mu}(\alpha^{n-1}))\vert}{\|\alpha^{n}-\alpha^{n-1}\|^{2}},1\right\},10^6\right\}.$$
Algorithm \ref{alg:npg} is terminated when
\begin{equation*}
  \|\alpha^{n}-\alpha^{n-1}\|_{\infty}\leq\sqrt{\varepsilon^k}\quad  {\rm{and}} \quad
  \frac{\vert F_{\lambda,\mu}(\alpha^{n})-F_{\lambda,\mu}(\alpha^{n-1})\vert}{\max\{1,\vert F_{\lambda,\mu}(\alpha^{n})\vert\}}\leq\min\{(\varepsilon^k)^{2.2},10^{-4}\}.
\end{equation*}
In Algorithm \ref{alg:smoothnpg}, we set $\lambda^0=20$, $\mu^0=\varepsilon^0=1$, $\varsigma_{1}=2$, $\varsigma_{2}=\frac{1}{2}$,  $\alpha^0=\hat{Y}^{-1}\hat{\alpha}^{\circ}$.
The smoothing penalty algorithm is terminated when
$\max\{g(\alpha^k)_+, 0.01\varepsilon^k \}\leq10^{-6},$
where $\varepsilon^k$ is updated by $\varepsilon^{k+1}=\max\{\varsigma_{2}\varepsilon^k,10^{-6}\}$ instead of $\varsigma_{2}\varepsilon^k$ in the experiments.
All codes were written in MATLAB and the realizations were implemented in Python.

\subsection{Random data}
In the following experiments, we consider randomly generated instances which are generated as follows. First, we choose a subset $D\subset\{0,1,\ldots,L-1\}$ and generate a group sparse coefficient vector $\alpha_{L}^{{\rm{true}}}\in\mathbb{C}^{(L+1)^2\times(L+1)^2}$ such that $\alpha^{{\rm{true}}}_{l\cdot}=0$ if $l\in D$ and $\alpha^{{\rm{true}}}_{l\cdot}={\alpha_{l\cdot}^{{\rm{cmb}}}}/{\|\alpha_{L}^{{\rm{cmb}}}\|^{1.5}}$ if $l\in D^c$,
where $\alpha_{L}^{{\rm{cmb}}}$ is the coefficient vector with maximum degree $L$ of the CMB 2018 map computed by the HEALPy package \cite{healpix}.

We assume that the noise $\Delta$ also has the K-L expansion and set $\varrho=\|\Delta\|_{L_{2}(\mathbb{S}^2)}^2$. We generate the data for the noise on the HEALPix points with $N_{\rm{side}}=2048$  by the MATLAB command: $\delta {\rm{randn}}(N_{\rm{pix}},1)$, where $N_{\rm{pix}}=12\times N_{\rm{side}}^2$ and $\delta>0$ is a scaling parameter. Then we use the Python HEALPy package to compute the coefficients of the noise.
The masks denoted by $\Gamma^c=\mathbb{S}^{2}{{\setminus}}\Gamma$ are shown in Figure \ref{masks}.
\begin{figure}[tbh]
\begin{center}
  \subfigure[$\Gamma^c_{1}$]{
  \includegraphics[width=2.6cm]{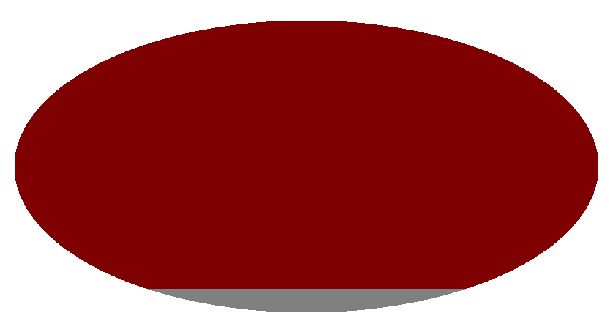}}
  \subfigure[$\Gamma^c_{2}$]{
  \includegraphics[width=2.6cm]{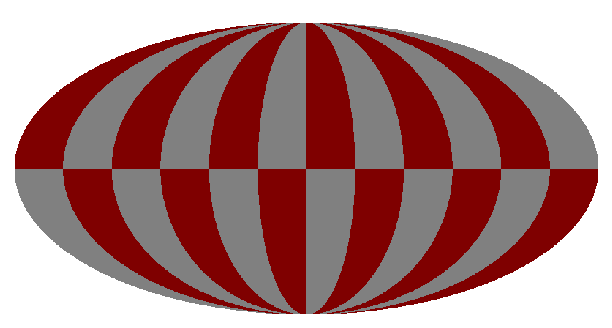}}
  \subfigure[$\Gamma^c_{3}$]{
  \includegraphics[width=2.6cm]{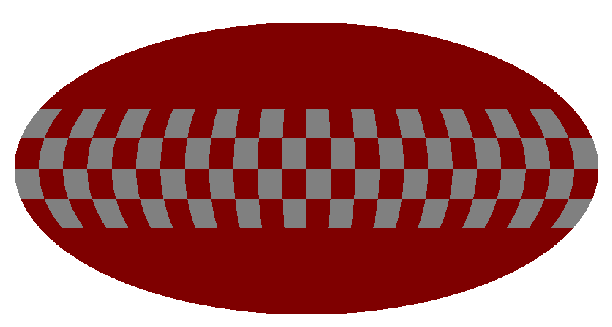}}
  \subfigure[$\Gamma^c_{4}$]{
  \includegraphics[width=2.6cm]{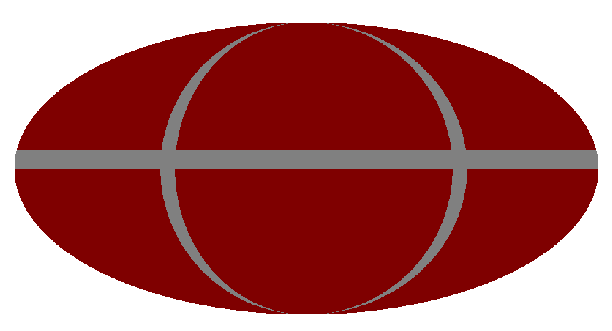}}
  \caption{ Masks (grey part).}\label{masks}
  \end{center}
\end{figure}

In the experiments,
we set $p=0.5$, $\beta_{l}=(1+10^{-4})^{l}l^{p}$ for $l\ge 1$  and $\beta_0=1$.
We compare  the $\ell_{p}$-$\ell_{2}$ optimization model (\ref{fini:cmodel}) by using Algorithm \ref{alg:smoothnpg} with the  $\ell_{1}$ optimization model (31) in \cite{JD} using the YALL1 method \cite{yang} under the KKM sampling scheme \cite{kh2014}.
We also compare our method with the  group SPGl1 method (https://friedlander.io/spgl1/).
We present the results in Tables \ref{table:inp1} and \ref{table:inp2}. Following \cite{chen2018},
${\rm{nnz}}:=\|\alpha_{L}^* \,\verb"&" \,\alpha^{{\rm{true}}}_{L}\|_{2,0}$  denotes the number of nonzero groups that $\alpha_{L}^*$ and $\alpha_{L}^{{\rm{true}}}$ have in common and ${\rm{false}}:=\|\alpha_{L}^* \|_{2,0}-\|\alpha_{L}^*\, \verb"&"\,  \alpha^{{\rm{true}}}_{L}\|_{2,0}$ denotes the number of ``false positives" where
$\alpha^{{\rm{true}}}_{l\cdot}$ is a zero vector, but $\alpha_{l\cdot}^* $ is a nonzero vector. The  signal-to-noise ratio (SNR) and relative error are defined by
\begin{equation*}\label{relerr}
 {\rm{SNR}}=20\times\log_{10}\tfrac{\|x^{\rm{true}}\|}{\|x^{\rm{true}}-x^*\|}, \quad {\rm{RelErr}}:=\tfrac{\|\alpha_{L}^*-\alpha_{L}^{{\rm{true}}}\|}{\|\alpha_{L}^{{\rm{true}}}\|} ,
\end{equation*}
respectively,
where $x^{\rm{true}}=(T^{\rm{true}}(\mathrm{x}_{1}),\ldots,T^{\rm{true}}(\mathrm{x}_{n}))^T$ and $x^*=(T^*(\mathrm{x}_{1}),\ldots,T^*(\mathrm{x}_{n}))^T$ are estimated on the HEALPix points with $n=12\times 2048^2$ and  $\alpha_{L}^*$ is the terminating solution,.

From Tables \ref{table:inp1} and \ref{table:inp2}, we can see that our optimization model (\ref{fini:cmodel}) using Algorithm \ref{alg:smoothnpg}  achieves smaller relative errors and higher SNR values than the  $\ell_{1}$ optimization model (31) in \cite{JD}.
Although the group SPGl1 method archives small relative errors for some experiments, the SNR values is smaller than ours.
Moreover, compared with the $\ell_{1}$ optimization model {and group SPGl1}, our  optimization model (\ref{fini:cmodel}) can exactly recover the number and positions of nonzero groups of $\alpha_{L}^*$.

We select some numerical results from Table \ref{table:inp1} with  $L=50$ and $\|\alpha_{L}^{{\rm{true}}}\|_{2,0}=26$   to show the inpainting quality in Figure \ref{example}. We observe that our model using Algorithm \ref{alg:smoothnpg} achieves smaller pointwise errors than the $\ell_{1}$ optimization model (31) in \cite{JD} and group SPGl1 method.

\begin{table} [H]
\scriptsize
\begin{center}
\caption{Numerical results for the inpainting of band-limited random fields with $\delta=1$.}\label{table:inp1}
\setlength{\tabcolsep}{0.6mm}{
\begin{tabular}{|cc|c|ccc|r|c|ccc|r|c|ccc|r|}
  \hline
          &   &  \multicolumn{5}{c|}{ Algorithm \ref{alg:smoothnpg}} & \multicolumn{5}{c|}{ YALL1}& \multicolumn{5}{c|}{ {Group SPGl1}}\\ \cline{3-17}
  Mask & $\|\alpha^{{\rm{true}}}_{L}\|_{2,0}$ &
  RelErr  &    $\|\alpha^{*}_{L}\|_{2,0}$ &${\rm{nnz}}$ & ${\rm{false}}$   &SNR &
  RelErr  &    $\|\alpha^{*}_{L}\|_{2,0}$ &${\rm{nnz}}$ & ${\rm{false}}$  &SNR &
  RelErr  &    $\|\alpha^{*}_{L}\|_{2,0}$ &${\rm{nnz}}$ & ${\rm{false}}$  &SNR\\ \hline
     \multicolumn{17}{|c|}{$L=35$} \\ \hline
  $\Gamma_{1}^c$    &  7 &  0.0017 & 7   & 7 & 0  & 55.54 & 0.511 & 14& 6 & 8 &5.82 & 0.0023 & 12 & 7 & 5 & 52.65\\
                  & 11 &  0.0012 &  11 & 11& 0  & 58.28 & 0.575 & 13& 9 & 4 &4.80 & 0.0788 &  23& 11 & 12 &56.19\\
                  & 19 &  0.0013 &  19 & 19& 0  & 57.73 & 0.374 & 25& 17& 8 &8.55 & 0.3965 & 29 & 19 & 10 &46.25\\\hline
  $\Gamma_{2}^c$    & 7  &  0.0023 &  7  & 7 & 0  & 52.47 & 0.668 & 19& 7 & 12&3.08 & 0.0043 & 17 & 7 & 10 & 47.29\\
                  & 11 &  0.0022 &  11 & 11& 0  & 53.19 & 0.651 & 25& 11& 14&2.67 & 0.0788 & 32 & 11 & 21 &22.07\\
                  & 19 &  0.0059 &  19 & 19& 0  & 44.54 & 0.683 & 26& 19& 7 &1.87 & 0.3965 & 33 & 19 & 14 & 8.03\\
  \hline
  $\Gamma^c_{3}$    & 7  &  0.0019 &  7  & 7 & 0  & 54.60 & 0.378 & 21& 7 & 14&8.50 & 0.0035 & 16 & 7 & 9 &49.01\\
                  & 11 &  0.0014 &  11 & 11& 0  & 57.30 & 0.357 & 13& 11& 2 &8.89 & 0.0024 & 26 & 11 & 15 &52.57\\
                  & 19 &  0.0015 &  19 & 19& 0  & 56.28 & 0.321 & 25& 19& 6 &9.70 & 0.1195 & 34 & 19 & 15& 18.45\\
  \hline
  $\Gamma^c_{4}$    & 7  &  0.0017 &  7  & 7 & 0  & 55.33 & 0.351 & 16& 6 & 10&8.91 & 0.0022 & 12 & 7 & 5 &52.65\\
                  & 11 &  0.0013 &  11 & 11& 0  & 57.72 & 0.321 & 11& 9 & 2 &10.78& 0.0019 & 23 & 11 & 12 &54.29\\
                  & 19 &  0.0014 &  19 & 19& 0  & 56.84 & 0.273 & 22& 17 & 5 &11.25& 0.0073 & 32 & 19 & 13 &42.76\\
  \hline
     \multicolumn{17}{|c|}{$L=50$} \\ \hline
  $\Gamma^c_{1}$    &  6 &  0.0027 & 6   & 6 & 0  & 51.43 & 0.551 & 28& 5 & 23 &5.18 &  0.0030 & 13 & 6 & 7 & 50.40\\
                  & 16 &  0.0024 &  16 & 16& 0  & 52.43 & 0.591 & 33& 13& 20 &4.57 &  0.0024 & 33 & 16 & 17 & 52.29\\
                  & 26 &  0.0031 &  26 & 26& 0  & 50.21 & 0.511 & 35& 23& 12 &5.84 &  0.0038 & 37 & 26 & 11 & 48.50\\\hline
  $\Gamma_{2}^c$    & 6  &  0.0032 &  6  & 6 & 0  & 49.84 & 0.676 & 39& 6 & 33 &2.83 &  0.0044 & 15 & 6 & 9 &47.04\\
                  & 16 &  0.0044 &  16 & 16& 0  & 47.15 & 0.643 & 43& 16& 27 & 2.68 & 0.0729 & 45 & 16 & 29 &22.74\\
                  & 26 & 0.0090  & 26  &26 & 0  & 40.93 & 0.664 & 42& 26& 16 &2.73 &  0.3549 & 47 & 26 & 21 &8.99\\
  \hline
  $\Gamma_{3}^c$    & 6  &  0.0027 &  6  & 6 & 0  & 51.43 & 0.212 & 41& 6 & 35 &13.21 & 0.0036 & 14 & 6 & 8 &48.89\\
                  & 16 &  0.0024 &  16 & 16& 0  & 52.43 & 0.191 & 37& 16& 21 &14.61 &  0.0041 & 38 & 16 & 22 &47.72\\
                  & 26 &  0.0031 &  26 & 26& 0  & 50.21 & 0.298 & 38& 26& 12 &10.50 & 0.1393 & 48 & 26 & 22 &17.11\\
  \hline
  $\Gamma_{4}^c$    & 6  &  0.0027 &  6  & 6 & 0  & 51.25 & 0.197 & 37& 6 & 31 &16.42 & 0.0031 & 14 & 6 & 8 &50.14\\
                  & {{16}} &  0.0023 &  16 & 16& 0  & 52.47 & 0.193 & 37& 16& 21 &14.21 &  0.0030 & 32 & 16 & 16 &50.51\\
                  & 26 &  0.0030 &  26 & 26& 0  & 50.55 & 0.239 & 37& 25& 12 &13.05 & 0.0078 & 44 & 26 & 18 &42.21\\
  \hline
  \end{tabular}
  }
\end{center}
\end{table}

\begin{table} [H]
\scriptsize
\begin{center}
\caption{Numerical results for the inpainting of band-limited random fields with $\delta=0.1$.}\label{table:inp2}
\setlength{\tabcolsep}{0.6mm}{
\begin{tabular}{|cc|c|ccc|r|c|ccc|r|c|ccc|r|}
  \hline
          &   &  \multicolumn{5}{c|}{ Algorithm \ref{alg:smoothnpg}} & \multicolumn{5}{c|}{ YALL1}& \multicolumn{5}{c|}{ {Group SPGl1}}\\ \cline{3-17}
  Mask & $\|\alpha^{{\rm{true}}}_{L}\|_{2,0}$ &
  RelErr  &    $\|\alpha^{*}_{L}\|_{2,0}$ &${\rm{nnz}}$ & ${\rm{false}}$   &SNR &
  RelErr  &    $\|\alpha^{*}_{L}\|_{2,0}$ &${\rm{nnz}}$ & ${\rm{false}}$  &SNR &
   RelErr  &    $\|\alpha^{*}_{L}\|_{2,0}$ &${\rm{nnz}}$ & ${\rm{false}}$  &SNR\\ \hline
      \multicolumn{17}{|c|}{$L=35$} \\ \hline
  $\Gamma^c_{1}$    &  7 &  1.90e-4 & 7   & 7 & 0  & 74.42 & 0.551 & 6 &  6& 0 &5.83 &2.88e-4 & 11 & 7 & 4 & 70.78\\
                  & 11 &  1.27e-4 &  11 & 11& 0  & 77.89 & 0.591 & 12& 9 & 3 &4.80 &1.61e-4 &21 &11 &10 & 75.82\\
                  & 19 &  1.39e-4 &  19 & 19& 0  & 77.13 & 0.511 & 25& 17& 8 &8.55 & 0.0035 &33 &19 & 14& 49.18\\\hline
  $\Gamma^c_{2}$    & 7  &  2.58e-4 &  7  & 7 & 0  & 71.73 & 0.676 & 19& 19& 0 &3.08 & 5.24e-4& 17& 7& 10& 65.60\\
                  & 11 &  2.35e-4 &  11 & 11& 0  & 72.57 & 0.643 & 21& 11& 10&2.67 & 0.0859 &34 & 11&23 & 21.32\\
                  & 19 &  6.89e-4 &  19 & 19& 0  & 63.22 & 0.664 & 10&  7&  3&1.87 &0.3949  & 33& 19&14 &8.06 \\
  \hline
  $\Gamma^c_{3}$    & 7  &  2.08e-4 &  7  & 7 & 0  & 73.60 & 0.212 & 16&  7& 9 &8.50 &4.05e-4 & 15&7 &8 & 67.84\\
                  & 11 &  1.40e-4 &  11 & 11& 0  & 77.05 & 0.191 & 14& 11& 3 &8.89 &2.54e-4 &25 &11 &14 & 71.89\\
                  & 19 &  1.70e-4 &  19 & 19& 0  & 75.37 & 0.298 & 21& 19& 2 &9.70 &0.1172  &35 & 19& 16&18.62 \\
  \hline
  $\Gamma^c_{4}$    & 7  &  1.97e-4 &  7  & 7 & 0  & 74.09 & 0.197 & 11& 6 & 5 &8.91 &2.65e-4 &11 & 7& 4& 71.50\\
                  & 11 &  1.28e-4 &  11 & 11& 0  & 77.83 & 0.193 & 9 & 9 & 0 &10.78 &1.93e-4 &22 &11 &11 & 74.25\\
                  & 19 &  1.56e-4 &  19 & 19& 0  & 76.13 & 0.239 & 20&17 & 3 &11.26 &0.0046 &35 &19 & 16&46.66 \\
  \hline
     \multicolumn{17}{|c|}{$L=50$} \\ \hline
  $\Gamma^c_{1}$    &  6 &  2.56e-4 & 6   & 6 & 0 & 71.81 & 0.551 &13 & 5 &8  &5.18 &3.10e-4 &11 &6 &5 & 70.17\\
                  & 16 &  2.15e-4 &  16 & 16& 0 & 73.32 & 0.591 &18 & 13& 5 &4.57 &2.67e-4 & 27&16 &11 & 71.44\\
                  & 26 &  2.46e-4 &  26 & 26& 0 & 72.17 & 0.511 &35 & 22 &12 &5.84 &3.94e-4 &35 &26 & 9& 68.08\\\hline
  $\Gamma^c_{2}$    & 6  &  3.18e-4 & 6   & 6 & 0 & 69.92 & 0.676 &16 &6  &10  &2.83 &4.48e-4 &16 & 6& 10&  66.96\\
                  & 16 &  5.09e-4 & 16  & 16& 0 & 65.85 & 0.643 &36 &16 &20  &2.68 &0.0829 & 45&16 &29 & 21.63\\
                  & 26 &  9.97e-4 & 26   & 26& 0 & 60.01 & 0.664 &32 &26 & 6 &2.73 &0.3562 &47 &26 &21 & 8.98\\
  \hline
  $\Gamma^c_{3}$    & 6  &  2.97e-4 &  6  & 6 & 0  & 70.52 & 0.212 &30 & 6 &24 &13.21&3.84e-4 & 13&6 &7 & 68.31\\
                  & 16 &  2.70e-4 &  16 & 16& 0  & 71.34 & 0.191 &24 &16 & 8 &14.61&4.59e-4 &34 &16 &18 & 66.75 \\
                  & 26 &  3.28e-4 &  26 & 26& 0  & 69.65 & 0.298 &31 & 26& 5 &10.50&0.1387 &50 & 26&24 &17.15 \\
  \hline
  $\Gamma^c_{4}$    & 6  &  2.86e-4 &  6  & 6 & 0  & 70.85 & 0.197 &23 & 6 & 17 &16.42&3.43e-4 & 11& 6& 5& 69.29\\
                  &16 &  2.56e-4 &  16 & 16& 0  & 71.82 & 0.193 &24 & 16& 8 &14.21 & 3.27e-4&28 &16 & 12&69.69 \\
                  & 26 &  3.09e-4 &  26 & 26& 0  & 70.19 & 0.239 &29 &25 & 4 &13.05& 8.06e-4& 43& 26& 17& 61.86\\
  \hline
  \end{tabular}
  }
\end{center}
\end{table}

\begin{figure}[H]
  \centering
  \subfigure[Observed fields with $L=50$, $\|\alpha^{\rm{true}}_{L}\|_{2,0}=6$ and $\delta=1$ (masked by $\Gamma^c_1$, $\Gamma^c_2$, $\Gamma^c_3$ and $\Gamma^c_4$ from left to right)]{
  \includegraphics[width=13cm]{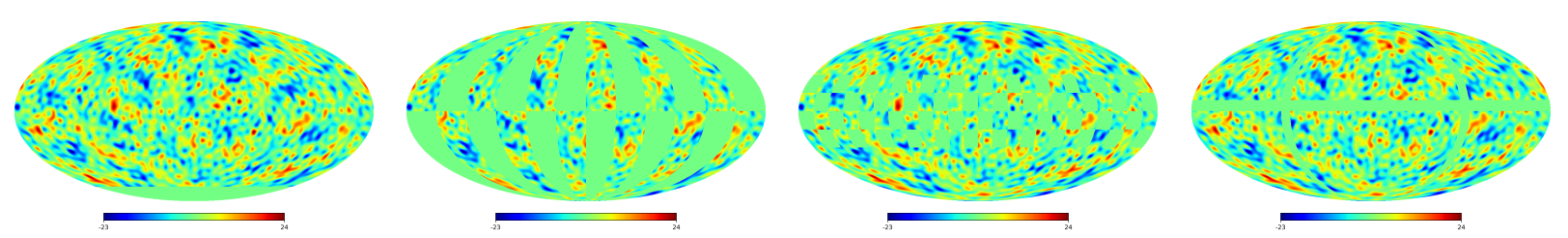}}\\
  \subfigure[Inpainted fields by Algorithm \ref{alg:smoothnpg}]{
  \includegraphics[width=13cm]{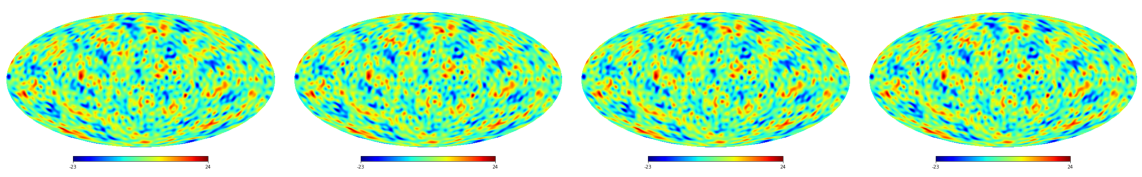}}\\
  \subfigure[Pointwise error of inpainted fields in (b) by Algorithm \ref{alg:smoothnpg}]{
  \includegraphics[width=13cm]{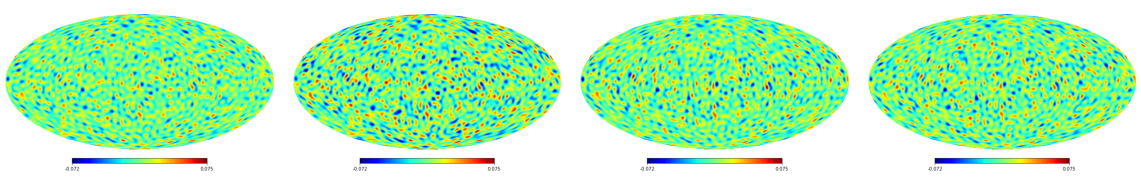}}\\
  \subfigure[Inpainted fields by YALL1]{
  \includegraphics[width=13cm]{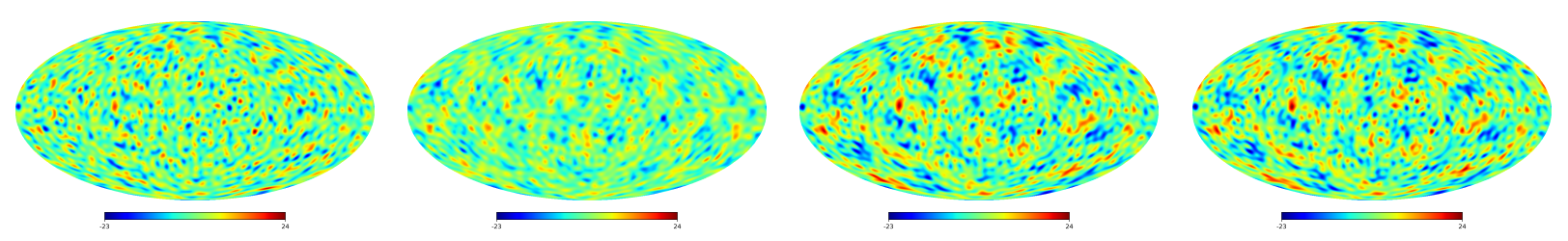}}\\
  \subfigure[Pointwise error of inpainted fields in (d) by YALL1]{
  \includegraphics[width=13cm]{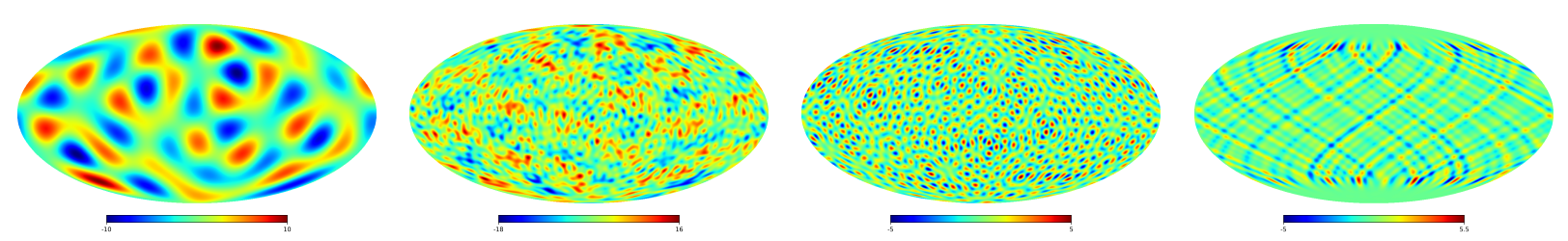}}\\
  \subfigure[{Inpainted fields by group SPGl1}]{
  \includegraphics[width=13cm]{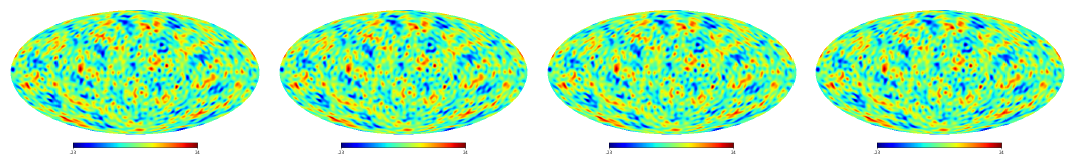}}\\
  \subfigure[{Pointwise error of inpainted fields in (f) by group SPGl1}]{
  \includegraphics[width=13cm]{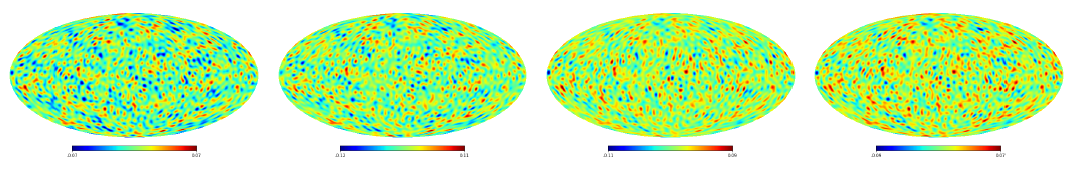}}\\
  \caption{Inpainting results.}\label{example}
\end{figure}

To give some insight for the choice of $L$ of the truncated space, we conduct experiments on a random field with different $L$ under the noiseless case.
We choose the mask $\Gamma_4$ and  the observed field $T^{\circ}$ with degree $50$.
 In Figure \ref{error}(a) we show the approximation error $\|\mathcal{A}(T_{L}^{*})-T^\circ\|^{2}_{L_{2}( \mathbb{S}^2)}$ on a logarithmic scale.
We can see that the error decreases as $L$ increases and the errors are same for $L=48,49,50$. Moreover, $\alpha^*_{49\cdot}$ and $\alpha^*_{50\cdot}$ are zero groups of $\alpha_L^*$ for $L=50$.
From our discussion in Section 2 and Theorem 5.2, we can guess that the value of $L$ for (a) is $48$.
We plot the error $\|T_{50}^{*}-T_L^*\|^{2}_{L_{2}( \mathbb{S}^2)}$ for different $L$ in Figure \ref{error}(b).
We can observe that the error decreases as $L$ increases and the errors  are zero for $L=48,49,50$ due to $\alpha^*_{49\cdot}=0$ and $\alpha^*_{50\cdot}=0$.
In Figure \ref{error} (c) and (d), we plot the values of $\|\alpha_{l\cdot}^*\|$ of $T^*_L$  with different $ L$ to give some insight of the sparsity pattern of $\alpha^*$.
\begin{figure}[H]
  \centering
  \subfigure[$\|\mathcal{A}(T_{L}^*)-T^{\circ}\|^2_{L_{2}(\mathbb{S}^2)}$]{
  \includegraphics[width=7cm]{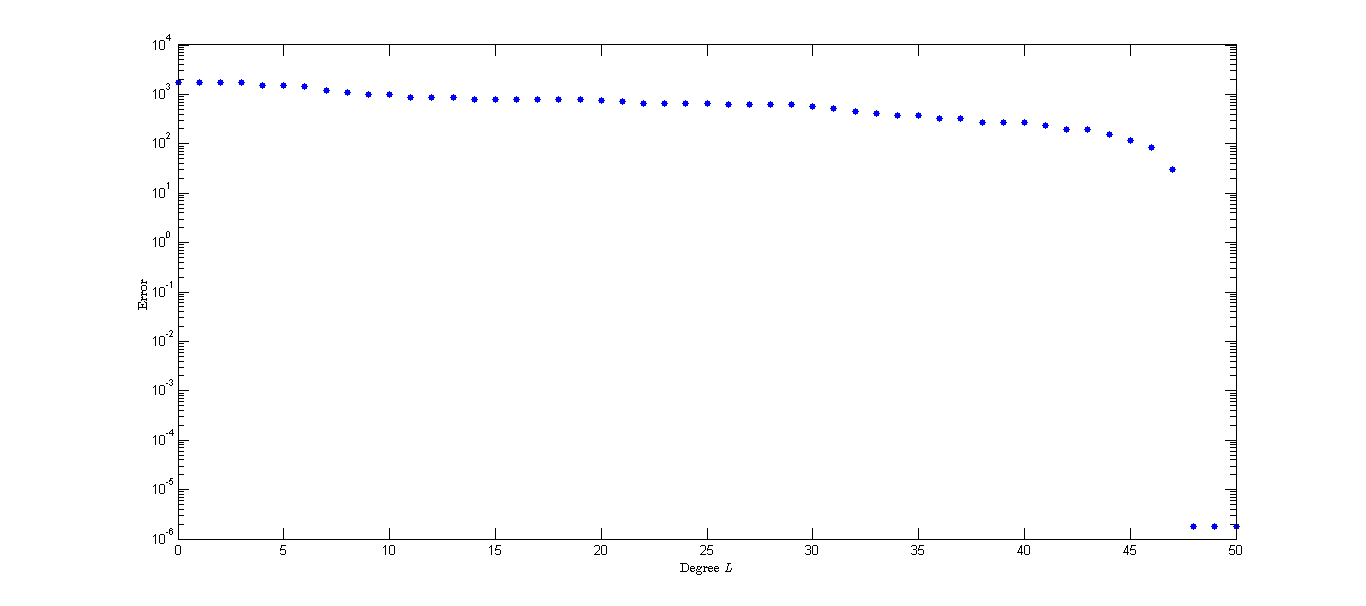}}
  \subfigure[$\|T_{50}^*-T^{*}_L\|^2_{L_{2}(\mathbb{S}^2)}$]{
  \includegraphics[width=7cm]{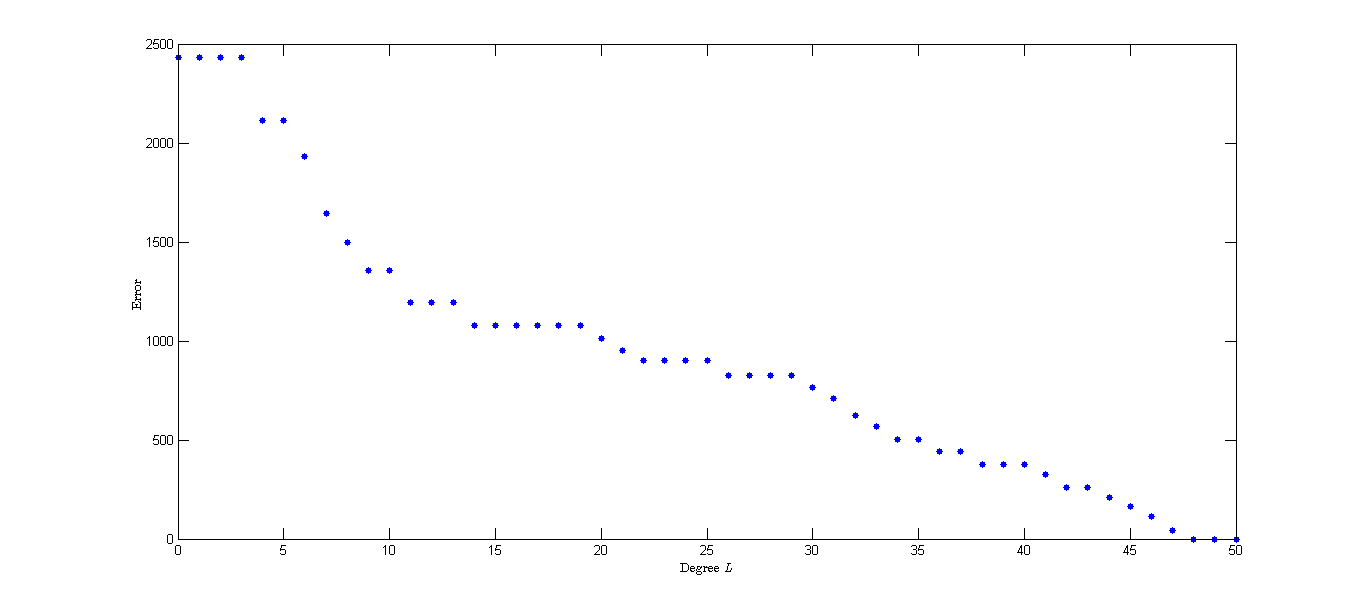}}\\
  \subfigure[$\|\alpha_{l\cdot}^*\|$]{
  \includegraphics[width=7cm]{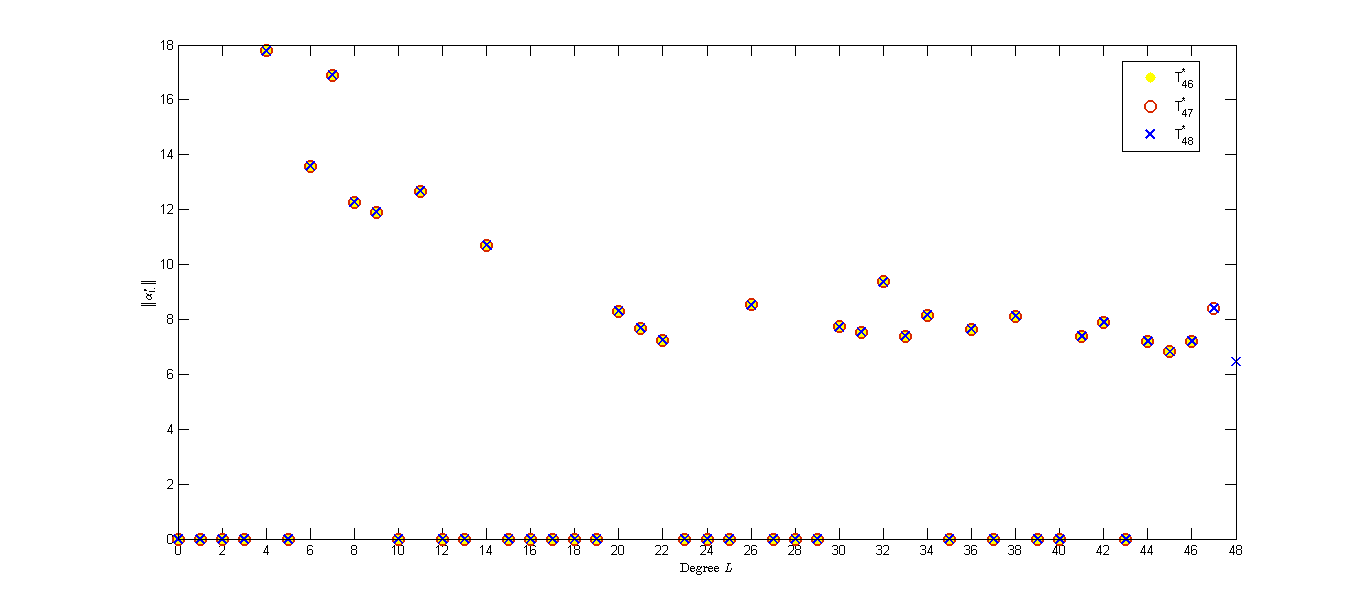}}
  \subfigure[$\|\alpha_{l\cdot}^*\|$]{
  \includegraphics[width=7cm]{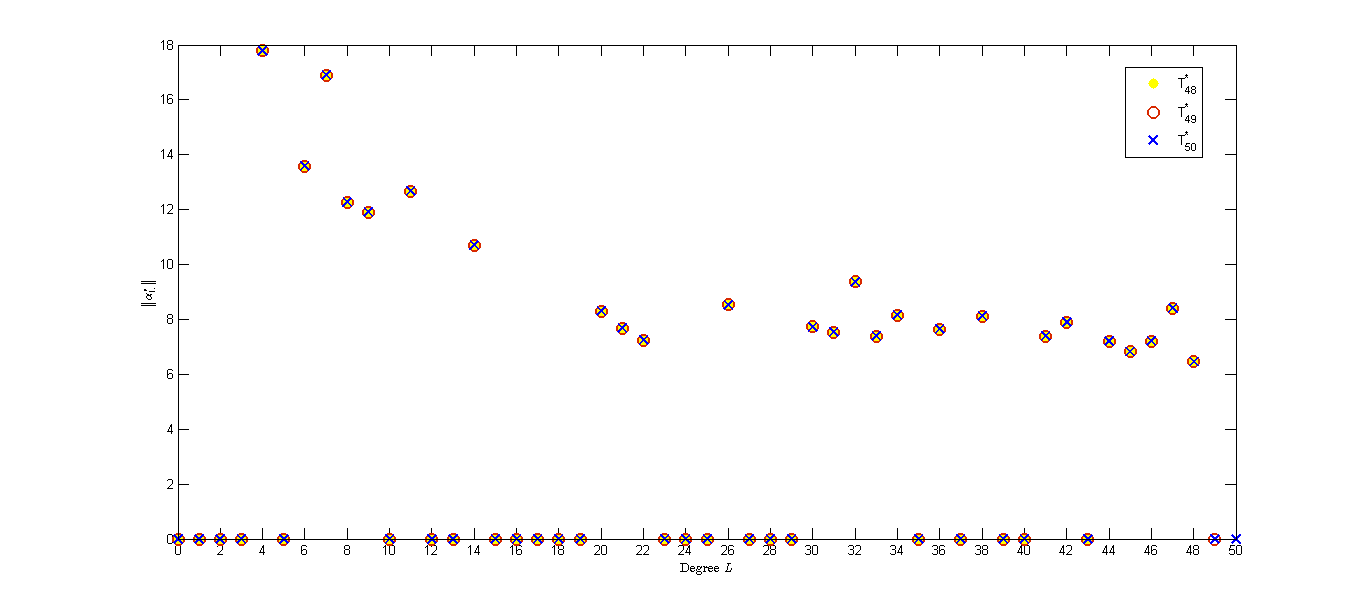}}
  \caption{{Approximation errors and sparsity pattern.}}\label{error}
\end{figure}
\subsection{Real image}
We consider the inpainting of images from Earth topographic data, taken from Earth Gravitational Model (EGM2008)
publicly released by the U.S. National Geospatial-Intelligence
Agency (NGA) EGM Development Team\footnote{http://geoweb.princeton.edu/people/simons/DOTM/Earth.mat}.
We generate four resolution test images from the
Earth topography data which are shown in Figures \ref{earth35}(a), \ref{earth50}(a), \ref{earth80}(a) and \ref{earth128}(a).
For each resolution test image we consider a mask which is shown in Figures \ref{earth35}(d), \ref{earth50}(d), \ref{earth80}(d), and \ref{earth128}(d) respectively. For all the images we set $\delta=1$. The recovered images and pointwise errors are presented in Figures \ref{earth35}, \ref{earth50}, \ref{earth80} and \ref{earth128}. We can see that model (\ref{fini:cmodel}) using Algorithm \ref{alg:smoothnpg} yields higher quality images with smaller pointwise errors and higher SNR values  than the $\ell_{1}$ optimization model (31) in \cite{JD} and the group SPGl1 method under different resolutions.
\begin{figure}[H]
  \centering
  \subfigure[True]{
 \includegraphics[width=3cm]{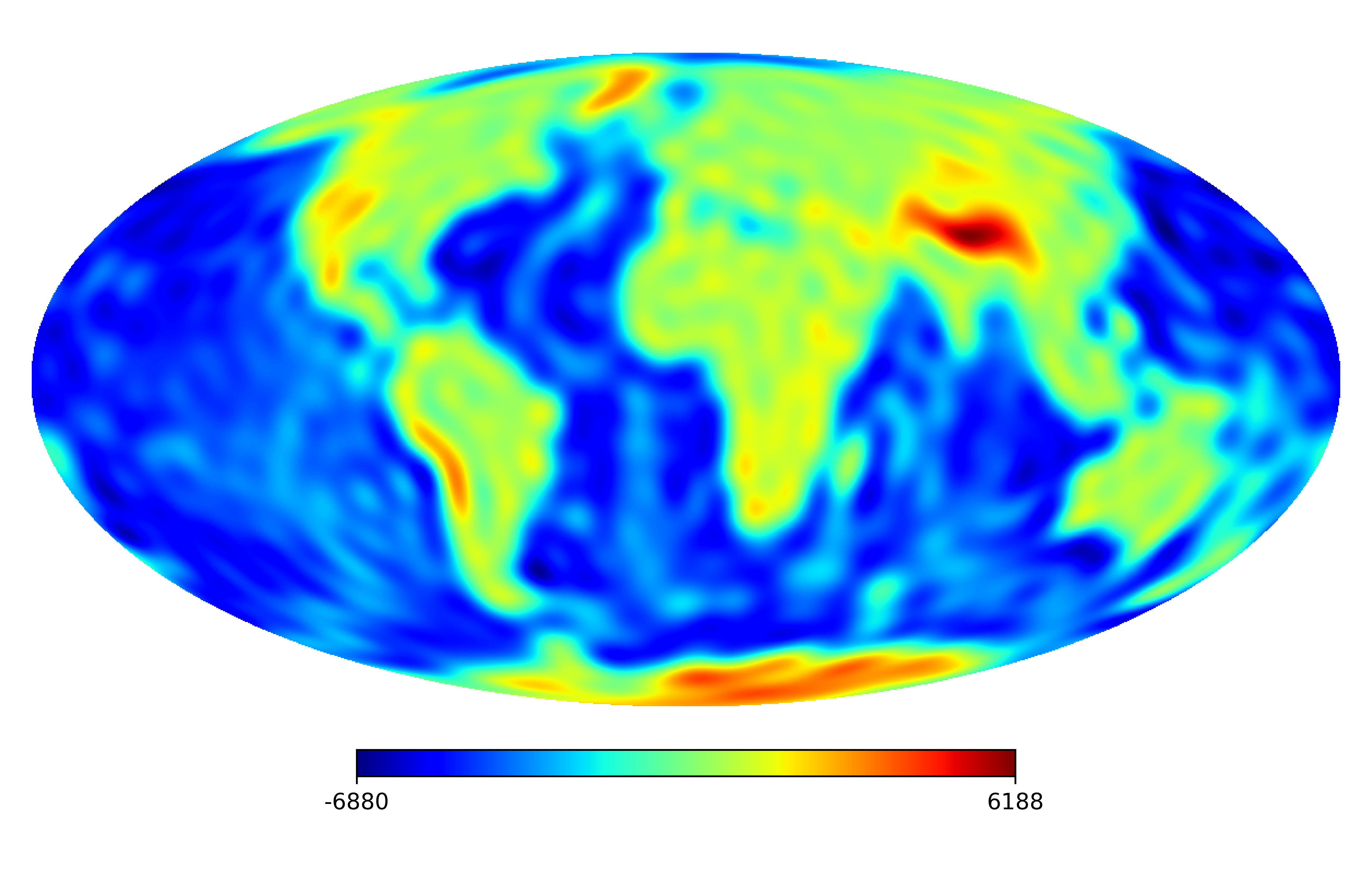}}
  \subfigure[ Algorithm \ref{alg:smoothnpg}]{
  \includegraphics[width=3cm]{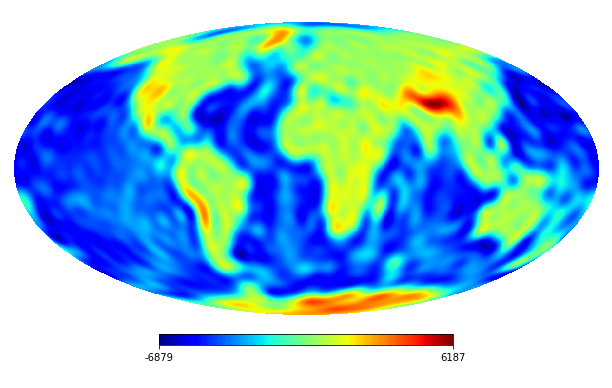}}
  \subfigure[ YALL1]{
  \includegraphics[width=3cm]{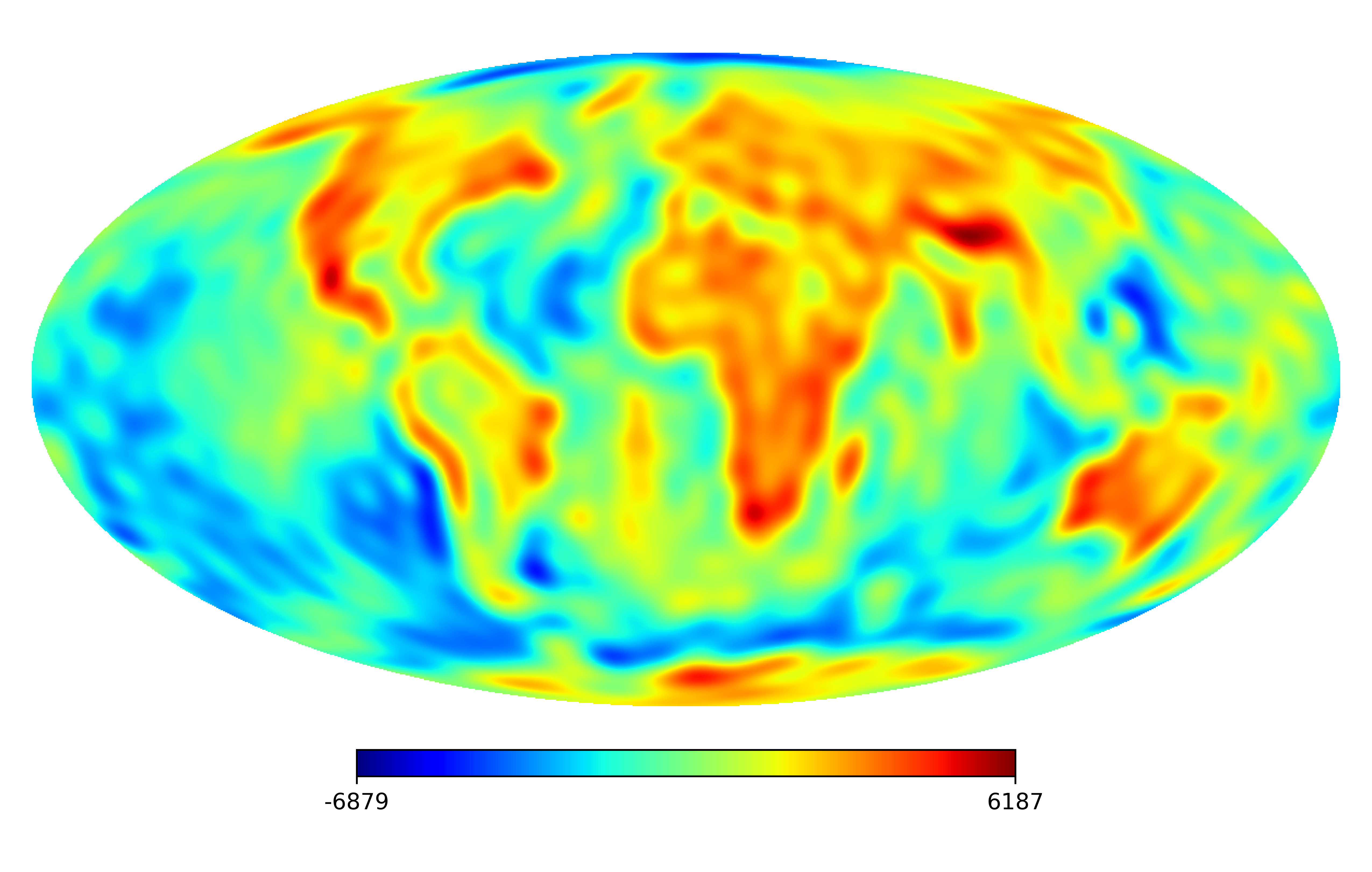}}
  \subfigure[ group SPGl1]{
  \includegraphics[width=3cm]{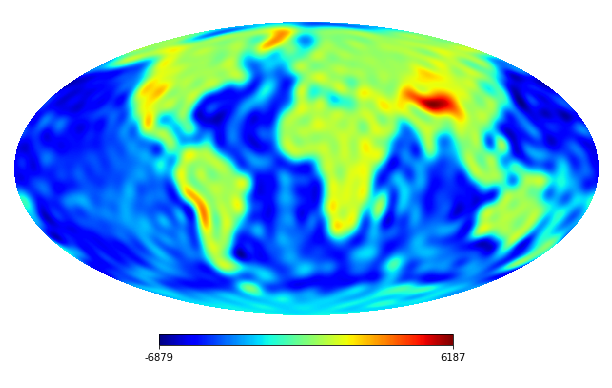}}\\
  \subfigure[Observed]{
  \includegraphics[width=3cm]{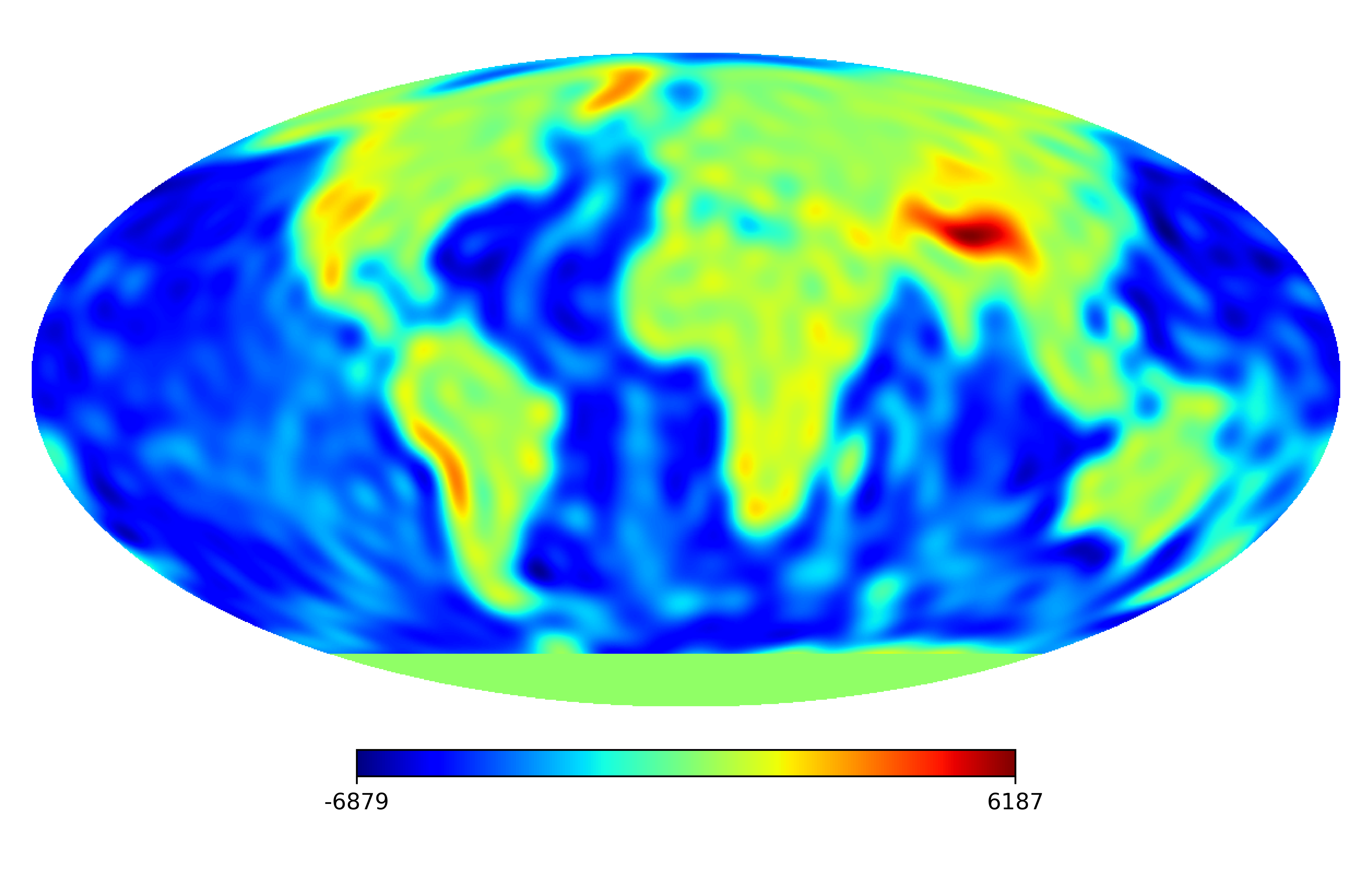}}
  \subfigure[Error of (b)]{
  \includegraphics[width=3cm]{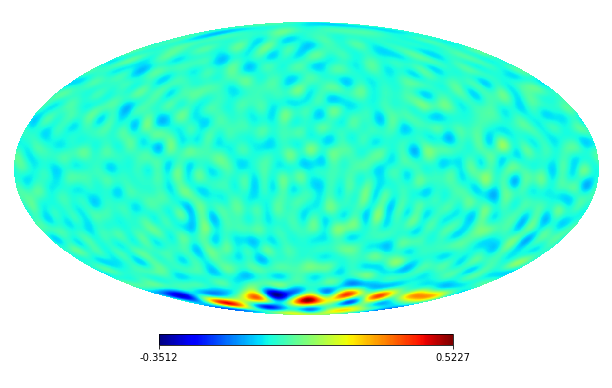}}
  \subfigure[Error of (c)]{
  \includegraphics[width=3cm]{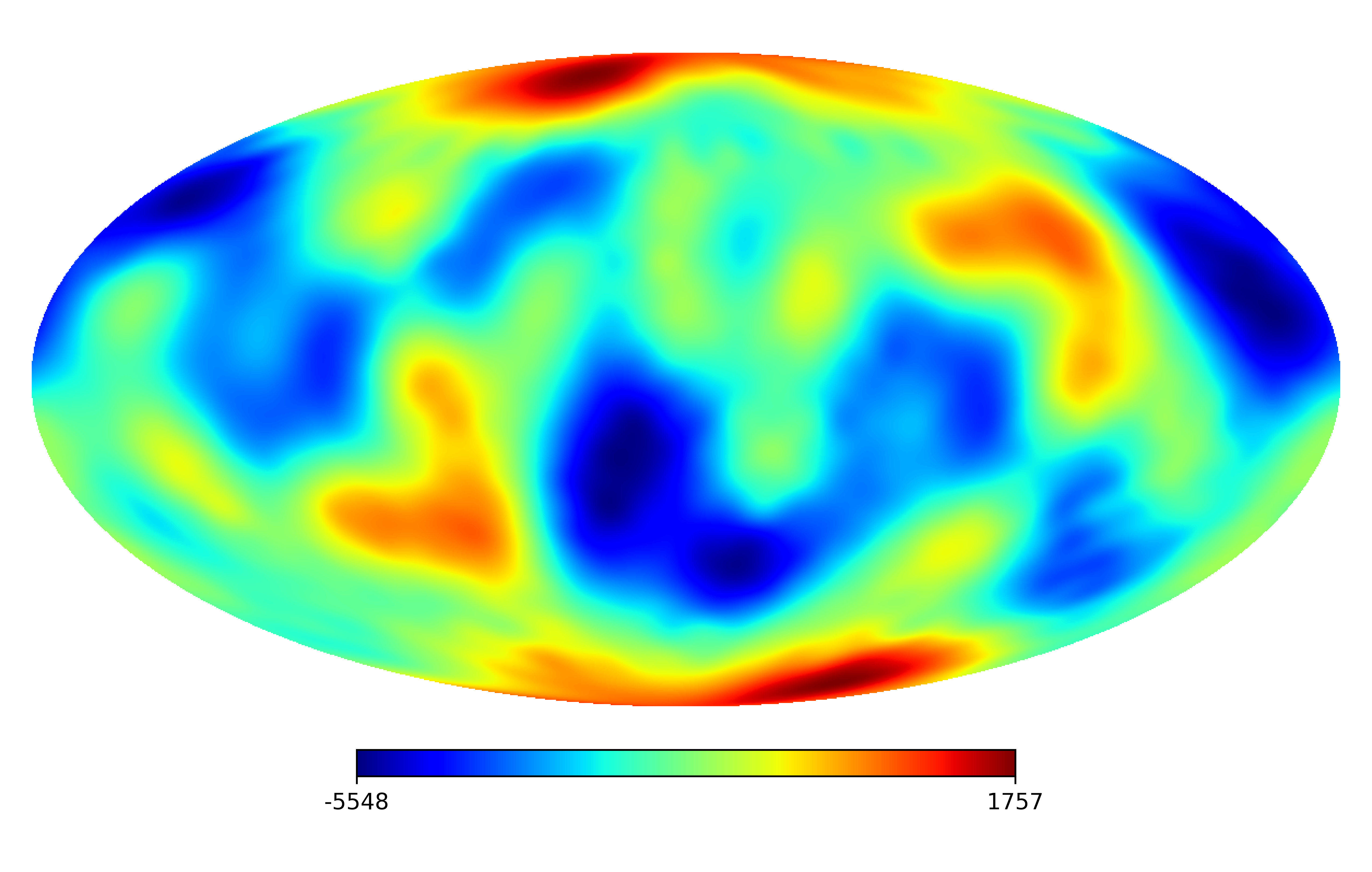}}
  \subfigure[Error of (d)]{
  \includegraphics[width=3cm]{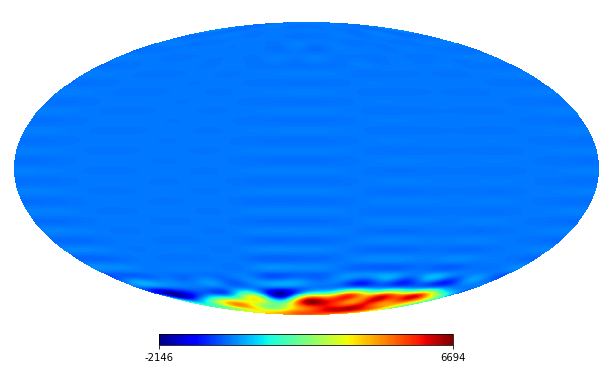}}
  \caption{ Test images of Earth topographic data constructed to be band-limited at $L=35$. The SNR of (b) is 97.30 and the SNR of (c) is 1.81 {and the SNR of (d) is 13.04}.}\label{earth35}
\end{figure}

\begin{figure}[h]
  \centering
  \subfigure[True]{
 \includegraphics[width=3cm]{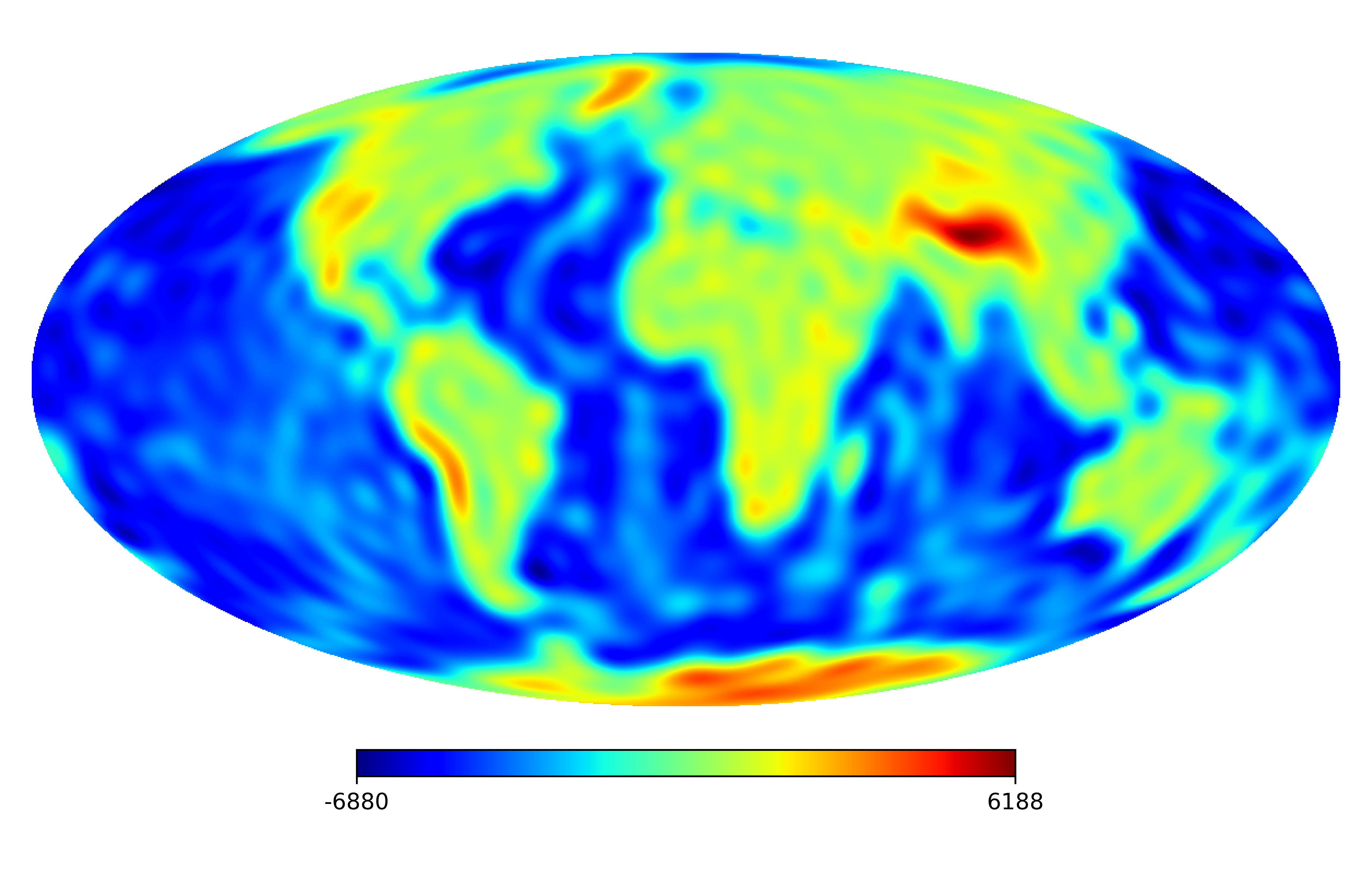}}
  \subfigure[ Algorithm \ref{alg:smoothnpg}]{
  \includegraphics[width=3cm]{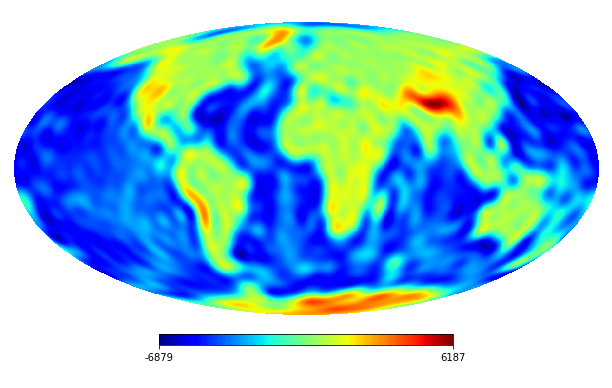}}
  \subfigure[ YALL1]{
  \includegraphics[width=3cm]{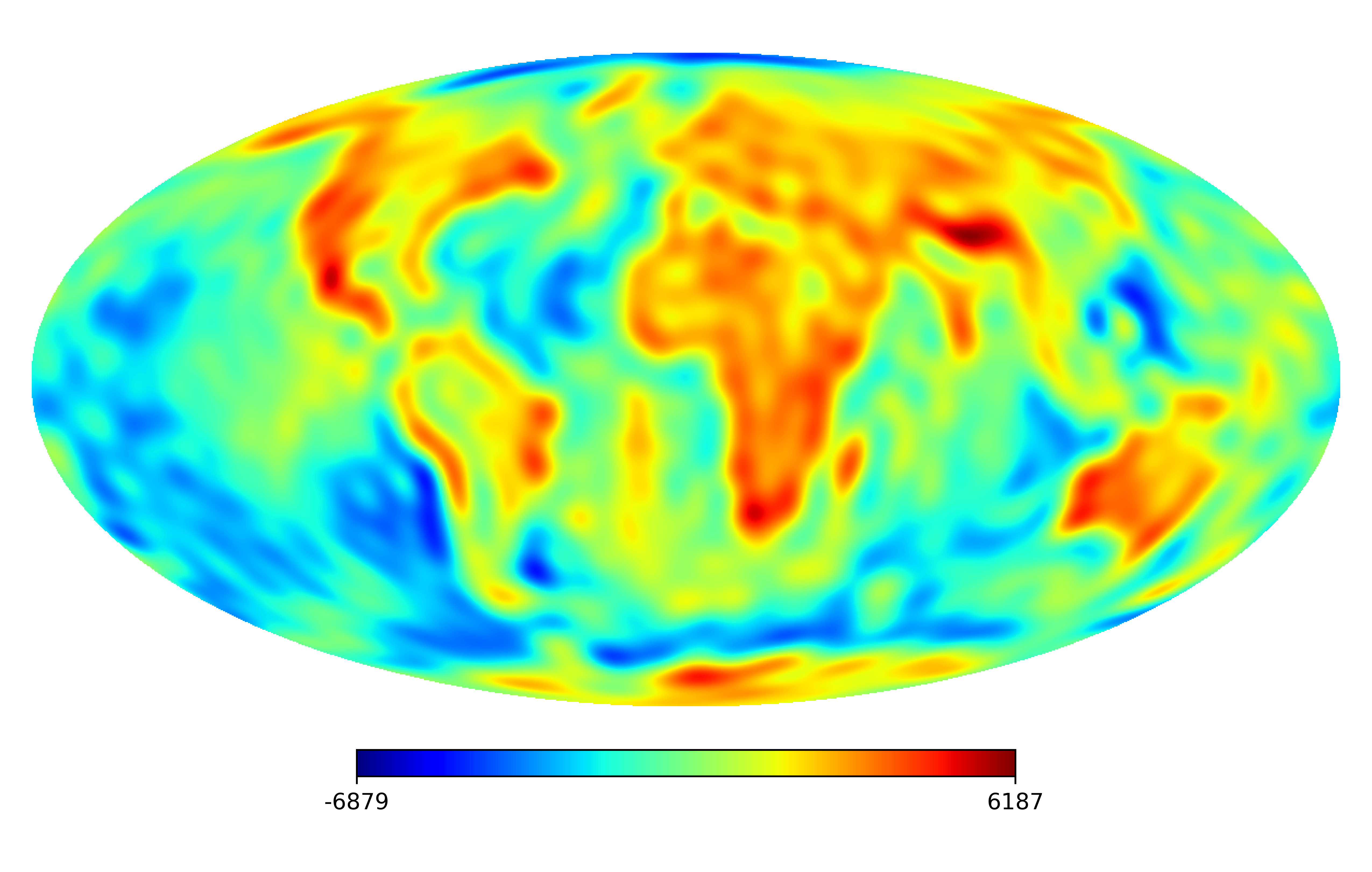}}
  \subfigure[ group SPGl1]{
  \includegraphics[width=3cm]{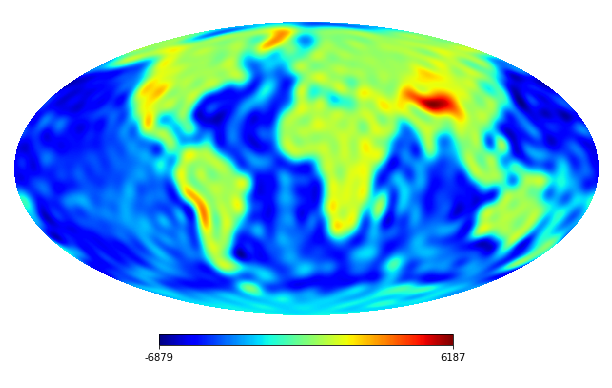}}\\
  \subfigure[Observed]{
  \includegraphics[width=3cm]{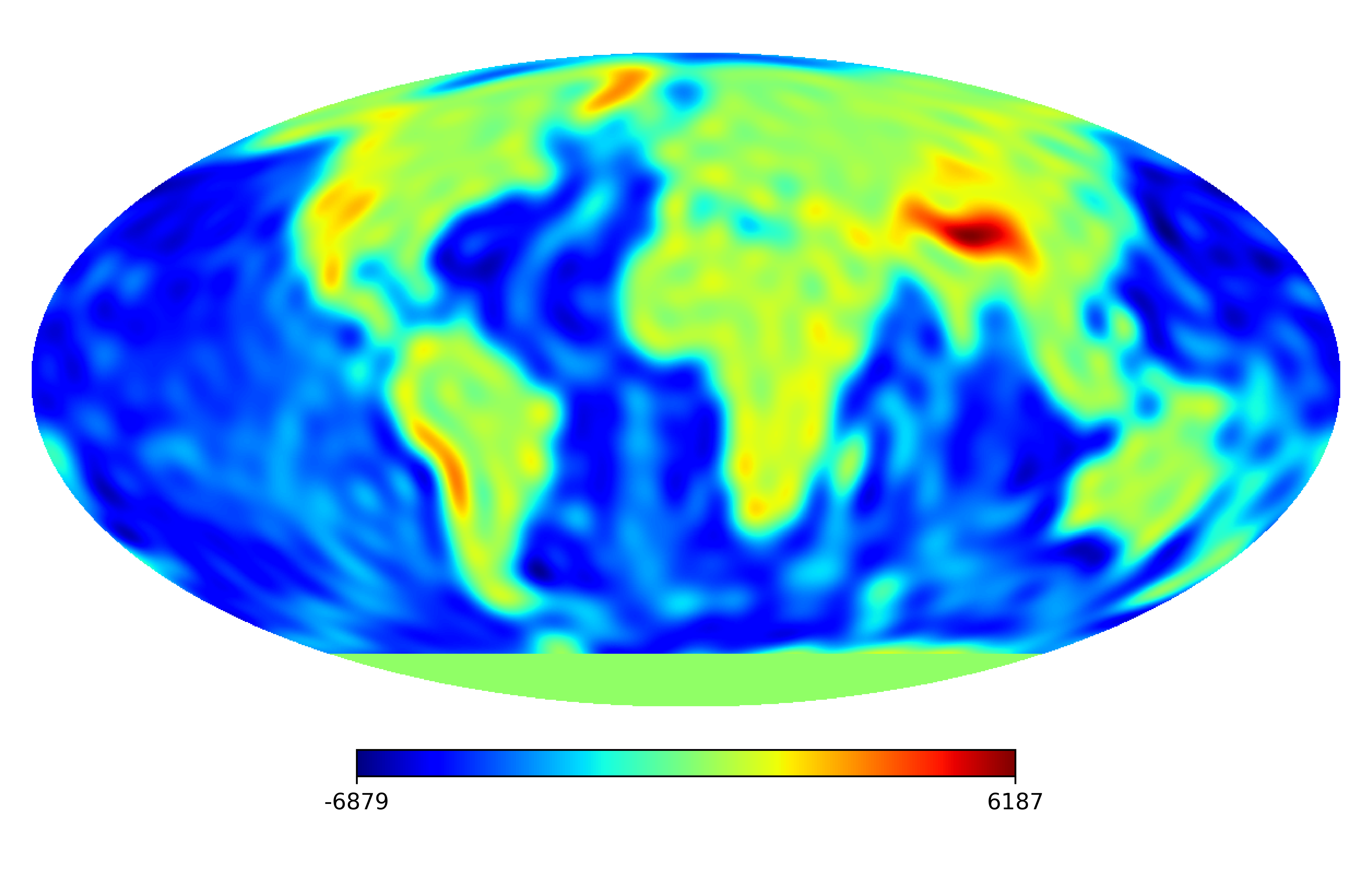}}
  \subfigure[Error of (b)]{
  \includegraphics[width=3cm]{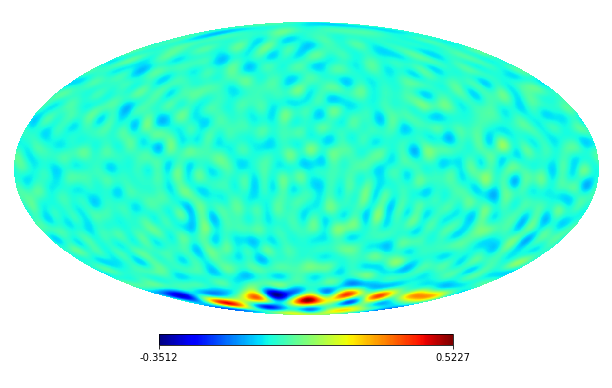}}
  \subfigure[Error of (c)]{
  \includegraphics[width=3cm]{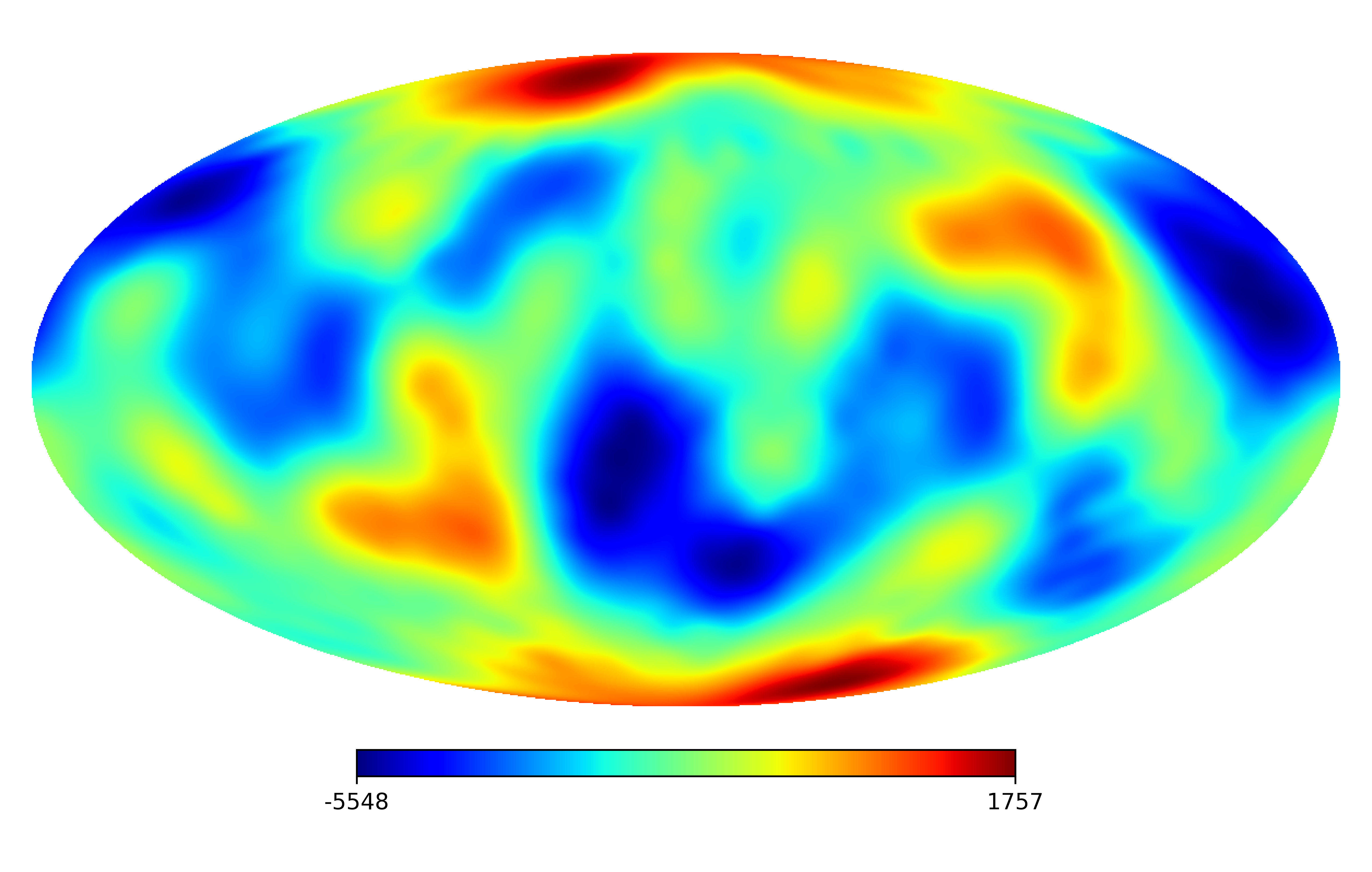}}
  \subfigure[Error of (d)]{
  \includegraphics[width=3cm]{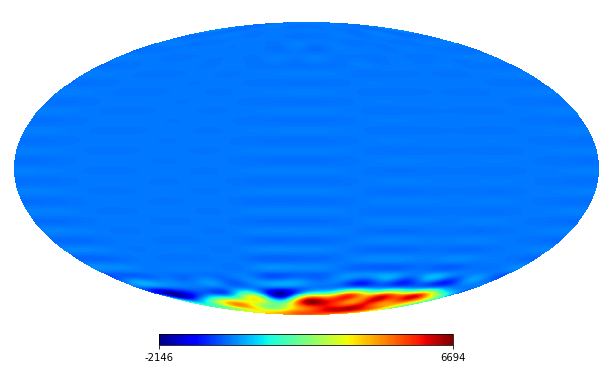}}
  \caption{ Test images of Earth topographic data constructed to be band-limited at $L=35$. The SNR of (b) is 97.30 and the SNR of (c) is 1.81 {and the SNR of (d) is 13.04}.}\label{earth35}
\end{figure}

\begin{figure}[h]
  \centering
  \subfigure[Ture]{
  \includegraphics[width=3cm]{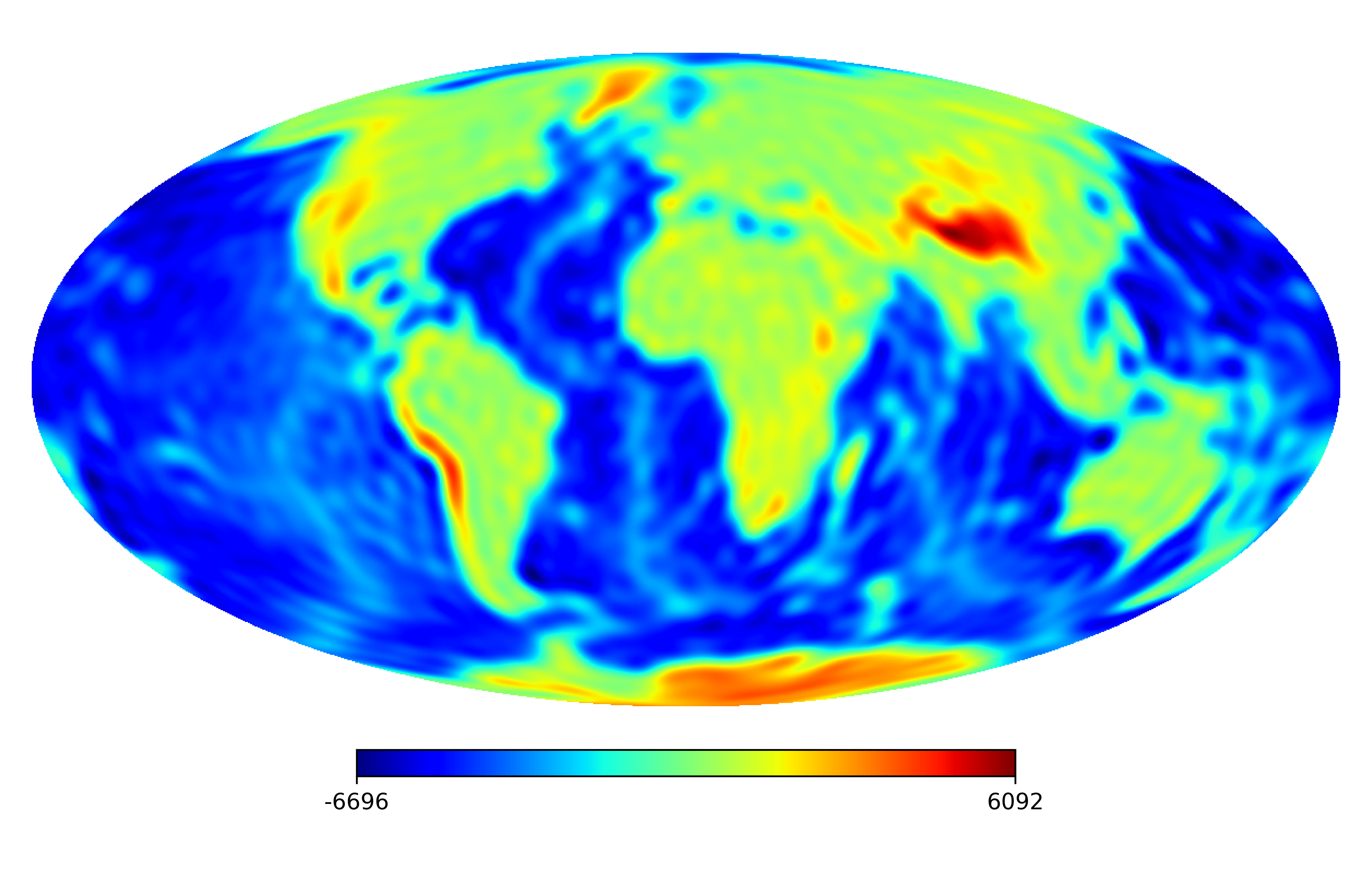}}
  \subfigure[ Algorithm \ref{alg:smoothnpg}]{
  \includegraphics[width=3cm]{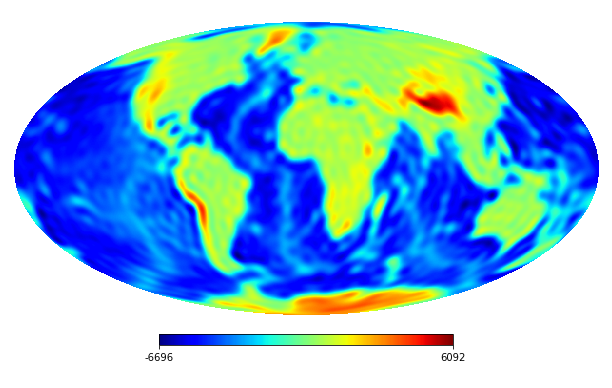}}
  \subfigure[ YALL1]{
  \includegraphics[width=3cm]{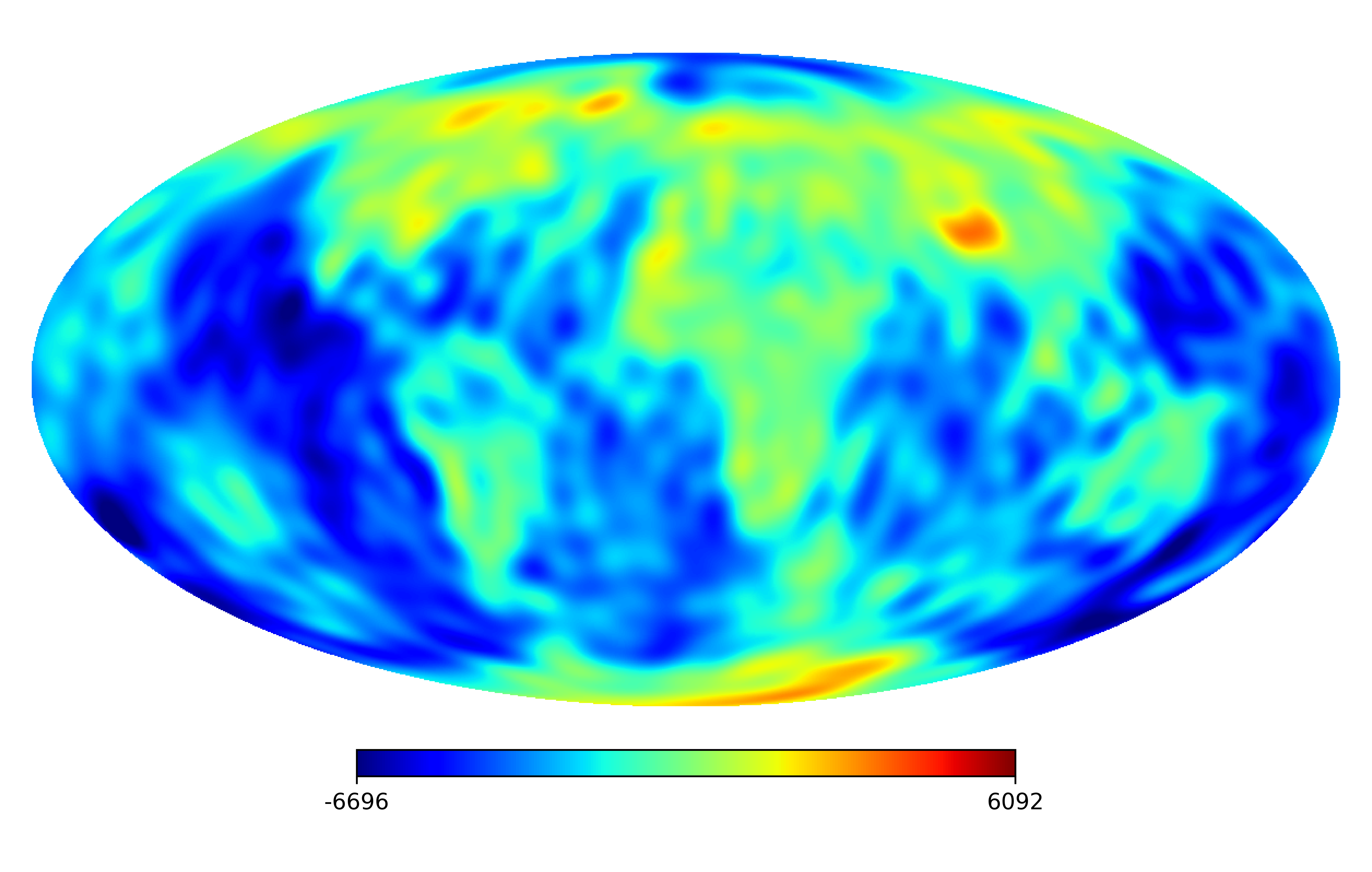}}
  \subfigure[{ group SPGl1}]{
  \includegraphics[width=3cm]{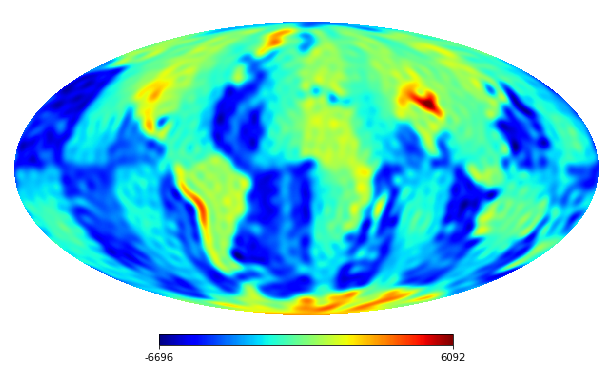}}\\
  \subfigure[Observed]{
  \includegraphics[width=3cm]{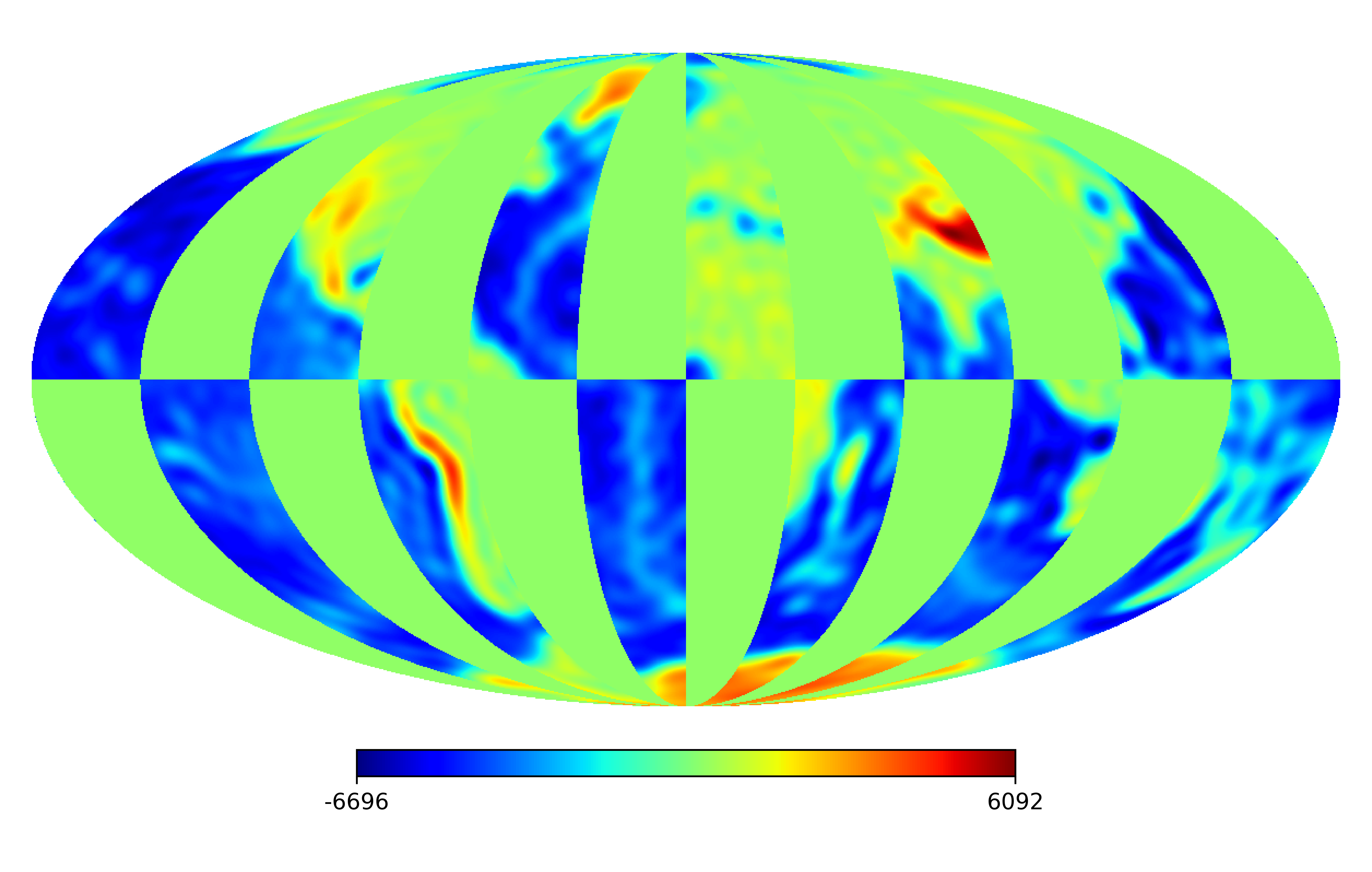}}
  \subfigure[Error of (b)]{
  \includegraphics[width=3cm]{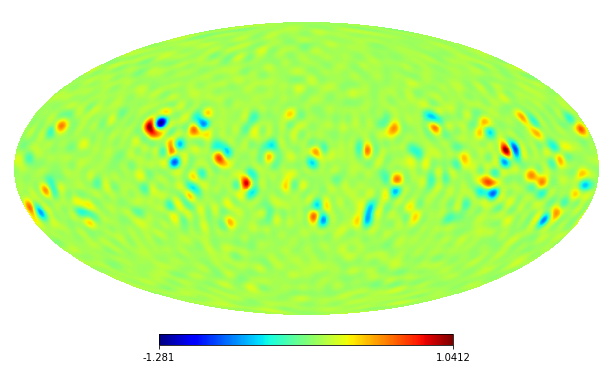}}
  \subfigure[Error of (c)]{
  \includegraphics[width=3cm]{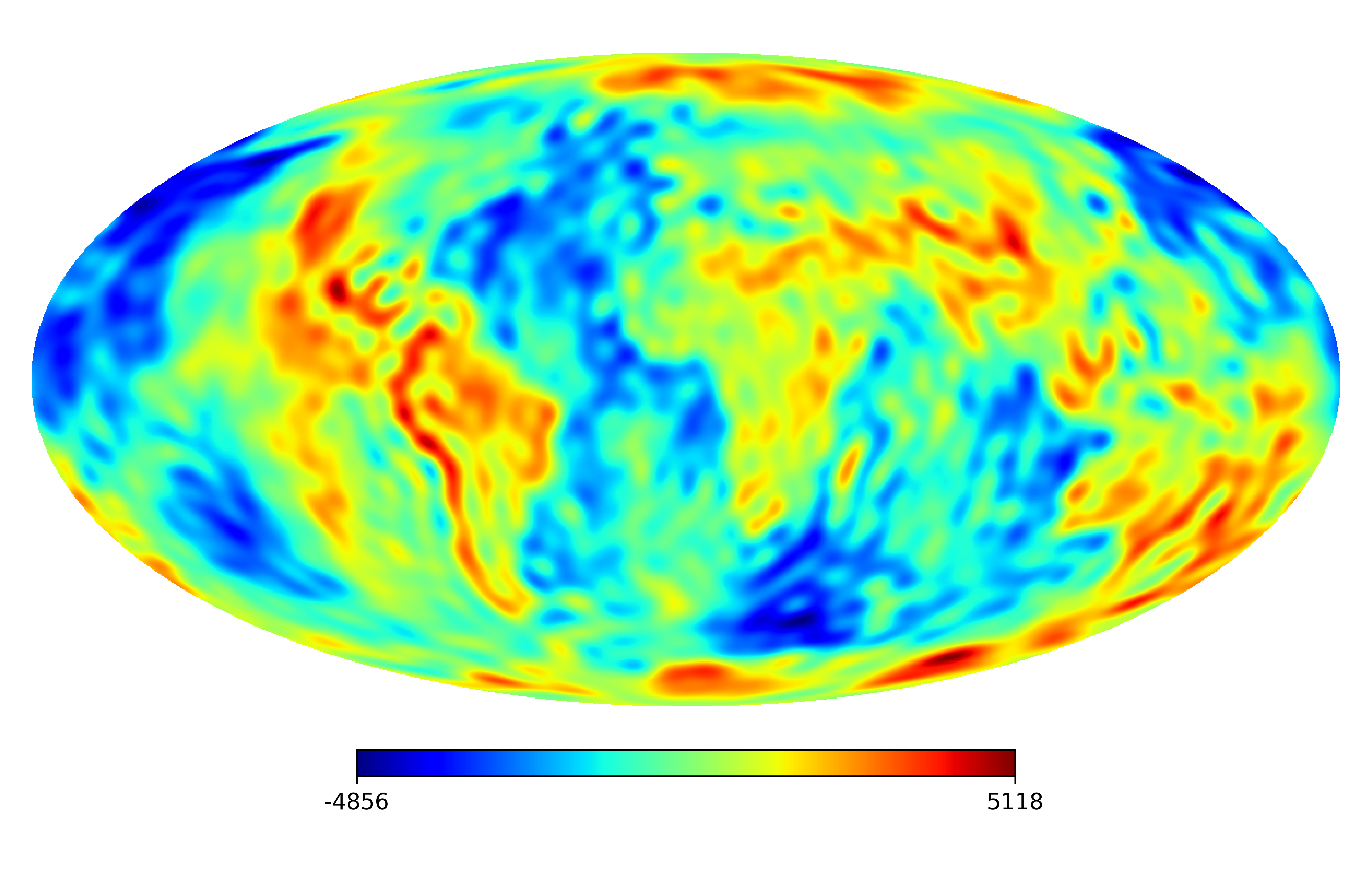}}
  \subfigure[{Error of (d)}]{
  \includegraphics[width=3cm]{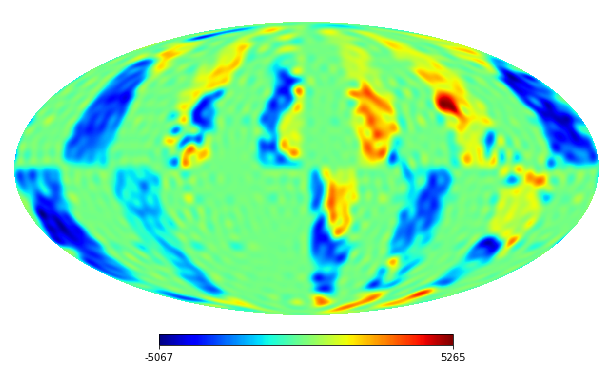}}
  \caption{Test images of Earth topographic data constructed to be band-limited at  $L=50$. The SNR of (b) is 90.47, the SNR of (c) is 6.44 {and the SNR of (d) is 8.07}.}\label{earth50}
\end{figure}

\begin{figure}[h]
  \centering
  \subfigure[True]{
  \includegraphics[width=3cm]{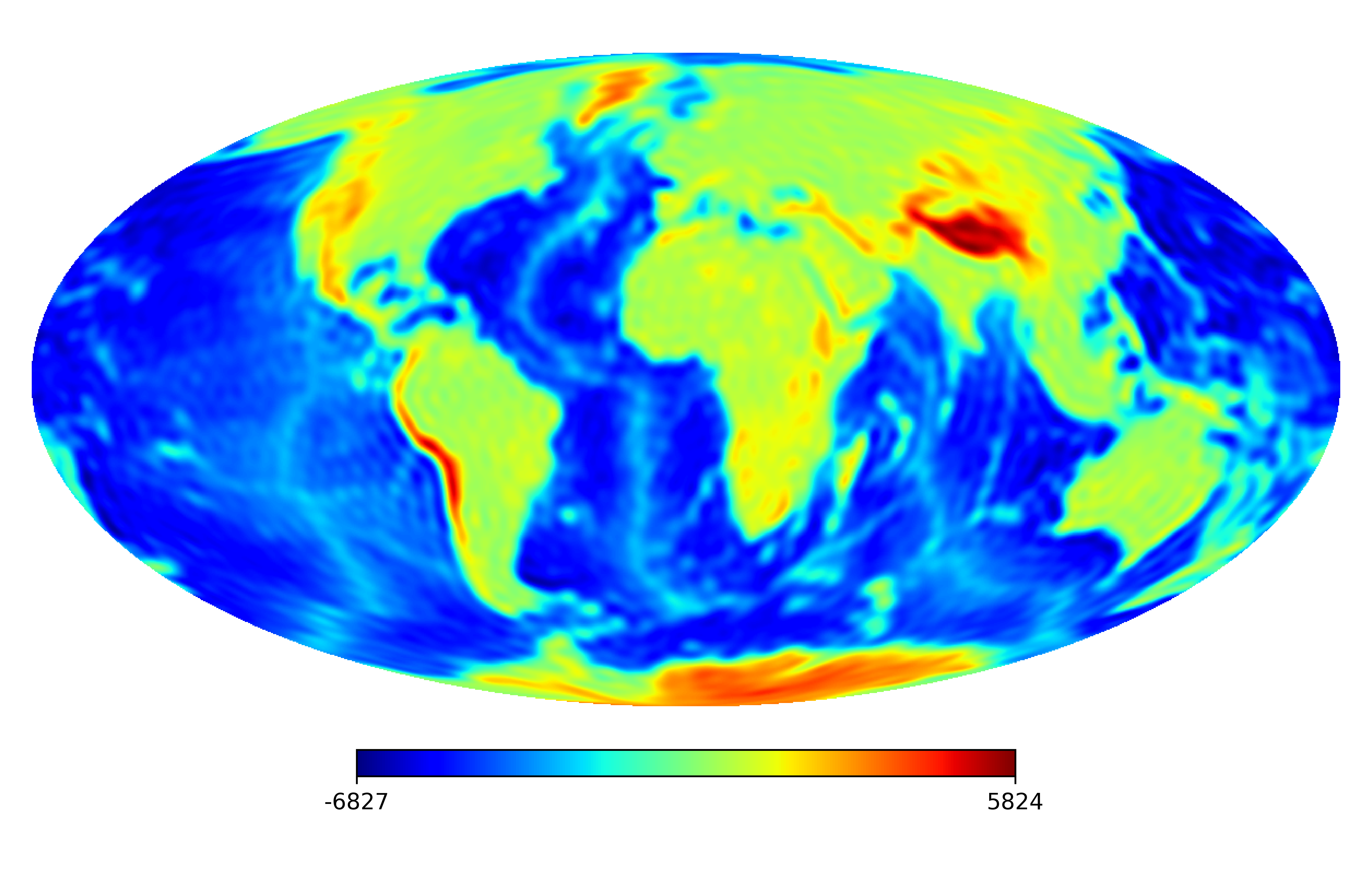}}
  \subfigure[ Algorithm \ref{alg:smoothnpg}]{
  \includegraphics[width=3cm]{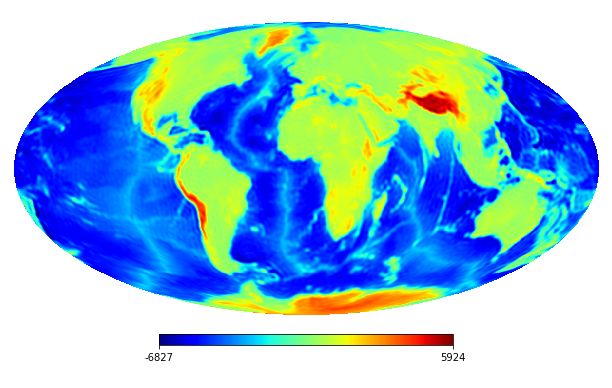}}
  \subfigure[YALL1]{
  \includegraphics[width=3cm]{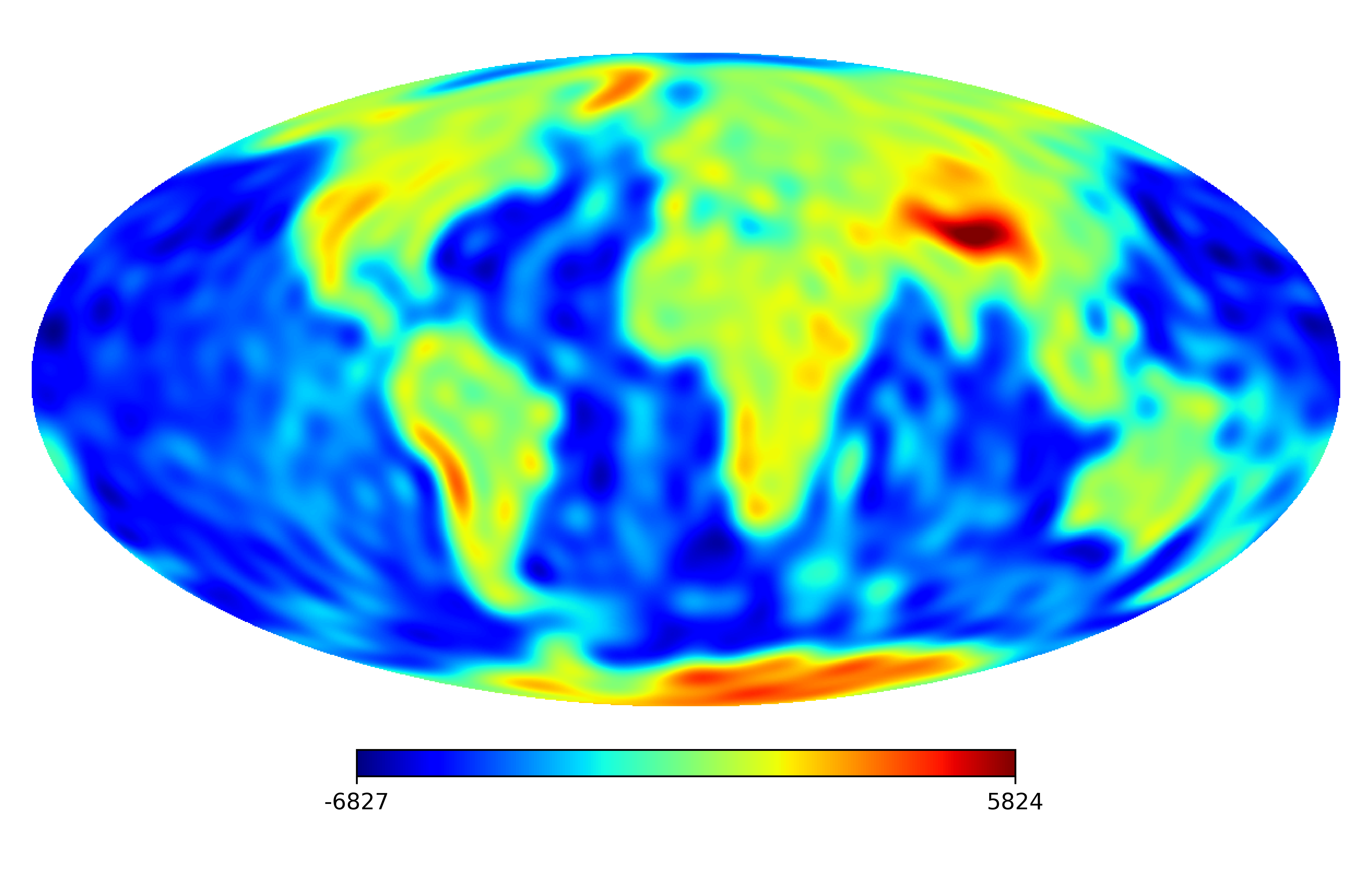}}
   \subfigure[{group SPGl1}]{
  \includegraphics[width=3cm]{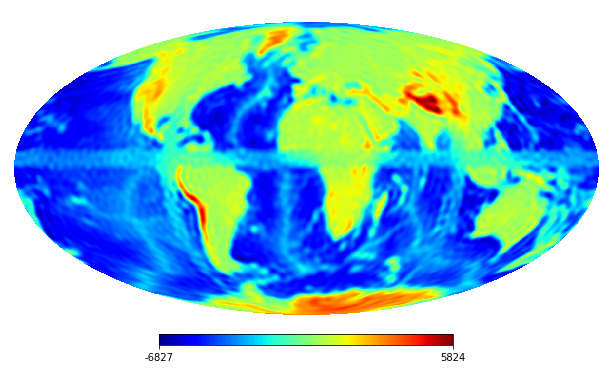}}\\
   \subfigure[Observed]{
  \includegraphics[width=3cm]{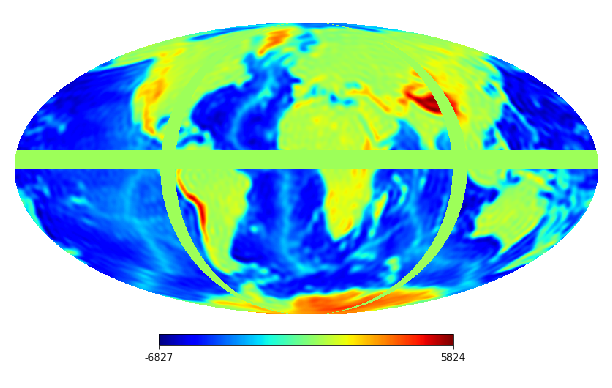}}
  \subfigure[Error of (b)]{
  \includegraphics[width=3cm]{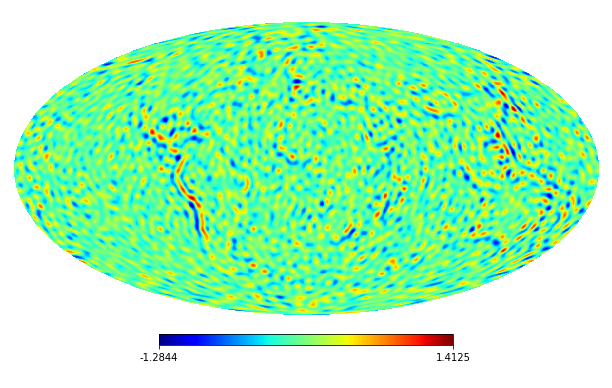}}
  \subfigure[Error of (c)]{
  \includegraphics[width=3cm]{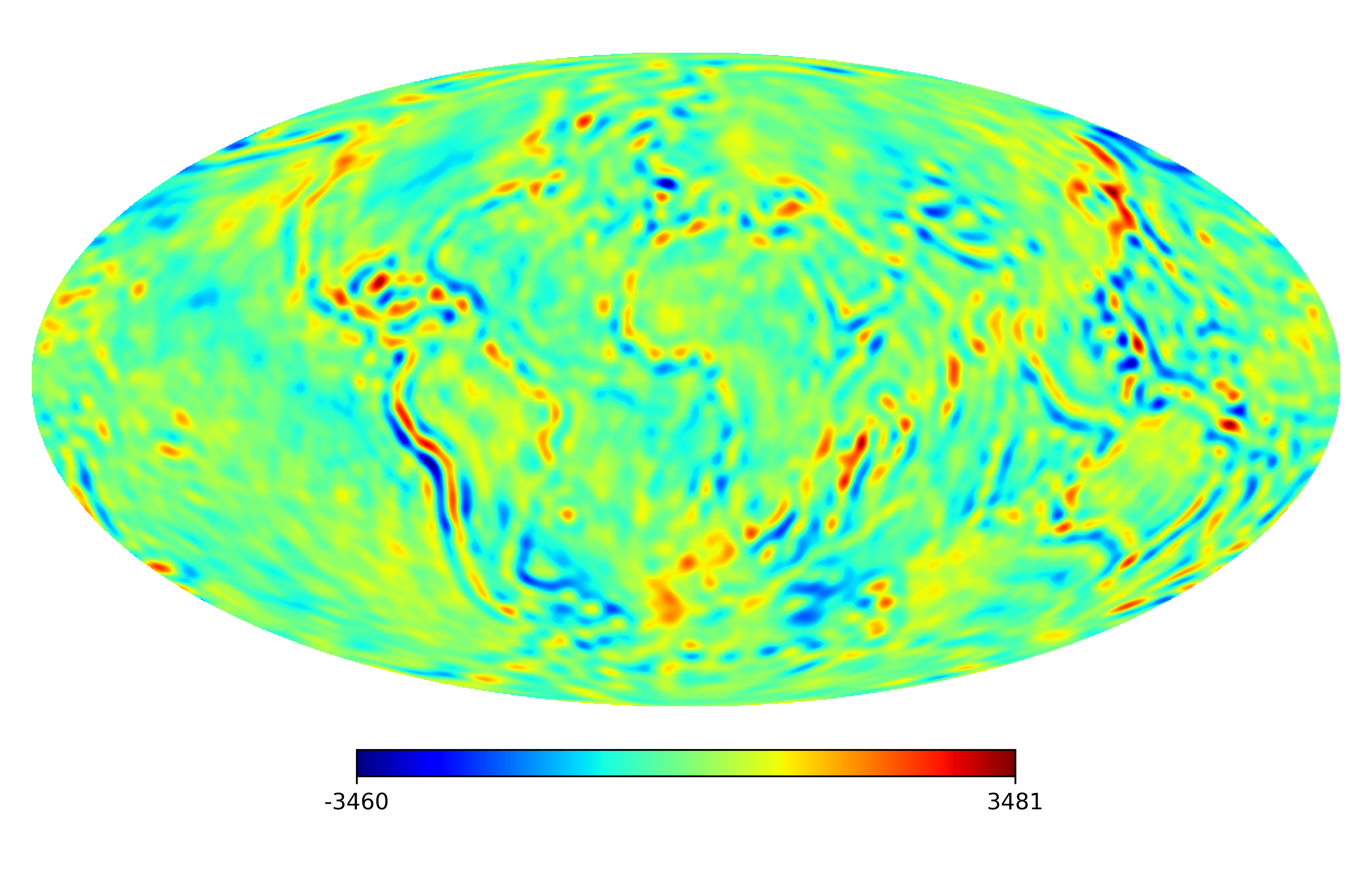}}
  \subfigure[{Error of (d)}]{
  \includegraphics[width=3cm]{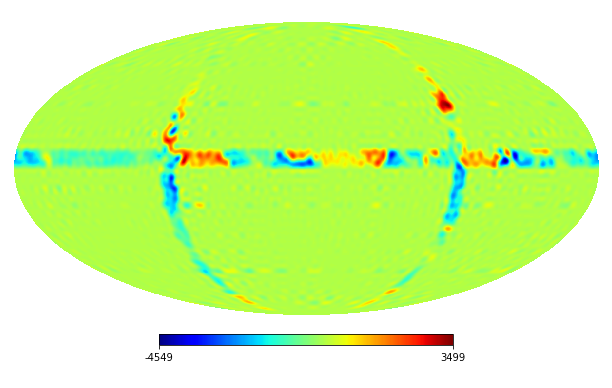}}
  \caption{ Test images of Earth topographic data constructed to be band-limited at $L=80$. The SNR of (b) is 82.48,  the SNR of (c) is 14.70 {and the SNR of (d) is 18.04}.}\label{earth80}
\end{figure}

\begin{figure}[h]
  \centering
  \subfigure[True]{
  \includegraphics[width=3cm]{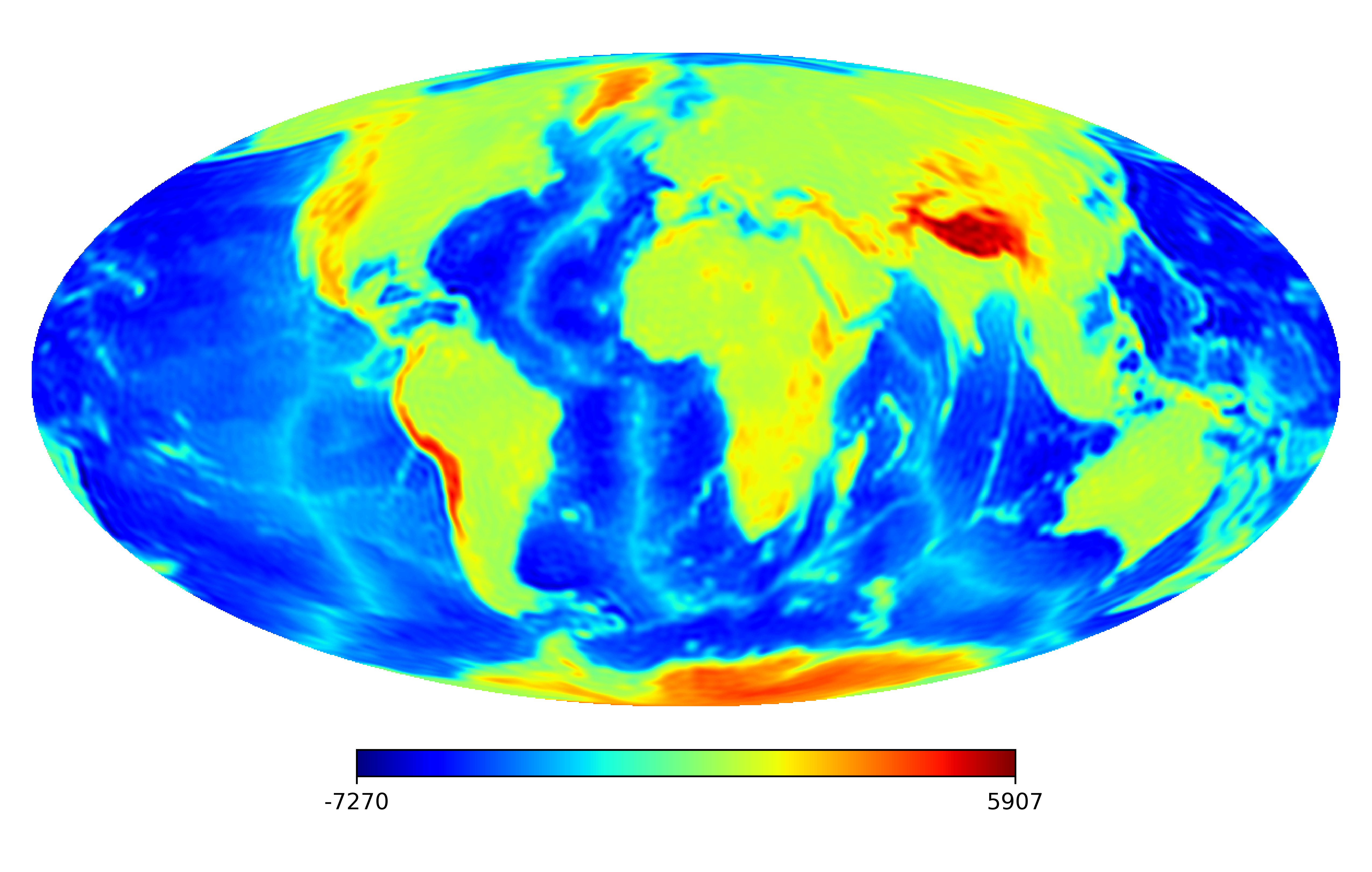}}
  \subfigure[Algorithm \ref{alg:smoothnpg}]{
  \includegraphics[width=3cm]{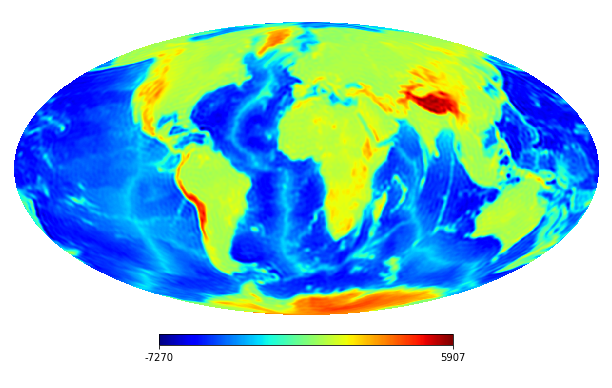}}
  \subfigure[ YALL1]{
  \includegraphics[width=3cm]{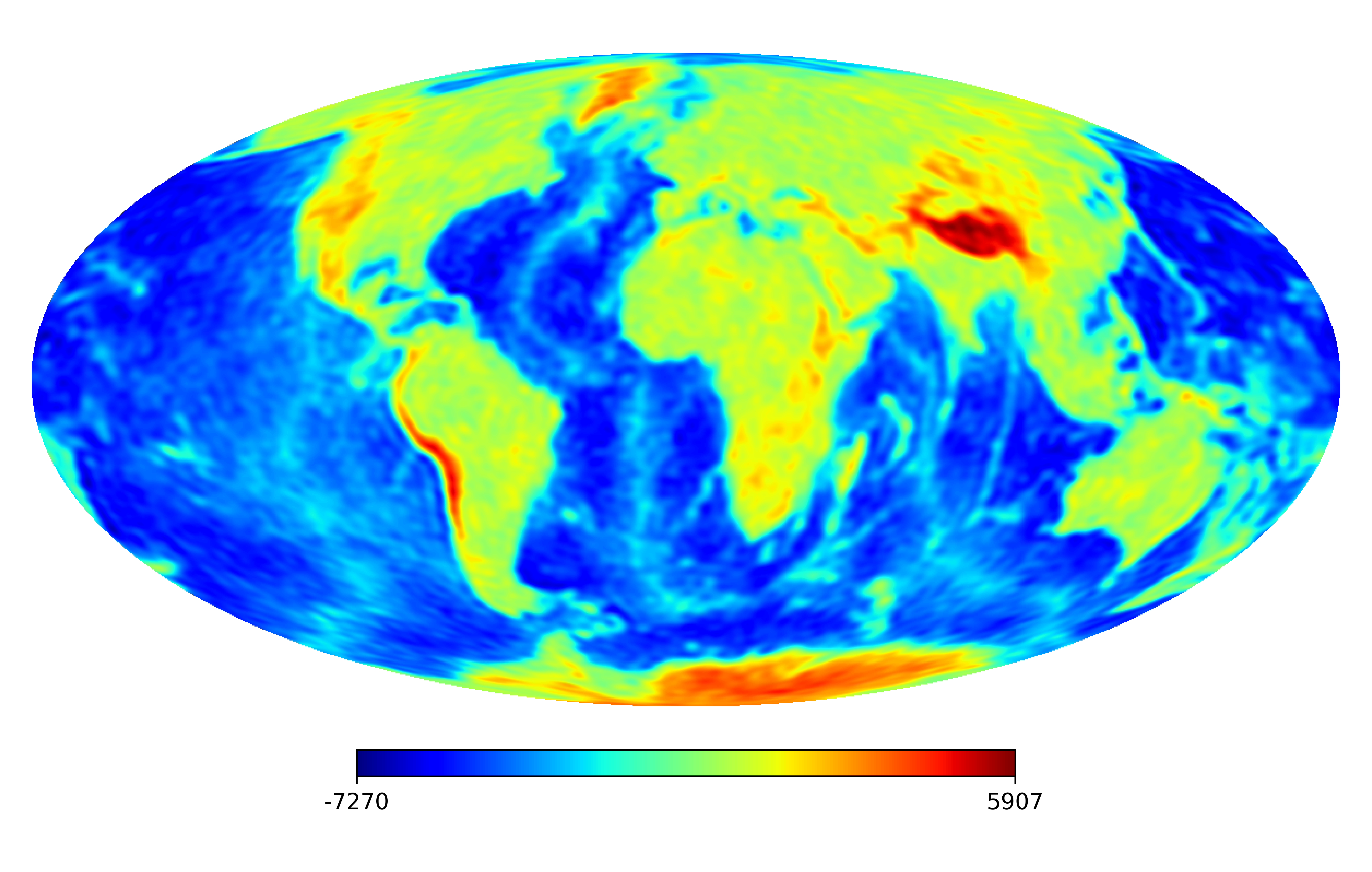}}
  \subfigure[{ group SPGl1}]{
  \includegraphics[width=3cm]{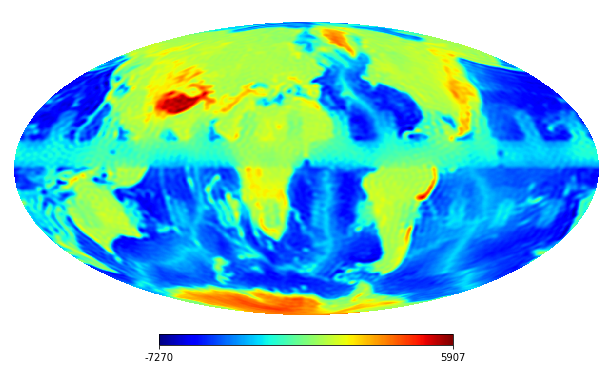}}\\
   \subfigure[Observed]{
  \includegraphics[width=3cm]{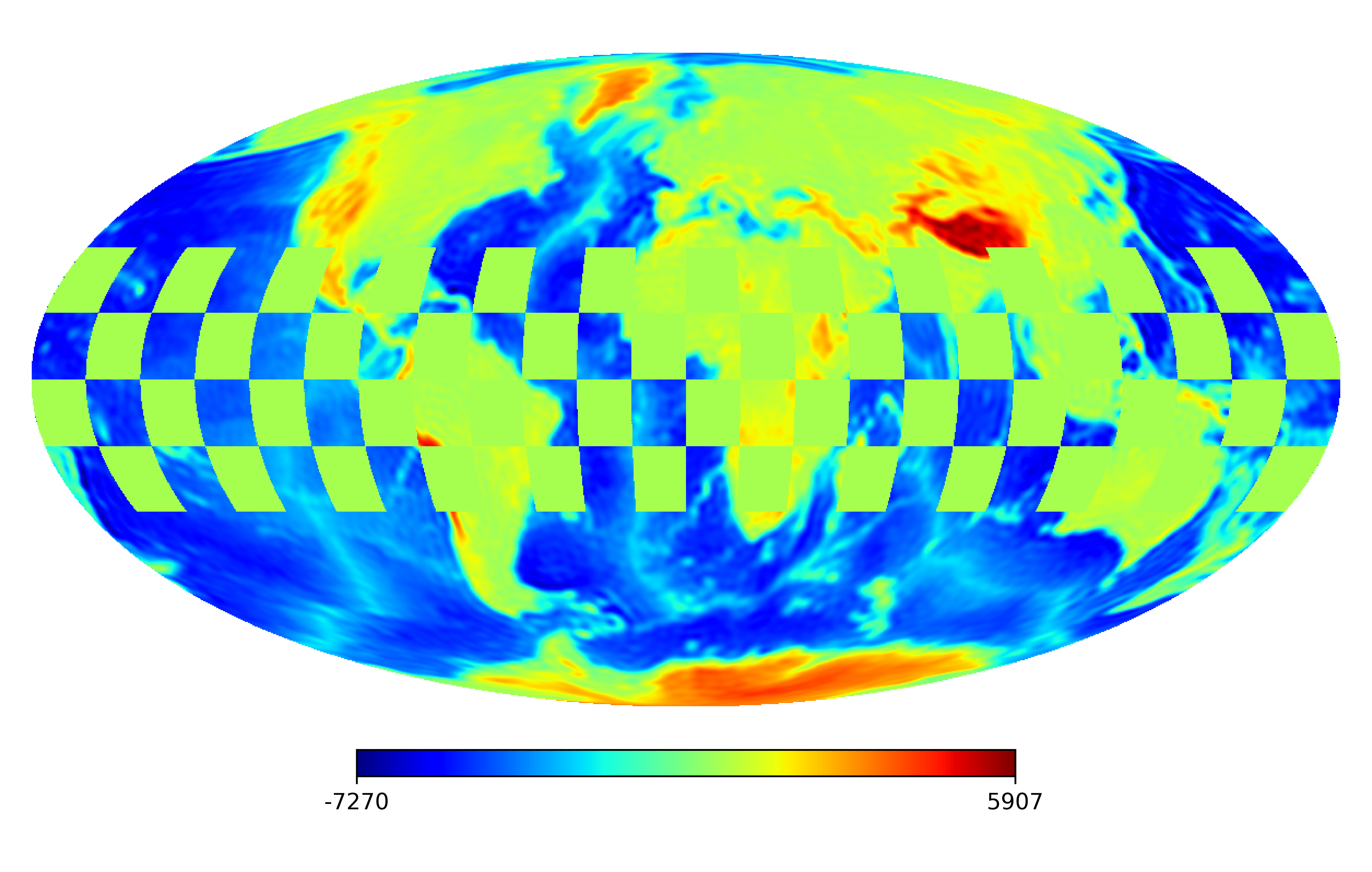}}
   \subfigure[Error of  (b)]{
  \includegraphics[width=3cm]{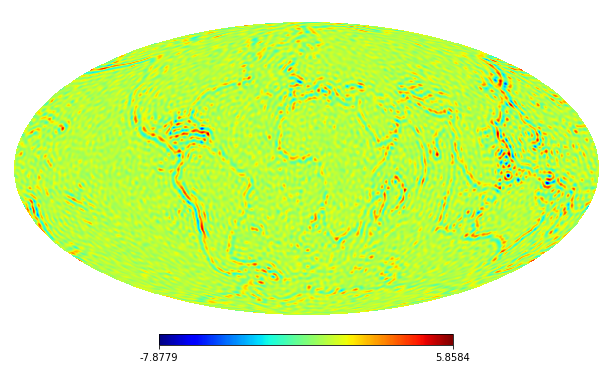}}
  \subfigure[Error of (c) ]{
  \includegraphics[width=3cm]{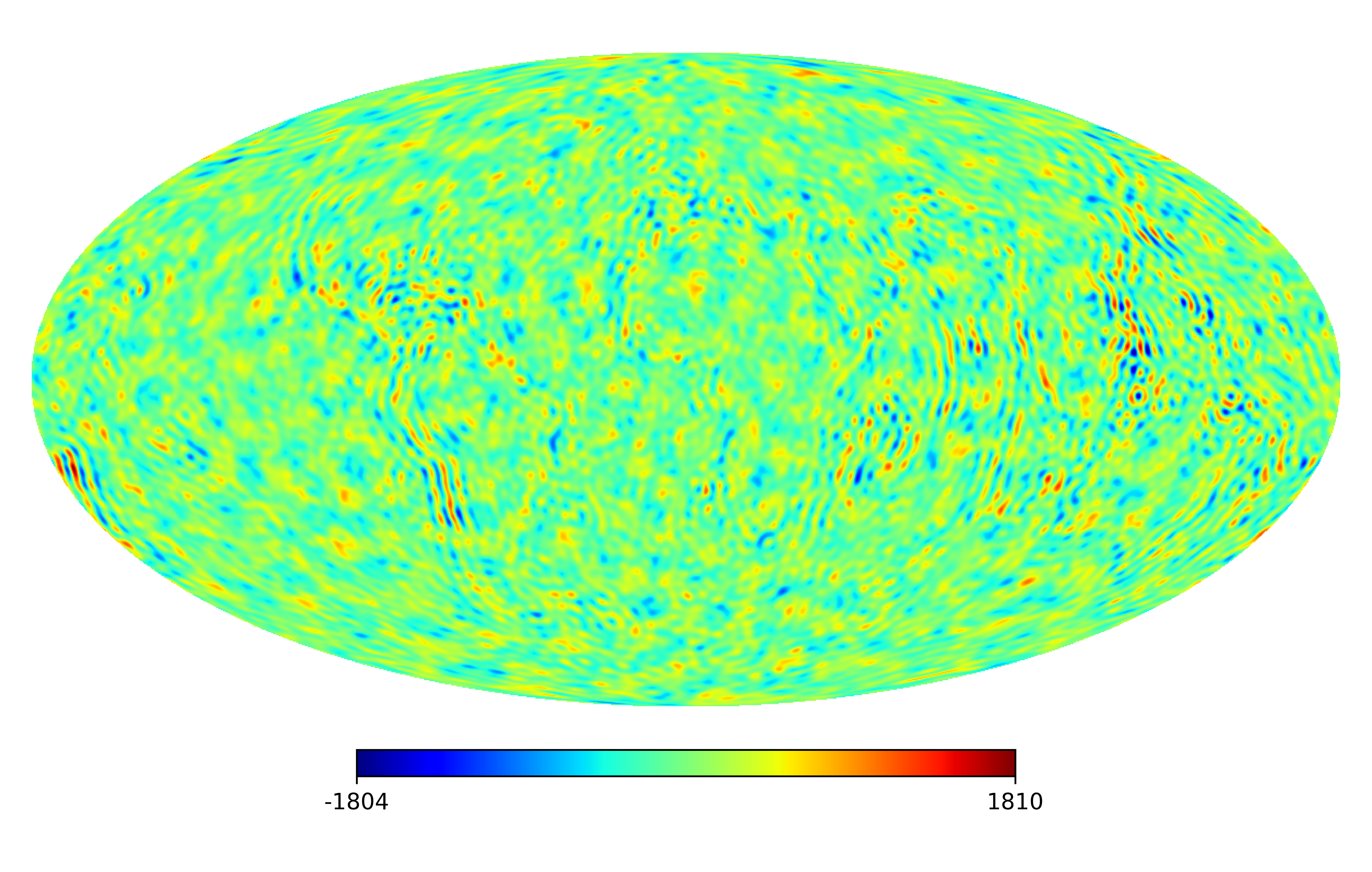}}
  \subfigure[{Error of (d)} ]{
  \includegraphics[width=3cm]{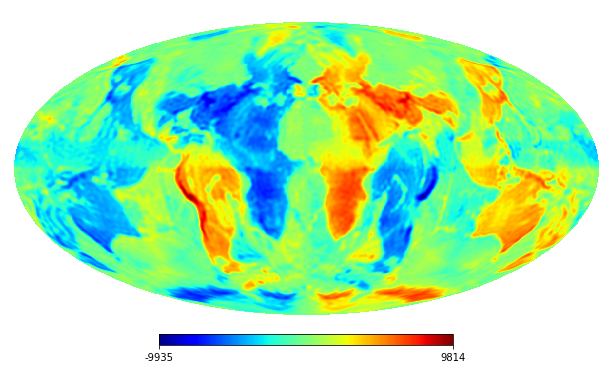}}
  \caption{ Test images of Earth topographic data constructed to be band-limited at $L=128$. The SNR of (b) is 73.18, the SNR of (c) is 21.60 {and  the SNR of (d) is 1.65}.}\label{earth128}
\end{figure}

\section{Conclusion}
In this paper, we propose a constrained group sparse optimization model (\ref{inin:model}) for the inpainting of random fields on the unit sphere with unique continuation property.
Based on the K-L expansion, we rewrite the constraint in (\ref{inin:model}) in a discrete form and derive an equivalence formulation (\ref{model2}) of problem  (\ref{inin:model}).
For  problem (\ref{model2}), we derive a lower bound (\ref{lowerbound:kkt2}) for the $\ell_{2}$ norm of nonzero groups of its scaled KKT points.
We show that problem (\ref{model2}) is equivalent to the finite-dimensional problem (\ref{fini:cmodel}).
For this finite-dimensional problem (\ref{fini:cmodel}) and its penalty problem (\ref{fini:penalty}), we prove the exact penalization in terms of local minimizers and $\varepsilon$-minimizers.
Moreover, we propose a smoothing penalty algorithm for solving problem (\ref{fini:cmodel}) and prove its convergence.
We also present the approximation error of the inpainted random field represented by the scaled KKT point in the space $L_{2}(\Omega\times\mathbb{S}^2)$.
Finally, we present numerical results on band-limited random fields on the sphere and the images from Earth topographic data to show the promising performance of our model by using the smoothing penalty algorithm.

\bibliographystyle{plain}
\bibliography{reference}


\appendix
\section* {Appendix A.\ Unique continuation property}
\label{appendix}
In this Appendix, we give details about the unique continuation property of a class of fields with spherical harmonic representations.

Let $\hat{r}\in(1,\eta)$ be a positive constant for some $\eta>1$, $B(0;\hat{r}):=\{(r,\theta,\phi): r\in[0, \hat{r}), \theta\in[0,\pi], \phi\in[0,2\pi]\}\subset\mathbb{R}^3$ be an open ball centered at $0$ of radius $\hat{r}$.
 Let
   \begin{equation}\label{T:FI}
   \mathcal{T}=\left\{ T(\theta,\phi)\in L_{2}(\mathbb{S}^2):  T(\theta,\phi)
 =\sum_{l=0}^\infty \sum_{m=-l}^l\alpha_{l,m}Y_{l,m}(\theta,\phi)
 \ {\rm{and}}\ \alpha\in\ell_{\beta}^p \right\}\footnote{With a slight abuse of notation, in this subsection, $T(\theta,\phi)\equiv T(\mathrm{x})$ and $Y_{l,m}(\theta,\phi)\equiv Y_{l,m}(\mathrm{x})$, where $\mathrm{x}=(\sin\theta\cos\phi,\sin\theta\sin\phi,\cos\theta)^T\in\mathbb{S}^2$.}
 \end{equation}
 be a class of  fields with spherical harmonic representations and coefficients belong to  $\ell_{\beta}^p$.
Then we show that  any  $T\in\mathcal{T}$ has the unique continuation property.
Let $H: B(0;\hat{r})\rightarrow\mathbb{R}$ be given by
$$H(r,\theta,\phi)=\sum\limits_{l=0}^\infty r^l \sum_{m=-l}^l\alpha_{l,m}Y_{l,m}(\theta,\phi),$$
where $\alpha\in\ell_{\beta}^p$.

\begin{lemma}\label{ap:seq}If $\alpha\in\ell_{\beta}^p$, then $\sum_{l=0}^\infty \eta^ll\|\alpha_{l\cdot}\|<\infty$.
\end{lemma}
\begin{proof}
Since $\alpha\in\ell_{\beta}^p$ and $\sum_{l=0}^\infty \eta^l =\infty$,  there exists an
integer $N$  such that for any $l>N$,
$\eta^{l}l^p\|\alpha_{l\cdot}\|^p< \eta^l.$
Then, we obtain, for  any $l>N$,  $l\|\alpha_{l\cdot}\|\leq l^p\|\alpha_{l\cdot}\|^p<1$
which implies that $\eta^ll\|\alpha_{l\cdot}\|\leq \eta^ll^p\|\alpha_{l\cdot}\|^p$.
Thus, $\sum_{l=0}^\infty \eta^ll\|\alpha_{l\cdot}\|<\infty.$
This completes the proof.
\end{proof}

\begin{lemma}\label{T:ana}
$H$ is real analytic in $B(0;\hat{r})$.
\end{lemma}
\begin{proof} Let
 $H_{L}(r,\theta,\phi)=\sum_{l=0}^LH_{l}(r,\theta,\phi)$,
 where $H_{l}(r,\theta,\phi)=r^l\sum_{m=-l}^l\alpha_{l,m}Y_{l,m}(\theta,\phi)$.
 By definition of harmonic functions \cite{r1}, $r^lY_{l,m}(\theta,\phi)$, $l\in\mathbb{N}_0$, $m=-l,\ldots,l$ are harmonic functions on $B(0;\hat{r})$.
 Then we obtain that $H_L$, $L\in\mathbb{N}_0$ are harmonics functions on $B(0;\hat{r})$.
Moreover, for any $(r,\theta,\phi)\in B(0;\hat{r})$ and $l>0$,
\begin{eqnarray*}
  \vert H_{l}(r,\theta,\phi)\vert&=& \left\vert\sum\limits_{m=-l}^lr^l\alpha_{l,m}Y_{l,m}(\theta,\phi)\right\vert
  \leq\left\vert\left(\sum\limits_{m=-l}^l(r^l\alpha_{l,m})^2\right)^{\frac{1}{2}}\left(\sum\limits_{m=-l}^l(Y_{l,m}(\theta,\phi))^2\right)^{\frac{1}{2}}\right\vert\\
  &=&r^l\|\alpha_{l\cdot}\|\sqrt{\tfrac{2l+1}{4\pi} }
  < \eta^ll\|\alpha_{l\cdot}\|,
\end{eqnarray*}
where  the first inequality follows from Cauchy-Schwarz inequality, the second equality follows from addition theorem of spherical harmonics and the last inequality follows from $r\in[0,\hat{r})$ and $\hat{r}<\eta$.

By Lemma \ref{ap:seq}, $\sum_{l=0}^\infty \eta^ll\|\alpha_{l\cdot}\|<\infty$, then for any $\epsilon>0$, there exists $N$ such that for any $L'>L>N$, we have that for any $(r,\theta,\phi)\in B(0;\hat{r})$,
\begin{eqnarray*}
  \vert H_{L'}(r,\theta,\phi)-H_L(r,\theta,\phi)\vert&=&\left\vert\sum_{l=L+1}^{L'}H_{l}(r,\theta,\phi)\right\vert
  \leq \sum_{l=L+1}^{L'}\left\vert H_{l}(r,\theta,\phi)\right\vert
  < \sum_{l=L+1}^{L'}\eta^ll\|\alpha_{l\cdot}\|<\epsilon.
\end{eqnarray*}
Thus, the sequence of harmonic functions $\{H_{L}\}$ converges uniformly to $H$ on $B(0;\hat{r})$.

By  \cite[Theorem 1.23]{r1},  $H$ is harmonic on $B(0;\hat{r})$.
Moreover, by \cite[Theorem 1.28]{r1}, $H$ is real analytic in $B(0;\hat{r})$. This completes the proof.
\end{proof}

\begin{lemma}\label{le:t0}
 For any $T\in\mathcal{T}$, if  there exists a sequence of distinct points $\{(\theta_n,\phi_n)\}\subseteq \mathbb{S}^2$, with at least one limit point in $\mathbb{S}^2$, and if $T(\theta_n,\phi_n)=0$, $n=1,2,\ldots$, then $T\equiv0$ on $\mathbb{S}^2$.
\end{lemma}
\begin{proof}
It is obvious that
  $H(1,\theta,\phi)=T(\theta,\phi)$, then
  $T(\theta_n,\phi_n)=0$, $n=1,2,\ldots$ implies that $H(1,\theta_n,\phi_n)=0$,  $n=1,2,\ldots$
  By Lemma \ref{T:ana}, $H$ is real analytic in $B(0;\hat{r})$.
  Then by \cite[Theorem 8.1.3]{hill}, we obtain $H\equiv0$ in $B(0;\hat{r})$
  which implies that $T\equiv0$ on $\mathbb{S}^2$. This completes the proof.
\end{proof}

Now we present the unique continuation property of  any $T\in\mathcal{T}$.
\begin{theorem}\label{uniq}
For any $T\in\mathcal{T}$, if $T=0$ on $ \Gamma$, then $T\equiv0$ on $\mathbb{S}^2$.
 \end{theorem}
 \begin{proof} Since the coefficients $\alpha\in\ell_{\beta}^{p}$ and $ \Gamma$ has an open subset, we know from Lemma \ref{le:t0} that if $T=0$ on $\Gamma$ then $T\equiv0$ on $\mathbb{S}^2$.
\end{proof}

By Parseval's theorem, for any $T\in L_{2}(\mathbb{S}^2)$, we have $\|T\|_{L_{2}(\mathbb{S}^2)}^2=\|\alpha\|^2$, which implies that $T\equiv0$ if and only if $\alpha=0$.
Hence, we can claim that for any $T\in\mathcal{T}$, $\mathcal{A}(T)\equiv0$ if and only if $\alpha=0$.

{\section* {Appendix B.\ Wirtinger gradient}
In this appendix, we briefly introduce the Wirtinger gradient of real-valued functions with complex variables.
For more details about Wirtinger calculus, we refer  to  \cite{BR,Ken,chen20}.

Let $f:\mathbb{C}^{n}\rightarrow\mathbb{R}$ be a real-valued function of a vector of complex variables $z\in\mathbb{C}^{n}$, which can be written as {{$f_1(z,{{\bar{z}}})$, that is $f(z)=f_1(z,{{\bar{z}}})$}}. Let $z=x+iy$, where $x, y\in\mathbb{R}^n$.
We say  $f(z)=f_{{2}}(x+iy)$ is  Wirtinger differentiable, if $f_{{{2}}}$ is differentiable with respect to $x$ and $y$, respectively.
Following \cite{BR,Ken}, under the conjugate coordinates
 $(z^T,\bar{z}^T)^T\in\mathbb{C}^{n}\times\mathbb{C}^{n}$, the Wirtinger gradient of $f$ at $z \in \mathbb{C}^{n}$ is
$$\nabla f(z)=\left(
\begin{array}{c}
\partial_{z}f(z) \vspace{1mm}\\
\partial_{\bar{z}}f(z) \\
\end{array}
\right),
$$
where $
\partial_{z}f(z):=\frac{\partial f_{{1}}(z,\bar{z})}{\partial z}\big|_{\bar{z}={\rm{const}}}$, $
\partial_{\bar{z}}f(z):= \frac{\partial f_{{1}}(z,\bar{z})}{\partial \bar{z}}\big|_{z={\rm{const}}}.
$
Since $f$ is real-valued, $\overline{\partial_{z}f}=\partial_{\bar{z}}f$. Thus, $$\nabla f(z)=0 \textrm{ } \quad  \Leftrightarrow\textrm{ } \quad    \partial_{\bar{z}}f(z)=0.$$
We say $z^{*}\in\mathbb{C}^{n}$ is a stationary point of $f$ if it satisfies
$$
\partial_{\bar{z}}f(z^*)=0.
$$
It is easy to verify that the
Wirtinger  gradient of $f(z)=\|z\|^{2}$,  $z\in\mathbb{C}^{n}$, is
$\nabla f(z)= (\bar z^T, z^T)^T,$
and $z^{*}\in\mathbb{C}^{n}$ is a stationary point of $f$ if $z^{*}=0$.
}
\end{document}